\documentclass[12pt,psamsfonts,leqno,oneside,letterpaper]{amsart}
\usepackage[dvips,text={6.5truein,9truein},left=1truein,top=1truein]{geometry}
\usepackage{amssymb,amsmath,amscd,enumerate}
\usepackage[pdftex]{graphicx}
\usepackage{url}

\usepackage[colorlinks,linkcolor=blue,citecolor=blue,pdfstartview=FitH]{hyperref}
\input xy
\xyoption{all}
\SelectTips{cm}{12}

\usepackage{color}

\theoremstyle{definition}
\newtheorem{para}{}[section]

\newtheorem{remark}[para]{Remark}
\newtheorem{remarks}[para]{Remarks}
\newtheorem{notation}[para]{Notation}
\newtheorem{convention}[para]{Convention}
\newtheorem{definition}[para]{Definition}
\newtheorem{definitions}[para]{Definitions}
\newtheorem{claim}[equation]{Claim}
\newtheorem*{fact}{Fact}

\newtheorem{example}[para]{Example}

\newcommand\Alternatives{\begin{enumerate}[(i)]}
\newcommand\EndAlternatives{\end{enumerate}}
\newcommand\Conditions{\begin{enumerate}[(1)]}
\newcommand\EndConditions{\end{enumerate}}

\theoremstyle{plain}
\newtheorem{theorem}[para]{Theorem}
\newtheorem{lemma}[para]{Lemma}
\newtheorem{proposition}[para]{Proposition}
\newtheorem{corollary}[para]{Corollary}
\newtheorem{conjecture}[para]{Conjecture}

\numberwithin{equation}{para}
\numberwithin{figure}{section}

\newcommand\Number{\begin{para}}
\newcommand\EndNumber{\end{para}}
\newcommand\Definition{\begin{definition}}
\newcommand\EndDefinition{\end{definition}}
\newcommand\Definitions{\begin{definitions}}
\newcommand\EndDefinitions{\end{definitions}}
\newcommand\Theorem{\begin{theorem}}
\newcommand\EndTheorem{\end{theorem}}
\newcommand\Conjecture{\begin{conjecture}}
\newcommand\EndConjecture{\end{conjecture}}
\newcommand\Remark{\begin{remark}}
\newcommand\EndRemark{\end{remark}}
\newcommand\Remarks{\begin{remarks}}
\newcommand\EndRemarks{\end{remarks}}
\newcommand\Convention{\begin{convention}}
\newcommand\EndConvention{\end{convention}}
\newcommand\Notation{\begin{notation}}
\newcommand\EndNotation{\end{notation}}
\newcommand\Lemma{\begin{lemma}}
\newcommand\EndLemma{\end{lemma}}
\newcommand\Proposition{\begin{proposition}}
\newcommand\EndProposition{\end{proposition}}
\newcommand\Corollary{\begin{corollary}}
\newcommand\EndCorollary{\end{corollary}}
\newcommand\Claim{\begin{claim}}
\newcommand\EndClaim{\end{claim}}
\newcommand\Proof{\begin{proof}}
\newcommand\EndProof{\end{proof}}
\newcommand\Equation{\begin{equation}}
\newcommand\EndEquation{\end{equation}}

\newcommand\Bullets{\begin{itemize}}
\newcommand\EndBullets{\end{itemize}}

\newcommand\CenteredPoly{\cite[Definition 1.1]{DeB_cyclic_geom}}
\newcommand\CyclicallyOrderedV{\cite[Definition 1.3]{DeB_cyclic_geom}}
\newcommand\PtsToPoly{\cite[Lemma 1.4]{DeB_cyclic_geom}}
\newcommand\OneSide{\cite[Lemma 1.5]{DeB_cyclic_geom}}
\newcommand\IsoscelesDecomp{\cite[Lemma 1.6]{DeB_cyclic_geom}}
\newcommand\ParameterFunction{\cite[Lemma 1.7]{DeB_cyclic_geom}}
\newcommand\MonotonicParameterFunction{\cite[Lemma 1.8]{DeB_cyclic_geom}}
\newcommand\LongestSide{\cite[Corollary 1.11]{DeB_cyclic_geom}}
\newcommand\Ad{\cite[Lemma 2.3]{DeB_cyclic_geom}}
\newcommand\SufficientParameters{\cite[Lemma 2.4]{DeB_cyclic_geom}}
\newcommand\CyclicallyOrderedP{\cite[Definition 2.5]{DeB_cyclic_geom}}
\newcommand\UniqueParameters{\cite[Proposition 2.7]{DeB_cyclic_geom}}
\newcommand\CenteredSpace{\cite[Definition 3.1]{DeB_cyclic_geom}}
\newcommand\FundamentalDomain{\cite[Lemma 3.2]{DeB_cyclic_geom}}
\newcommand\CenteredClosure{\cite[Lemma 3.3]{DeB_cyclic_geom}}
\newcommand\BCn{\cite[Lemma 3.4]{DeB_cyclic_geom}}
\newcommand\RadiusFunction{\cite[Lemma 3.6]{DeB_cyclic_geom}}
\newcommand\InCenteredClosure{\cite[Lemma 3.9]{DeB_cyclic_geom}}
\newcommand\HCn{\cite[Lemma 3.10]{DeB_cyclic_geom}}
\newcommand\VerticalRay{\cite[Corollary 3.11]{DeB_cyclic_geom}}
\newcommand\SmoothAnglesRadius{\cite[Proposition 4.1]{DeB_cyclic_geom}}
\newcommand\RadiusDeriv{\cite[Lemma 4.5]{DeB_cyclic_geom}}
\newcommand\RadiusUpNDown{\cite[Lemma 4.7]{DeB_cyclic_geom}}
\newcommand\PolyConvergence{\cite[Proposition 4.10]{DeB_cyclic_geom}}
\newcommand\Defect{\cite[Definition 5.1]{DeB_cyclic_geom}}
\newcommand\FullSector{\cite[Lemma 5.2]{DeB_cyclic_geom}}
\newcommand\PackingVsDecomp{\cite[Lemma 5.3]{DeB_cyclic_geom}}
\newcommand\DefectFunction{\cite[Lemma 5.4]{DeB_cyclic_geom}}
\newcommand\DefectDerivative{\cite[Proposition 5.5]{DeB_cyclic_geom}}
\newcommand\Order{\cite[Definition 5.7]{DeB_cyclic_geom}}
\newcommand\Monotonicity{\cite[Corollary 5.8]{DeB_cyclic_geom}}

\newcommand\HorocyclicBound{\cite[Corollary 5.11]{DeB_cyclic_geom}}
\newcommand\TriDefectRadius{\cite[Lemma 6.1]{DeB_cyclic_geom}}
\newcommand\TriCenter{\cite[Lemma 6.2]{DeB_cyclic_geom}}
\newcommand\CenteredDefectBound{\cite[Lemma 6.6]{DeB_cyclic_geom}}

\newcommand\HalfSymmQuadDefect{\cite[Lemma 6.8]{DeB_cyclic_geom}}
\newcommand\CenteredBoundaryBound{\cite[Lemma 6.9]{DeB_cyclic_geom}}

\newcommand\calp{{\mathcal P}}

\newcommand\cald{{\mathcal D}}
\newcommand\calAC{\mathcal{AC}}
\newcommand\calBC{\mathcal{BC}}
\newcommand\calHC{\mathcal{HC}}
\newcommand\calc{{\mathcal C}}
\newcommand\cale{\mathcal{E}}
\newcommand\calf{\mathcal{F}}

\newcommand\calh{{\mathcal H}}
\newcommand\calm{{\mathcal M}}

\newcommand\cals{\mathcal{S}}
\newcommand\calt{{\mathcal T}}

\newcommand\calv{\mathcal{V}}
\newcommand\co{\colon\thinspace}
\newcommand\bb{\mathbf{b}}
\newcommand\bd{\mathbf{d}}

\begin{document}

\title{Tessellations of hyperbolic surfaces}

\author{Jason DeBlois}
\address{Department of Mathematics\\Stanford University}
\email{jdeblois@math.stanford.edu}
\thanks{Partially supported by NSF grant DMS-1007175}

\begin{abstract}  A finite subset $S$ of a closed hyperbolic surface $F$ canonically determines a \textit{centered dual decomposition} of $F$: a cell structure with vertex set $S$, geodesic edges, and $2$-cells that are unions of the corresponding Delaunay polygons.  Unlike a Delaunay polygon, a centered dual $2$-cell $Q$ is not determined by its collection of edge lengths; but together with its combinatorics, these determine an \textit{admissible space} parametrizing geometric possibilities for the Delaunay cells comprising $Q$.  We illustrate its application by using the centered dual decomposition to extract combinatorial information about the Delaunay tessellation among certain genus-$2$ surfaces, and with this relate injectivity radius to covering radius here.\end{abstract}

\maketitle


A finite subset $\cals$ of a closed hyperbolic surface $F$ canonically determines a \textit{Voronoi tessellation} $V$ and \textit{Delaunay tessellation} $P$, polygonal decompositions of $F$ that are dual in a certain sense.  Let us briefly outline this construction.  Fix a locally isometric universal covering $\pi\co \mathbb{H}^2\to F$ and let $\widetilde{\cals} = \pi^{-1}(\cals)$.  The \textit{Voronoi tessellation} $\widetilde{V}$ of $\mathbb{H}^2$ determined by $\widetilde{\cals}$ is a cell complex structure where each $x\in \cals$ determines a polygonal $2$-cell $V_x$ defined by:\begin{align}\label{Voronoi cell} 
  V_x = \{p\in\mathbb{H}^2\,|\, d(p,x) \leq d(p,y)\ \mbox{for each}\ y \in \cals - \{x\}\} \end{align}
Then $\widetilde{V}^{(1)} = \bigcup\ \{ V_x\cap V_y\,|\,x\in\widetilde{\cals},\ y\in\widetilde{\cals}-\{x\}\}$, and each point of $\widetilde{V}^{(0)}$ is equidistant from at least $3$ points of $\widetilde{\cals}$ (see Section \ref{Voronoi dfn}).  The \textit{geometric dual} to an edge $V_x\cap V_y$ of $\widetilde{V}$ is the geodesic arc $\gamma_{xy}$ in $\mathbb{H}^2$ joining $x$ to $y$, and the set of geometric duals to edges of $\widetilde{V}$ is the edge set of the Delaunay tessellation $\widetilde{P}$ determined by $\widetilde{\cals}$.   The covering action of $\pi_1 F$ on $\mathbb{H}^2$ leaves $\widetilde{V}$ and $\widetilde{P}$ invariant, and these descend to the tessellations $V$ and $P$ of $F$.  

$P$ and $V$ are dual in the sense that their edge sets are canonically bijective, as is the vertex set of each with the face set of the other.  However, in some cases an edge $e = V_x \cap V_y$ of $\widetilde{V}$ is not \textit{centered} (see Definition \ref{centered edge}): $\mathit{int}\,e$ does not intersect the geometric dual $\gamma_{xy}\subset\widetilde{P}$ to $e$.  We will regard this as a pathology of $P$, and ``fix'' it with the centered dual decomposition.  Before we outline this construction, here is a sample application of our methods:

\newcommand\MainThm{Let $r_{\beta} = d_{\beta}/2 > 0$, where $\cosh d_{\beta}$ is the real root of $x^3-14x^2-15x-4$.  The Delaunay tessellation of a closed, orientable hyperbolic surface $F$ of genus $2$ determined by $\{x\}$ has all edges centered if $F$ has injectivity radius $r\geq r_{\beta}$ at $x$.  It is a triangulation unless $r = r_{\beta}$ and each edge has length $d_{\beta}$, in which case it has a single quadrilateral $2$-cell.}
\begin{theorem}\label{main}\MainThm\end{theorem}
\newtheorem*{mainthm}{Theorem \ref{main}}

The numerical value of $r_{\beta}$ is roughly $1.7006$, whereas Bor\"oczky's theorem \cite{Bor} implies a  universal upper bound of $r_{\alpha} \cong 1.7191$ on the injectivity radius of a genus-$2$ hyperbolic surface at a point $x$ (see Lemma \ref{Boroczky}).  Example \ref{sharp Boroczky} describes a surface $F_{\alpha}$ with injectivity radius $r_{\alpha}$ at some $x_{\alpha}\in F_{\alpha}$, showing that Bor\"oczky's upper bound is sharp.  Example \ref{sharp DeBlois} describes a surface $F_{\beta}$ with injectivity radius $r_{\beta}$ at $x_{\beta}\in F_{\beta}$ and a quadrilateral $2$-cell in the Delaunay tessellation determined by $\{x_{\beta}\}$, and Example \ref{dull DeBlois} describes arbitrarily small deformations of $F_{\beta}$ with Delaunay tessellations that have non-centered edges.

For surfaces satisfying its hypotheses, Theorem \ref{main} has the following geometric consequence:

\newcommand\MainCor{The space $\calm^{(r_{\beta})}_2$ of closed, orientable hyperbolic surfaces $F$ with injectivity radius at least $r_{\beta}$ at some $x\in F$ is compact.  If $F\in\calm^{(r_{\beta})}_2$ has injectivity radius $r\geq r_{\beta}$ at $x$, the covering radius $J$ of $F$ at $x$ satisfies $\sinh J \leq \sqrt{2}\sinh r$.}
\begin{theorem}\label{inj to cov}\MainCor\end{theorem}
\newtheorem*{maincor}{Theorem \ref{inj to cov}}

We regard $\calm_2^{(r_{\beta})}$ above as a subspace of the moduli space $\calm_2$ of closed, orientable hyperbolic surfaces of genus $2$, given its usual topology, see eg.~\cite{FaMa}.  By the \textit{covering radius} of $F$ at $x$ we refer to the infimum of $r > 0$ such that $F$ is contained in the open $r$-neighborhood of $x$.  The bound above is sharp, realized on $F_{\beta}$ from Example \ref{sharp DeBlois}, at $x_{\beta}$.

Let us briefly recall the well-known Mumford compactness criterion \cite{Mumford}, that for any $r>0$, the set of surfaces with injectivity radius at least $r$ at \textit{every} point is compact in $\calm_2$.  The analogous generalization of Corollary \ref{inj to cov} does not hold, since if $F$ has small enough injectivity radius at $x$ it can by deformed by making a distant curve arbitrarily short while keeping the injectivity radius at $x$ constant.  Note that this increases the covering radius at $x$ to infinity.

The main ``result'' of the paper is really the construction of a \textit{centered dual tessellation}, and the attachment of \textit{admissible spaces} to its $2$-cells, a process we now outline.  In Section \ref{Voronoi dfn} we introduce terminology and define the Voronoi and Delaunay tessellations determined by a finite subset $\cals$ of a surface $F$.  This material is standard.  Section \ref{examples} gives a series of examples to motivate what follows, including a Delaunay tessellation that is not a triangulation (see Corollary \ref{sharp DeBlois Delaunay}), and another with a non-centered edge (see Lemma \ref{dull DeBlois Delaunay}).

Section \ref{dirichlet dual} gives deeper information on this failure of duality.  If an edge $e$ of the Voronoi tessellation $V$ is not centered, one may orient it pointing ``away'' from $\gamma$.  Lemma \ref{vertex polygon centered} asserts that a Delaunay $2$-cell $P_v$ is centered if and only if the associated vertex $v\in\calv^{(0)}$ is not the initial vertex of any non-centered edge.  Moreover each component of the union $\calv^{(1)}_n$ of non-centered edges is a tree with a canonical root vertex, by Lemma \ref{tree components}.  

We define the \textit{centered dual graph} $P^{(1)}_c$ to the Voronoi tessellation to be the union of edges of $P$ geometrically dual to centered edges of $V$, and show in the remainder of Section \ref{dirichlet dual} that $P^{(1)}_c$ is the one-skeleton of a cell decomposition $P_c$, the \textit{centered dual decomposition} of $F$, with vertex set $\cals$ (see Definition \ref{centered dual}).  By Proposition \ref{dual decomp}, each $2$-cell $Q$ of $P_c$ is the union of Delaunay cells $P_v$ such that $v\in Q\cap V^{(0)}$.  This is either $T^{(0)}$ for a component $T$ of $V^{(1)}_n$ or a vertex $v$ not contained in any such component, by Lemma \ref{centered or not}.  

Each $2$-cell of the Delaunay tessellation is \textit{cyclic}: all its vertices are equidistant from the center of the corresponding vertex of $V$ (see Lemma \ref{vertex radius}).  It follows that such cells are each determined up to isometry by their side length collections \cite{Schlenker} (cf.~\cite{Walter_poly1} or \cite{DeB_cyclic_geom}).  This does not hold for a centered dual $2$-cell $Q$ containing a component $T$ of $V^{(1)}_n$. Instead, in Section \ref{moduli} we will describe an \textit{admissible space} $\mathit{Ad}(\bd_{\calf})$, determined by $T$ and the side length collection $\bd_{\calf}$ of $Q$, that in some sense parametrizes all possible combinations of Delaunay cells that can comprise $Q$ with side length collection $\bd_{\calf}$.  In particular see Lemma \ref{admissible Voronoi}.

Our main application of the centered dual/admissible space construction is a machine for turning lower bounds on the side lengths of a centered dual $2$-cell with few edges into a good lower bound on its area, described in Section \ref{bounds}.  The corresponding problem for Delaunay $2$-cells is complicated by non-centeredness: as we observed in \cite{DeB_cyclic_geom}, the area of a non-centered cyclic polygon decreases as the length of its longest side increases.  

By Lemma \ref{admissible closure}, given a rooted tree $T$ and side length collection $\bd_{\calf}$, the sum of areas of the Delaunay polygons comprising the corresponding centered dual $2$-cell determines a continuous function on the closure $\overline{\mathit{Ad}}(\bd_{\calf})$ of $\mathit{Ad}(\bd_{\calf})$.  We show that if $T$ has only one or two edges, the minimum of this function occurs at one of a few tightly-prescribed places.  Section \ref{bounds} describes an algorithm  that produces lower bounds for the values at such locations, given lower bounds on the coordinates of $\bd_{\calf}$.

In fact we work in Sections \ref{moduli} and \ref{bounds} with the \textit{radius-$R$ defect}, which in best cases records the area of the region in a polygon but outside the union of disks of radius $R$ centered at its vertices.  Section \ref{centered packing} is devoted to establishing Proposition \ref{2-cell defect}, which asserts that this does hold for a centered dual $2$-cell $Q$.  This proves convenient in applications.  

In Section \ref{computations} we make some computations and prove Theorem \ref{main}.  We prove Theorem \ref{inj to cov} in Section \ref{consequences}.

\subsection*{Acknowledgments}

This paper answers an objection of Marc Culler.  The author thanks Peter Shalen and Steve Kerckhoff for helpful conversations.

\section{The Voronoi and Delaunay tessellations}\label{Voronoi dfn}

This section gives a self-contained introduction to the Voronoi and Delaunay tessellations determined by a finite subset of a hyperbolic surface.  Let us first establish some notation.  

If $\gamma$ is a geodesic in $\mathbb{H}^2$, a \textit{half-space bounded by $\gamma$} is the closure of a component of $\mathbb{H}^2 - \gamma$.  We will say the \textit{frontier} of $K\subset \mathbb{H}^2$ is $\mathit{fr}(K)\doteq K\cap \overline{\mathbb{H}^2-K}$.  Thus for instance $\gamma$ is the frontier of either half-space bounded by $\gamma$.  If $\{\calh_i\}$ is a collection of half-spaces, each with bounding geodesic $\gamma_i$, we say that $P = \bigcap_i \calh_i$ is a \textit{convex polygon} if it is nonempty and the collection $\{\gamma_i\}$ is locally finite; ie, for each $p\in P$ there is an open set $U$ and a finite collection $\{\gamma_{i_1},\hdots,\gamma_{i_n}\}$ of boundary geodesics such that $U\cap(\bigcup \gamma_i) \subset \bigcup_{j=1}^n \gamma_{i_j}$.  

An \textit{edge} (or \textit{side}) of $P = \bigcap \calh_i$ is $\gamma_i\cap P$ for some $i$ such that this intersection is non-empty or a singleton, and the \textit{boundary} $\partial P$ of $P$ is the union of its edges.  One finds that $\partial P$ is the topological frontier $P\cap\overline{\mathbb{H}^2 - P}$ of $P$ in $\mathbb{H}^2$.  A \textit{vertex} of $P$ is the nonempty intersection of two edges.  We will say $P$ is \textit{cyclic} if all its vertices are equidistant from some $v\in \mathbb{H}^2$, the \textit{center} of $P$, and $P$ is \textit{centered} if $v\in\mathit{int}\, P$.  The \textit{radius} of a cyclic polygon $P$ is $J$ such that $d(v,x) = J$ for each vertex $x$ of $P$, where $v$ is the center of $P$.

By the \textit{injectivity radius} of $\cals\subset\mathbb{H}^2$ we refer to the supremum of  the set of $r\geq 0$ such that $d(x,y)>2r$ for all distinct $x$ and $y$ in $\cals$.  If $\cals$ has injectivity radius $R > 0$, then for any distinct $x$ and $y$ in $\cals$ the open disk $B_{R}(x) \doteq \{p\in\mathbb{H}^2\,|\,d(x,p) < R\}$ is disjoint from $B_{R}(y)$.

\begin{fact}  If $\cals\subset\mathbb{H}^2$ has injectivity radius $r >0$ then $\cals \cap K$ is finite for any bounded set $K \subset \mathbb{H}^2$.\end{fact} 

This is because the $R$-neighborhood of $K$ has finite area and so cannot contain infinitely many disjoint disks with a fixed positive area.  It follows that if $\cals\subset\mathbb{H}^2$ has positive injectivity radius then it is closed and discrete.  The converse is not true, but a closed and discrete set $\cals$ does satisfy the fact above.  This will suffice to define the Voronoi and Delaunay tessellations.

For distinct $x$ and $y$ in $\mathbb{H}^2$, we will often refer by $\gamma_{xy}$ to the unique geodesic arc joining $x$ to $y$, and by $\gamma_{xy}^{\perp}$ to its \textit{perpendicular bisector}: the hyperbolic geodesic intersecting $\gamma_{xy}$ at its midpoint $m$, at right angles.  For each $p\in\gamma_{xy}^{\perp}$ the hyperbolic law of cosines gives:
$$ \cosh d(x,p) = \cosh d(x,m)\cosh d(p,m) = \cosh d(y,p) $$
Thus the points of $\gamma_{xy}^{\perp}$ are equidistant from $x$ and $y$; in fact $\gamma_{xy}^{\perp} = \{p\in\mathbb{H}^2\,|\, d(p,x) = d(p,y)\}$.

\begin{lemma}\label{Voronoi poly}  If $\cals\subset \mathbb{H}^2$ is closed and discrete then for each $x\in\cals$, $V_x$ as defined in (\ref{Voronoi cell}) is a convex polygon in $\mathbb{H}^2$, and if $\cals$ has injectivity radius $R > 0$ then $\overline{B_R(x)} \subset V_x$.  \end{lemma}

\begin{proof}  For $y \in \cals-\{x\}$, let $\calh_{xy}$ be the half-space containing $x$ and bounded by $\gamma_{xy}^{\perp}$.  Then $\calh_{xy} = \{p\,|\,d(p,x) \leq d(p,y)\}$.  It follows immediately that (\ref{Voronoi cell}) may be rewritten as:
$$V_x = \bigcap_{y\,\in\,\cals-\{x\}} \calh_{xy} $$
Fix $p\in V_x$ and let $J = d(x,p)$.  If $K =B_{3J}(p)$ then as we pointed out above the lemma, $K\cap \cals$ is a finite set $\{x,y_1,\hdots,y_n\}$.  If $U = B_{J/2}(p)$, then for $q\in U$ the triangle inequality gives $d(x,q) < 3J/2$, so for $y$ outside of $B_{3J}(p)$ it follows that $\gamma_{xy}^{\perp}\cap U = \emptyset$.  Therefore: 
$$U \cap \left(\bigcup_{y \in \cals - \{x\}} \gamma_{xy}^{\perp} \right) \subset \bigcup_{i=1}^n \gamma_{xy_i}^{\perp},$$
and $V_x$ is a convex polygon.  

If $\cals$ has injectivity radius $R>0$, then since the open disk $B_{R}(x)$ does not intersect $B_{R}(y)$ for any $y\in\cals-\{x\}$, the points of $B_{R}(x)$ are closer to $x$ than any $y\in\cals-\{x\}$.  Thus $B_{R}(x) \subset V_x$; in particular, $V_x$ is nonempty, and since it is closed it contains $\overline{B_{R}(x)}$.\end{proof}

It is easy to show that $\mathit{int}\, V_x = \{ p\,|\,d(p,x) < d(p,y)\ \mbox{for each}\ y\in \cals - \{x\}\}$, and that $\partial V_x = \bigcup_{y\in\cals-\{x\}} V_x\cap V_y$.  Furthermore, for any $y\in\cals-\{x\}$ such that $V_x\cap V_y$ is nonempty, it is contained in the equidistant locus $\gamma_{xy}^{\perp}$.  This gives:  

\begin{fact} If $\cals\subset \mathbb{H}^2$ is closed and discrete, then for distinct $x$, $y$, and $z$ in $\cals$, $V_x\cap V_y\cap V_z$ contains at most a single point.\end{fact}

This is because each point of $V_x\cap V_y\cap V_z$ is in both $\gamma_{xy}^{\perp}$ and $\gamma_{xz}^{\perp}$, and since $y\neq z$ these are distinct geodesics which therefore meet transversely in a single point (if at all).

\begin{lemma}\label{vertex radius}  For $\cals\subset \mathbb{H}^2$ closed and discrete, define $V^{(0)} = \{ V_x\cap V_y\cap V_z\,|\,x,y,z\in\cals\ \mbox{distinct}\}$.  For $v\in V^{(0)}$, there exists $J_v > 0$ such that  $v \in V_x$ if and only if $d(v,x) = J_v$ for each $x \in \cals$.  In particular, $\cals\cap B_{J_v}(v) = \emptyset$.  \end{lemma}


\begin{proof}  Suppose $x\in \cals$ has $v\in V_x$, and let $J_v = d(x,v)$.  It follows directly from the definition (\ref{Voronoi cell}) that $d(v,y) \geq J_v$ for all $y\in\cals- \{x\}$.  For some such $y$, if $v\in V_y$ then $d(v,x) \geq d(v,y)$ by the definition of $V_y$, so we must have $d(v,y) = J_v$.  This proves the lemma.\end{proof}

\begin{corollary}\label{discrete V^0}  For $\cals\subset\mathbb{H}^2$ closed and discrete, $V^{(0)}$ from Lemma \ref{vertex radius} is closed and discrete.\end{corollary}

\begin{proof}  For $v\in V^{(0)}$ let $J_v > 0$ be prescribed by Lemma \ref{vertex radius}.  Each point of $B_{J_v}(v)$ is within $2J_v$ of a point of $\cals$, so for $v'\in V^{(0)}\cap B_{J_v}(v)$ we have $J_{v'} < 2J_v$.  Thus for such $v'$, $B_{3J_v}(v)$ contains the set of $x\in\cals$ such that $v'\in V_x$.  The fact above Lemma \ref{vertex radius} implies that $v'$ is determined by this set, so since $\cals\cap B_{3J_v}(v)$ is finite there can be only finitely many $v'\in V^{(0)}\cap B_{J_v}(v)$.\end{proof}

\begin{fact} For $\cals\subset\mathbb{H}^2$ closed and discrete, $x\in\cals$, and each edge $e = V_x\cap V_y$ of $V_x$, $\mathit{int}(e)\cap V_z = \emptyset$ for each $z\in\widetilde{\cals}-\{x,y\}$.\end{fact} 

This is because $v = e\cap V_z$ is an intersection of the geodesics $\gamma_{xy}^{\perp}$ and $\gamma_{xz}^{\perp}$, being equidistant from $x$, $y$, and $z$, and so one interval of $\gamma_{xy}^{\perp} - v$ consists entirely of points closer to $z$ than $x$.  

By the fact above, the set of edges of Voronoi polygons has the structure of an embedded graph in $\mathbb{H}^2$ with vertex set $V^{(0)}$.

\begin{definition} For $\cals\subset \mathbb{H}^2$ closed and discrete, the \textit{Voronoi tessellation $V$ determined by $\cals$} is the cell complex structure with $2$-cells of the form $V_x$ for $x\in\cals$, and with $V^{(1)} = \bigcup\ \{ V_x\cap V_y\,|\,\mbox{distinct}\ x,y\in\cals\}$ and $V^{(0)} = \bigcup\ \{ V_x\cap V_y\cap V_z\,|\, x, y, z \in\cals\ \mbox{distinct}\}$.  \end{definition}

We caution that this definition does not imply that the Voronoi tessellation has trivalent one-skeleton; only that each vertex is contained in at least three $2$-cells.

\begin{definition}\label{cyclically ordered V}  Let $V$ be the Voronoi tesselation determined by $\cals\subset \mathbb{H}^2$ closed and discrete.  For $v\in V^{(0)}$, say the collection $\{e_0,\hdots,e_{n-1}\}$ of edges of $V$ containing $v$ is \textit{cyclically ordered} if for each $i$ there exists $x_i \in\cals$ so that $e_i$ and $e_{i+1}$ are edges of $V_{x_i}$ (taking $i+1$ modulo $n$).\end{definition}

Note that if the edges containing $v\in V^{(0)}$ are cyclically ordered $e_0,\hdots,e_{n-1}$ then $e_i = V_{x_i}\cap V_{x_{i-1}}$ for each $i$ (with $i-1$ taken modulo $n$), where the $x_i$ are as in the definition above.

\begin{lemma}\label{cyclically ordered vertices}  Let $V$ be the Voronoi tesselation determined by $\cals\subset \mathbb{H}^2$ closed and discrete, and fix $v\in V^{(0)}$.  Let $C = \{ p\in\mathbb{H}^2\,|\,d(v,p) = J_v\}$, let the edges containing $v$ be cyclically ordered $\{e_0,\hdots,e_{n-1}\}$ as in Definition \ref{cyclically ordered V}, and let $\{x_i\} \subset \cals$ be the associated collection with $e_i$ and $e_{i+1}\subset V_{x_i}$ for each $i$.   Then $\{x_i\}_{i=0}^{n-1}=\{x\in\cals\,|\,v\in V_x\}$ is cyclically ordered in the sense of \CyclicallyOrderedV. \end{lemma}

\begin{proof}  For $y\in\cals$, if $d(v,y) = J_v$ then $v$ is as close to $y$ as to any other element of $\cals$; hence $v\in V_y$.   Since $d(v,y) = J_v=d(v,x_i)$ for each $i$, $v$ is in the frontier of $V_y$ and hence is in an edge $e$ of $V_y$.  By hypothesis there is an $i$ such that $e=e_i$, so by the observation above the lemma $y = x_i$ or $x_{i-1}$.  This shows that $\{x\in\cals\,|\,v\in V_x\}=\{x_i\}_{i=0}^{n-1}$.

For each $i$, since $V_{x_i}$ is convex it contains the geodesic arc $\lambda_i$ joining $x_i$ to $v$.  Now fix some $i \in \{0,\hdots,n-1\}$ and a point $w\in e_i \cap B_{J_v}(v)$ near $v$.  Then by convexity of $V_{x_i}$, the triangle $T_i$ determined by $v$, $w$, and $x_i$ is entirely contained in $V_{x_i}$.  Also $T_i \subset \overline{B_{J_v}(v)}$, again by convexity (since $x_i \in C = \partial \overline{B_{J_v}(v)}$).  By the same argument, the triangle $T_{i-1}$ determined by $v$, $w$, and $x_{i-1}$ is contained in $V_{x_{i-1}}\cap\overline{B_{J_v}(v)}$.  (See Figure \ref{cyclic vertices}.)

\begin{figure}
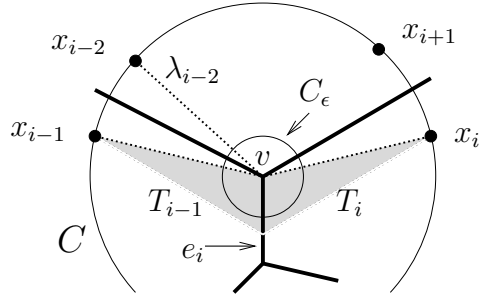
\caption{Some objects from the proof of Lemma \ref{cyclically ordered vertices}.  (Edges of $V$ are bold.)}
\label{cyclic vertices}
\end{figure}

Let $\alpha_{i-1}$ be the vertex angle of $T_{i-1}$ at $v$; since $w\in\mathit{int}(e_i)$ this is the angle in $V_{x_{i-1}}$ between $\lambda_{i-1}$ and $e_i$.  Similarly, the angle $\alpha_i$ of $T_i$ at $v$ is the angle in $V_{x_i}$ between $e_i$ and $\lambda_i$.  If $\epsilon > 0$ is less than the distance from $v$ to the side $[x_i,w]$ of $T_i$ joining $x_i$ and $w$, then $T_i$ contains the entire sector of $B_{\epsilon}(v)$, with angle $\alpha_i$, determined by $e_i$ and $\lambda_i$.  If $\epsilon<d(v,[x_{i-1},w])$ then the analogous assertion holds for the sector of $B_{\epsilon}(v)$ determined by $\lambda_{i-1}$ and $e_i$.

Fix $\epsilon > 0$ satisfying the requirements of the paragraph above and let $C_{\epsilon}$ be the circle of radius $\epsilon$ centered at $v$.  The above implies that $C_{\epsilon}\cap (T_{i-1}\cup T_i)$ is an interval of $C_{\epsilon}$ with angle measure $\alpha_{i-1}+\alpha_i$ and endpoints $\lambda_{i-1}\cap C_{\epsilon}$ and $\lambda_i\cap C_{\epsilon}$.  There is a homeomorphism $C\to C_{\epsilon}$ that takes $x\in C$ to $C_{\epsilon}\cap \lambda_x$, where $\lambda_x$ is the geodesic arc joining $x$ to $v$.  In particular, $x_j$ maps to $\lambda_j \cap C_{\epsilon}$ for each $j\in\{0,\hdots,n-1\}$.  Let $I$ be the preimage of $C_{\epsilon}\cap(T_{i-1}\cup T_i)$ in $C$, a closed subinterval bounded by $x_{i-1}$ and $x_i$.  For $j \neq i$ or $i-1$, $x_j$ maps to a point outside $C_{\epsilon}\cap(T_{i-1}\cup T_i)$ since $\lambda_j \subset V_{x_j}$, and so $x_j \in C - I$, a component of $C-\{x_{i-1},x_i\}$.
\end{proof}

We will define the Delaunay tessellation $P$ determined by $\cals\subset\mathbb{H}^2$ as a sort of ``dual'' to the Voronoi tessellation determined by $\cals$.  In particular we take $P^{(0)} = \cals$, in $1-1$ correspondence with the set of $2$-cells of $V$.  The edges of $P$ are also determined by edges of $V$:

\begin{definition}  Let $V$ be the Voronoi tesselation determined by $\cals\subset \mathbb{H}^2$ closed and discrete.  For an edge $e = V_x\cap V_y$ of $V$, the \textit{geometric dual} to $e$ is the geodesic arc $\gamma_{xy}$ joining $x$ to $y$. \end{definition}

We define $P^{(1)}$ to be the union of geometric duals to edges of $V$.  The lemma below establishes that $P^{(1)}$ has the structure of an embedded graph in $\mathbb{H}^2$ with vertex set $\cals$.

\begin{lemma}\label{embedded P^1}  Let $V$ be the Voronoi tesselation determined by $\cals\subset \mathbb{H}^2$ closed and discrete.  If $e = V_x\cap V_y$ and $e' = V_{x'} \cap V_{y'}$ are distinct edges of $V$, then their geometric duals satisfy $\gamma_{xy}\cap\gamma_{x'y'} = \{x,y\}\cap\{x',y'\}$.  \end{lemma}

\begin{proof}  If $x = x'$, say, then $y\neq y'$ since $e$ and $e'$ are distinct, so since $\gamma_{xy}$ and $\gamma_{x'y'}$ are geodesic arcs they intersect only at $x$.  We will thus assume that $\{x,y\}\cap\{x',y'\}=\emptyset$.  

If $e$ and $e'$ share a vertex $v$, then upon cyclically ordering the edges containing $v$ as $e_0,\hdots,e_{n-1}$ as in Definition \ref{cyclically ordered V}, we have $e = e_i$ and $e' = e_{i'}$ for distinct $i$ and $i'$ in $\{0,\hdots,n-1\}$.  Then $\{x,y\} = \{x_{i-1},x_i\}$ and $\{x',y'\} = \{x_{i'-1},x_{i'}\}$, as we observed above Lemma \ref{cyclically ordered vertices}, and Lemma \ref{vertex radius} implies that $\gamma_{xy}$ and $\gamma_{x'y'}$ are chords of the circle $C$ of radius $J_v$ centered at $v$.   

Chords of $C$ with distinct endpoints intersect if and only if the endpoints of one separate the endpoints of the other on $C$.  But Lemma \ref{cyclically ordered vertices} implies that $x_{i'-1}$ and $x_{i'}$ share a component of $C - \{x_i,x_{i-1}\}$, so $\gamma_{xy}\cap\gamma_{x'y'}= \emptyset$ in this case.  We therefore assume below that $e\cap e' = \emptyset$.

Let $v$ be the nearer of the two endpoints of $e$ to $x$ and $y$, and let $D_{v}$ be the disk of radius $J_v$ centered at $v$.  The same construction yields $v'$ and a disk $D_{v'}$ of radius $J_{v'}$ associated to $\gamma_{x'y'}$.  By Lemma \ref{vertex radius}, $\gamma_{xy}$ is a chord of the circle $\partial D_v$, and $\gamma_{x'y'}$ is a chord of $\partial D_{v'}$.  If $\gamma_{xy}$ intersects $\gamma_{x'y'}$, their intersection point is contained in $D_{v} \cap D_{v'}$. 

\begin{claim}  If distinct circles in $\mathbb{H}^2$ have intersecting chords with distinct endpoints, then one chord has an endpoint in the open disk in $\mathbb{H}^2$ complementary to the other circle.  \end{claim}

\begin{proof}[Proof of claim]  Let $C$ and $C'$ be distinct circles with chords $\gamma$ and $\gamma'$, respectively, that intersect.  If $C$ is contained in the open disk determined by $C'$, then the claim is immediate.  The same holds if $C'$ is contained in the open disk determined by $C$, so we will assume that neither of these possibilities occurs.  Then $C$ must intersect $C'$, since each chord of $C$ is contained in the disk that it bounds, and similarly for $C'$.  

If $C\cap C'$ is a single point, this must also be $\gamma\cap\gamma'$, an endpoint of each, so this cannot occur.  It follows that $C\cap C'$ consists of two points.  We may assume that neither endpoint of $\gamma$ is contained in the open disk complementary to $C'$, since otherwise the claim holds.  Since $\gamma$ intersects $\gamma'$, it nonetheless intersects $C'$.  If $\gamma\cap C'$ is a single point, then this is also $\gamma\cap\gamma'$, an endpoint of $\gamma'$ and therefore not an endpoint of $\gamma$.  Since this endpoint of $\gamma'$ is in the interior of $\gamma$, the claim holds in this case.

Let us now suppose that $\gamma\cap C'$ consists of two points, and let $\gamma_0$ be the closed subarc of $\gamma$ bounded by $\gamma\cap C'$.  Then $\gamma\cap\gamma' \subset \gamma_0$.  The geodesics $\gamma$ and $\gamma'$ intersect transversely, so $\gamma'$ has one endpoint in each of the open sub-arcs of $C'$ complementary to $C'\cap \gamma$.  Since one of these is contained in the open disk complementary to $C$, the claim holds.  \end{proof}

Using the claim, we may assume that the endpoint $x$ of $\gamma_{xy}$ is contained in the open disk with radius $J_{v'}$ centered at $v'$.  But this contradicts Lemma \ref{vertex radius}, and the result follows.  \end{proof}

The $2$-cells of the Delaunay tessellation are associated to $V^{(0)}$ by the lemma below.  Recall that a compact, convex polygon $P$ is \textit{cyclic} if its vertices are equidistant from a fixed point, its \textit{center} (see \cite{DeB_cyclic_geom}), and that the \textit{radius} of $P$ is the distance from its center to the vertices.

\begin{lemma}\label{vertex polygon}  Let $V$ be the Voronoi tesselation determined by $\cals\subset\mathbb{H}^2$ closed and discrete.  For each $v\in V^{(0)}$ there is a cyclic polygon $P_v$ in $\mathbb{H}^2$ with center $v$ and radius $J_v$ (as supplied by Lemma \ref{vertex radius}), such that: \begin{itemize}
\item  If the edges of $V$ containing $v$ are cyclically ordered $e_0,\hdots,e_{n-1}$, the vertex set of $P_v$ is the collection $\{x_i\}_{i=0}^{n-1}$ from Definition \ref{cyclically ordered V}.
\item  The edge set of $P_v$, cyclically ordered in the sense of \CyclicallyOrderedP, is $\{\gamma_i\}_{i=0}^{n-1}$, where $\gamma_i$ is the geometric dual to $e_i$ for each $i$.  Furthermore, $P_v \cap P^{(1)} = \gamma_{0}\cup\hdots \cup \gamma_{n-1}$.  \end{itemize}
For $v\neq w$, $\mathit{int}\,P_v\cap P_w = \emptyset$, and $P_v$ and $P_w$ share an edge if and only if $v$ and $w$ are opposite endpoints of an edge of $V$.  \end{lemma}

\begin{proof} Let $v$ and the collections $\{e_0,\hdots,e_{n-1}\}$, and $\{\gamma_0,\hdots,\gamma_{n-1}\}$ be as described in the hypotheses of the lemma.  The observation above Lemma \ref{cyclically ordered vertices} implies that for each $i$, $\gamma_i$ joins $x_{i-1}$ to $x_{i}$, where $\{x_i\}_{i=0}^{n-1}\subset \cals$ is the collection from Definition \ref{cyclically ordered V}.  Since the collection $\{x_i\}$ is cyclically ordered by Lemma \ref{cyclically ordered vertices}, Lemma 1.4 of \cite{DeB_cyclic_geom} asserts there is a cyclic $n$-gon $P_v$ center $v$, radius $J_v$, vertex set $\{x_i\}$ and edge set $\{\gamma_i\}$.  Furthermore, since the $x_i$ are cyclically ordered, the $\gamma_i$ are as well (see \CyclicallyOrderedP).

For any $x$ and $y\in\cals$, since $V_x$ and $V_y$ are convex their intersection is connected.  This implies in particular that for $i$ and $j \in \{0,\hdots, n-1\}$, if $V_{x_i}$ shares an edge $e$ with $V_{x_j}$ then $v\in e$.  Since the only edges of $V_{x_i}$ that contain $v$ are $e_i$ and $e_{i+1}$, it follows that $V_{x_i} \cap V_{x_j} = \{v\}$ unless $j = i\pm 1$ or $i$.  Therefore by definition, no edge of $P^{(1)}$ joins $x_i$ to $x_j$ for $j\neq i\pm 1$ (mod $n$).  It thus follows from Lemma \ref{embedded P^1} that $P_v \cap P^{(1)} = \bigcup_{i=0}^{n-1}\gamma_i$. 

If $p\in\mathit{int}\,P_v\cap P_w$, then since $\mathit{int}\, P_v\cap P^{(1)} = \emptyset$ and $\partial P_w\subset P^{(1)}$, $p\in\mathit{int}\, P_w$ and each geodesic ray from $p$ intersects $\partial P_v$ at or nearer to $p$ than its point of intersection with $\partial P_w$.  But since $\mathit{int}\, P_w\cap P^{(1)} = \emptyset$ and $\partial P_v\subset P^{(1)}$, each such point is in $\partial P_w$.  It follows that $\partial P_v = \partial P_w$, and hence that $P_v = P_w$ (again see \PtsToPoly).  This implies that $v = w$, since $P_v$ is cyclic and a circle (and hence also its center) is determined by three points on it.

If $P_{v}$ shares the edge $\gamma$ of $P^{(1)}$ with $P_{w}$, then by construction the edge $e$ of $V$ dual to $\gamma$ contains $v$ and $w$.  On the other hand, if $v$ and $w$ are vertices of an edge $e$ of $V$, then again by construction the edge of $P^{(1)}$ dual to $e$ is contained in $P_{v}$ and $P_{w}$.   \end{proof}

\begin{definition}\label{Delaunay tessellation}  Suppose $\cals\subset \mathbb{H}^2$ is closed and discrete.  We take the \textit{Delaunay tessellation} determined by $\cals$ to be the $2$-complex $P$ with vertex set $\cals$, edge set the geometric duals to edges of the Voronoi tessellation $V$, and $2$-cells $P_v$ supplied by Lemma \ref{vertex polygon}, for $v\in V^{(0)}$.  For such $v$ we will refer to $P_v$ as the associated \textit{vertex polygon}.  \end{definition}

The Delaunay tessellation $P$ is ``dual'' to the Voronoi tessellation $V$ in the sense that there is a canonical one-to-one correspondence between its $k$-cells and the $(2-k)$-cells of $V$ for each $k\in\{0,1,2\}$.  However, it is not necessarily dual in the sense of the intersection pairing: there is no reason in general that an edge of $V$ should intersect its geometric dual, or that $v\in V^{(0)}$ should be in $P_v$.

The Delaunay tessellation is sometimes defined using ``circumscribed circles,''  but this has its problems in the hyperbolic setting, as the example below will demonstrate.

\begin{figure}
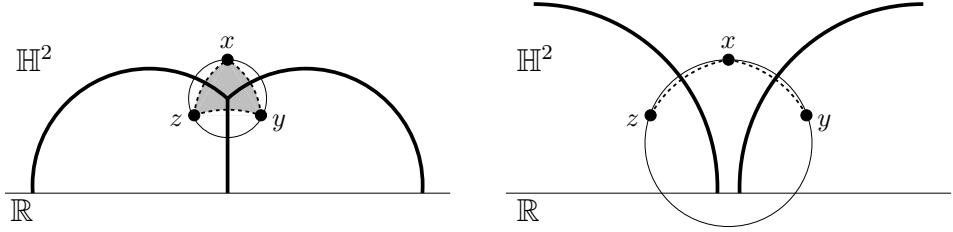
\caption{Combinatorial possibilities for the Voronoi and Delaunay tessellations determined by $\{x,y,z\}\subset\mathbb{H}^2$.  (Edges of $V$ are bold and of $P$, dashed.)}
\label{3 point Voronoi}
\end{figure}

\begin{example}\label{3 pt Vor possibilities}  Let $x, y, z\in\mathbb{H}^2$ be distinct, and suppose that $d(y,z) \geq \max\{d(x,y),d(x,z)\}$.  It follows from \SufficientParameters\ that $x$, $y$, and $z$ lie on a circle in $\mathbb{H}^2$ if and only if\begin{align}
\label{3 pt cyclic} \sinh(d(y,z)/2) < \sinh(d(x,y)/2)+\sinh(d(x,z)/2) \end{align}
We motivate this fact with Figure \ref{3 point Voronoi}, which uses the ``upper half-plane'' model for $\mathbb{H}^2$: the set of complex numbers with positive imaginary coordinate, equipped with the hyperbolic Riemannian metric.  It is well-known that in this model, each hyperbolic circle is also a Euclidean circle in $\mathbb{C}$, although with a different center and radius.  However, some $x$, $y$, and $z$ in $\mathbb{H}^2$ determine a Euclidean circle that does not lie entirely in $\mathbb{H}^2$, as illustrated on the right-hand side of the figure, even if they do not lie on a hyperbolic geodesic.

If (\ref{3 pt cyclic}) does not hold, then the equidistant locus $V_x \cap V_y$ does not intersect $V_z\cap V_y$, as on the right-hand side of Figure \ref{3 point Voronoi}.  In this case $V_x$ and $V_z$ are each a single half-plane, $V_y$ is bounded by two disjoint geodesics, and the Delaunay ``tessellation'' (as we have defined it) is the union of the dotted geodesic arcs $\gamma_{xy}$ and $\gamma_{xz}$.  If (\ref{3 pt cyclic}) does hold, then the Euclidean circle $C$ containing $x$, $y$, and $z$ lies in $\mathbb{H}^2$ and $V_x$, $V_y$, and $V_z$ intersect at its (hyperbolic) center $v$.  This is pictured on the left side of Figure \ref{3 point Voronoi}, with the vertex polygon $P_v$ shaded.
\end{example}

As in Example \ref{3 pt Vor possibilities}, the Delaunay tessellation determined by $\cals$ does not necessarily cover $\mathbb{H}^2$; indeed, if $\cals$ is finite then Lemma \ref{vertex polygon} implies that it is compact.  However, we are primarily concerned here with tessellations that arise from closed hyperbolic surfaces --- those which admit a locally isometric covering from $\mathbb{H}^2$.

\begin{lemma}\label{Voronoi Delaunay cover} Let $F$ be a closed hyperbolic surface, $\cals\subset F$ a finite set, and $\pi\co\mathbb{H}^2\to F$ a locally isometric universal covering.  Then $\widetilde{\cals}\doteq\pi^{-1}(\cals)$ has positive injectivity radius, and the Voronoi tessellation $\widetilde{V}$ and Delaunay tessellation $\widetilde{P}$ determined by $\widetilde{\cals}$ are invariant under the $\pi_1 F$-action on $\mathbb{H}^2$ by covering transformations.  Furthermore, $V_x$ as defined in (\ref{Voronoi cell}) is a compact polygon for each $x\in\widetilde{\cals}$, and $\widetilde{P}$ covers $\mathbb{H}^2$.  \end{lemma}

\begin{proof}  Since $\cals$ is finite and $F$ is compact there is a lower bound $r>0$ on the lengths of non-constant geodesic arcs in $F$ with endpoints in $\cals$.  The injectivity radius of $\widetilde{\cals}$ is then $r/2>0$.  Since $\widetilde{\cals}$ is invariant under the action of $\pi_1 F$, and this action is by isometries, it follows from (\ref{Voronoi cell}) that $p\in V_x$ if and only if $g.p \in V_{g.x}$ for $p\in \mathbb{H}^2$,  $x\in\widetilde{\cals}$, and $g\in\pi_1 F$.  Therefore $g.V_x = V_{g.x}$, and $g.(V_x\cap V_y) = V_{g.x}\cap V_{g.y}$ for $y\in\widetilde{\cals}-\{x\}$.  It follows that $\widetilde{V}$ is $\pi_1 F$-invariant, and from this that $\widetilde{P}$ is as well.  

Since $F$ is compact, there exists $R > 0$ such that $\bigcup_{y\in\cals} B(y,R)$ covers $F$ (here $B(y,R)$ is the open $R$-neighborhood of $y$ in the hyperbolic metric on $F$).  Then $V_x \subset B(x,R) \subset \mathbb{H}^2$ for each $x\in\widetilde{\cals}$, and hence it is compact.  In particular, $V_x$ has only finitely many edges and vertices.

We will show that $\widetilde{P}$ is open and closed in $\mathbb{H}^2$, and hence that it is all of $\mathbb{H}^2$.  First we claim that the collection of vertex polygons is locally finite: for any given $v$, Lemma \ref{vertex polygon} implies that $\mathit{int}(P_v)$ is disjoint from any other vertex polygon, and that the interior of an edge intersects exactly one other vertex polygon.  Each vertex of $P_v$ is some $x\in\widetilde{\cals}$, and since $V_x$ has only finitely many vertices $x$ is in only finitely many $P_v$.  The claim follows, and therefore $\widetilde{P} = \bigcup_{v\in \widetilde{V}^{(0)}} P_v$ is closed in $\mathbb{H}^2$.

We claim that $\widetilde{P}$ contains an open neighborhood of each $x\in\widetilde{\cals}$.  For such $x$, enumerate the edges of $V_x$ as $e_0,\hdots,e_{n-1}$ so that $e_i$ intersects $e_{i+1}$ in a vertex $v_i$ for each $i$, taking $i+1$ modulo $n$.  Then $x\in P_{v_i}$ for each $i$, and $P_{v_i}$ intersects $P_{v_{i-1}}$ along the geometric dual $\gamma_i$ to $e_i$ and $P_{v_{i+1}}$ along the geometric dual $\gamma_{i+1}$ to $e_{i+1}$ (here again take $i\pm 1$ modulo $n$).  For each $i$ there exists $\epsilon_i>0$ so that $P_{v_i}$ intersects $\overline{B_x(\epsilon_i)}$ in the full sector determined by $\gamma_i$ and $\gamma_{i+1}$, and we define $\epsilon =\min_i \{\epsilon_i\}$ and $C = \partial \overline{B_x(\epsilon)}$.  The claim follows from the fact that $C = \bigcup_{i=0}^{n-1} (P_{v_i} \cap C)$, since by the above this set is itself open and closed in $C$.  

Edges of $\widetilde{V}$ are compact, so for each $v\in \widetilde{V}^{(0)}$ and edge $\gamma$ of $P_v$, Lemma \ref{vertex polygon} gives $\gamma = P_v \cap P_w$, where $w$ is the other endpoint of the edge of $\widetilde{V}$ geometrically dual to $\gamma$.  Therefore any point in the interior of $\gamma$ has an open neighborhood in $\mathbb{H}^2$ contained in $P_v \cup P_w$.  The claim above implies that each vertex of $P_v$ also has an open neighborhood contained in $\widetilde{P}$, and it follows that $\widetilde{P}$ is open in $\mathbb{H}^2$.
\end{proof}

\begin{definition}\label{Delaunay down}  For a closed hyperbolic surface $F$, a locally isometric universal covering $\pi\co \mathbb{H}^2 \to F$, and $\cals \subset F$ finite, let $\widetilde{\cals} = \pi^{-1}(\cals)$ and take $V = \pi(\widetilde{V})$ and $P = \pi(\widetilde{P})$ to be the \textit{Voronoi tessellation} and \textit{Delaunay tessellation determined by $\cals$}, respectively, where $\widetilde{V}$ and $\widetilde{P}$ are as in Lemma \ref{Voronoi Delaunay cover}.\end{definition}

Since a convex polygon is homeomorphic to a disk, and $\pi$ takes the interior of each edge or $2$-cell of $\widetilde{V}$ or $\widetilde{P}$ isometrically to $F$, $V$ and $P$ have the structure of cell decompositions of $F$.  Note also that $\cals = \pi(\widetilde{\cals})$ is the vertex set of $P$.

\section{Examples and tools for recognition}\label{examples}

In this section we will take advantage of tools from \cite{DeB_cyclic_geom} for understanding the geometry of cyclic polygons, so let us begin by recalling some of its notation.

\begin{definition}[\cite{DeB_cyclic_geom}, Definition 2.1]\label{centered space}  For $n\geq 3$, let $\sigma\co\mathbb{R}^n\to\mathbb{R}^n$ be given by $\sigma(d_0,\hdots,d_{n-1})=(d_1,\hdots,d_{n-1},d_0)$, and refer by $\mathbb{R}^n/\mathbb{Z}_n$ to the quotient by the action of $\mathbb{Z}_n\doteq\langle\sigma\rangle$, and by $[d_0,\hdots,d_{n-1}]$ to the equivalence class in $\mathbb{R}^n/\mathbb{Z}_n$ of $(d_0,\hdots,d_{n-1})$.  Define:\begin{align*}
  \widetilde{\calAC}_n & = \left\{(d_0,\hdots,d_{n-1})\in(\mathbb{R}^+)^n\,|\,\sinh(d_i/2) < \sum_{j\ne i}\sinh(d_j/2)\ \mbox{for each}\ i\in\{0,\hdots,n-1\}\right\}  \\
  \widetilde{\calc}_n & = \left\{(d_0,\hdots,d_{n-1})\in(\mathbb{R}^+)^n\,|\,\sum_{j=0}^{n-1} A_{d_j}(d_i/2) > 2\pi,\ \mbox{where}\ d_i\geq d_j \forall j\in\{0,\hdots,n-1\}\right\}\end{align*}
Let $\calAC_n = \widetilde{\calAC}_n/\mathbb{Z}_n\subset\mathbb{R}^n/\mathbb{Z}_n$ and $\calc_n = \widetilde{\calc}_n/\mathbb{Z}_n\subset\calAC_n$.\end{definition}

The point of this definition is that by \UniqueParameters, each $(d_0,\hdots,d_{n-1})\in\widetilde{\calAC}_n$ determines a cyclic $n$-gon with cyclically ordered side length collection given by its entries; this $n$-gon is unique up to isometry of $\mathbb{H}^2$; and two such points determine the same (oriented) $n$-gon if and only if they have the same class in $\calAC_n$.  We will say a cyclic $n$-gon is \textit{represented} by $(d_0,\hdots,d_{n-1})\in\widetilde{\calAC}_n$ if this tuple describes its cyclically ordered side length collection.  \UniqueParameters\ further implies that each cyclic $n$-gon is represented by a point of $\widetilde{\calAC}_n$.

The function $A_d(J)$ used in the definition of $\widetilde{\calc}_n$ is defined in \ParameterFunction.  Each point in $\widetilde{\calc}_n$ determines a \textit{centered} $n$-gon, a cyclic polygon $P$ with center $v\in\mathit{int}\,P$ (recall from the beginning of Section \ref{Voronoi dfn} that the \textit{center} of $P$ is the center of the circle containing its vertices).  Conversely, if $(d_0,\hdots,d_{n-1})\in\widetilde{\calAC}_n$ represents a centered $n$-gon then it is in $\widetilde{\calc}_n$. (It is not immediately obvious that $\widetilde{\calc}_n\subset\widetilde{\calAC}_n$, but this is proved in \Ad.)

\begin{example}\label{sharp Boroczky}  Taking $D_0\co\widetilde{\calAC}_3\to\mathbb{R}^+$ as in \Defect, determine $d_{\alpha} > 0$ by:
$$ 4\pi = 6\cdot D_0(d_{\alpha},d_{\alpha},d_{\alpha}) = 6\cdot D_{0,3}(d_{\alpha}) = 6\cdot \left[ \pi - 6\sin^{-1}\left(\frac{1}{2\cosh(d_{\alpha}/2)}\right)\right] $$
The latter equalities above follow from \CenteredDefectBound.  Let $r_{\alpha} = d_{\alpha}/2$.  Rearranging the equation above and taking sines of both sides gives:\begin{align*}
  & \cosh r_{\alpha} = \frac{1}{2\sin(\pi/18)} \cong 2.8794 &
  & \cosh d_{\alpha} = \frac{1}{1-\cos(\pi/9)} - 1 \cong 15.5817 \end{align*}
By \UniqueParameters\ there is a centered triangle in $\mathbb{H}^2$, unique up to isometry, with all side lengths $d_{\alpha}$.  Six copies $T_1,\hdots,T_6$ of this triangle may be arranged in $\mathbb{H}^2$ so that they share a vertex and have disjoint interiors, and $T_i$ shares an edge with $T_{i\pm 1}$ for $1<i<6$.  Their union is thus an octahedron $O_{\alpha}$, with all side lengths $d_{\alpha}$ and area $4\pi$ by construction.  The Gauss-Bonnet formula implies that $O_{\alpha}$ has total angle defect $2\pi$, so its quotient by some scheme for pairing edges that reverses boundary orientations and identifies all vertices is a genus-$2$ surface $F_{\alpha}$.  Let $x_{\alpha}\in F_{\alpha}$ be the projection of the vertices of $O_{\alpha}$.

Since the angle measures of the $T_i$ total $2\pi$, each has angle $\pi/9$ at each vertex.  For each $i$, open disks of radius $r_{\alpha}$ centered at the vertices of $T_i$ do not intersect (see \PackingVsDecomp), and each intersects $T_i$ in a full sector of angle measure $\pi/9$.  Since the six non-overlapping $T_i$ comprise $O_{\alpha}$, a collection of open disks of radius $r_{\alpha}$ centered at each vertex of $O_{\alpha}$ intersects it in the non-overlapping union of $18$ sectors of angle measure $\pi/9$.  This projects to a hyperbolic disk embedded in $F_{\alpha}$, with center $x_{\alpha}$ and radius $r_{\alpha}$.  Since $O_{\alpha}$ has edge lengths $d_{\alpha} = 2r_{\alpha}$, $F_{\alpha}$ has injectivity radius $r_{\alpha}$ at $x_{\alpha}$.  \end{example}

Bor\"oczky's Theorem \cite{Bor} implies that $r_{\alpha}$ is the largest injectivity radius possible at any point in any genus-two hyperbolic surface.

\begin{lemma}\label{Boroczky}  A closed, orientable hyperbolic surface $F$ of genus $2$ has injectivity radius at most $r_{\alpha}$ at any $x\in F$.  \end{lemma}

\begin{proof}  Fix a locally isometric universal covering map $p\co \mathbb{H}^2\to F$.  If $F$ has injectivity radius $R$ at $x$, then by definition it contains an isometrically embedded open hyperbolic disk $D$, with radius $R$, centered at $x$.  Each point of $p^{-1}(x)$ is contained in a lift of $D$ to $\mathbb{H}^2$, and since $D$ is embedded in $F$ two such lifts do not overlap unless they are identical.  Thus $p^{-1}(D)$ is a packing of $\mathbb{H}^2$.

Let $\widetilde{V}$ be the Voronoi decomposition of $\mathbb{H}^2$ determined by $p^{-1}(x)$.  Since $p^{-1}(x)$ has injectivity radius $\alpha$, for any $\tilde{x}\in p^{-1}(x)$, $V_{\tilde{x}}$ contains the lift $\widetilde{D}$ of $D$ centered at $\tilde{x}$ (see Lemma \ref{Voronoi poly}).  The main theorem of \cite{Bor} implies:
$$ \frac{\mathit{area}(\widetilde{D})}{\mathit{area}(V_{\tilde{x}})} \leq d(R) = \frac{3\alpha(R)\cdot (\cosh R -1)}{\pi - 3\alpha(R)} $$
Here $\alpha(R)$ is the vertex angle of an equilateral triangle in $\mathbb{H}^2$ with sides of length $2R$.  Since $V_{\tilde{x}}$ projects onto $F$, isometrically on its interior, and $\widetilde{D}$ projects isometrically to $D$ we have:
$$ \frac{\mathit{area}(D)}{\mathit{area}(F)} \leq \frac{3\alpha(R)\cdot (\cosh R -1)}{\pi - 3\alpha(R)}  $$
Since $F$ has area $4\pi$ and $D$ has area $2\pi(\cosh R -1)$, the above inequality simplifies to $\alpha(R) \geq \pi/9$.  The hyperbolic law of cosines implies:
$$ \cos \alpha(R) = \frac{\cosh^2 (2R) -\cosh (2R)}{\sinh^2 (2R)} = \frac{\cosh (2R)}{\cosh (2R)+1} = 1 - \frac{1}{\cosh (2R) +1}  $$
Solving for $\cosh (2R)$ and applying the ``half-angle'' identities for the sine and hyperbolic cosine functions gives $\cosh R = 1/2\sin(\alpha(R)/2)\leq 1/2\sin(\pi/18)$.  The conclusion follows.
\end{proof}

\begin{example}\label{sharp DeBlois}  Let $d_{\beta}>0$ be determined by the following criterion:\begin{align*}
  4\pi & = 4\cdot D_0(d_{\beta},d_{\beta},d_{\beta}) + D_0(d_{\beta},d_{\beta},d_{\beta},d_{\beta}) = 4\cdot D_{0,3}(d_{\beta}) + D_{0,4}(d_{\beta})\\
    & = 4\cdot\left[\pi - 6\sin^{-1}\left(\frac{1}{2\cosh(d_{\beta}/2)}\right)\right] + \left[2\pi - 8\sin^{-1}\left(\frac{\sqrt{2}}{2\cosh(d_{\beta}/2)}\right)\right]\end{align*}
The latter equalities follow from \CenteredDefectBound.  Applying the identity $2\sin^{-1} x = \cos^{-1}(1-2x^2)$ and the half-angle identity for hyperbolic cosine, and re-arranging yields:
$$ \frac{\pi}{2} = 3\cos^{-1}\left(\frac{\cosh d_{\beta}}{\cosh d_{\beta} +1}\right) + \cos^{-1}\left(\frac{\cosh d_{\beta}-1}{\cosh d_{\beta} +1}\right) $$
After taking cosines of both sides and simplifying with trigonometric identities we find that $y = \cosh d_{\beta} +1$ satisfies $y^3 - 17y^2 + 16y -4$, and hence that $x = y-1$ is as described in Theorem \ref{main}.  With $r_{\beta} = d_{\beta}/2$ we have:\begin{align*}
  & \cosh d_{\beta} \cong 15.0166 &
  & \cosh r_{\beta}  \cong 2.8299  \end{align*}

Let $Q$ be a centered quadrilateral and $T_1,T_2,T_3,T_4$ centered triangles, each with all side lengths $d_{\beta}$.  These exist by \UniqueParameters.  Arrange them in $\mathbb{H}^2$ so that they are pairwise non-overlapping and $T_i$ shares an edge with $Q$ for each $i$.  Then $O_{\beta}\doteq Q \cup \left(\bigcup_i T_i\right)$ is a hyperbolic octahedron with area $4\pi$, and hence total angle defect $2\pi$.  An edge-pairing scheme as in Example \ref{sharp Boroczky} yields a genus-$2$ surface $F_{\beta}$, and arguing as in Example \ref{sharp Boroczky} we find that $F_{\beta}$ has injectivity radius $r_{\beta}$ at the point $x_{\beta}\in F_{\beta}$ descended from the vertices of $O_{\beta}$.
\end{example}

We now prove a few preliminary results that will allow us to pin down the Voronoi and Delaunay tessellations in Examples \ref{sharp Boroczky} and \ref{sharp DeBlois}.  

\begin{lemma}\label{shortest edge centered}  Let $V$ be the Voronoi tessellation determined by $\cals\subset\mathbb{H}^2$ with injectivity radius $R>0$, and let $B_0 = \cosh^{-1}(2\cosh(2R) -1)$.  For $x,y\in \cals$, if $d(x,y) \leq B_0$ then the midpoint $m$ of the geodesic arc $\gamma_{xy}$ joining $x$ to $y$ is in $V_x\cap V_y$.  If $d(x,y) < B_0$ then $\gamma_{xy}$ is the geometric dual to an edge $e = V_x\cap V_y$ of $V$ with $m = \gamma_{xy}\cap e\in\mathit{int}(e)$, and $V_x \cup V_y$ contains an open neighborhood of $\gamma_{xy}$.  \end{lemma}

\begin{proof}  For $z\in \cals - \{x,y\}$, let $\alpha_x\in[0,\pi]$ be the angle between $z$ and $x$ as measured from $m$, and let $\alpha_y$ be the angle from $z$ to $y$.  Since $m$ is in the interior of the geodesic arc $\gamma_{xy}$ we have $\alpha_x+\alpha_y = \pi$, so one of $\alpha_x$ and $\alpha_y$ is at most $\pi/2$.  Assuming (without loss of generality) that $\alpha_x \leq \pi/2$, the hyperbolic law of cosines gives:
\begin{align*}
  \cosh d(x,z) & = \cosh d(m,z)\cosh d(x,m) - \sinh d(m,z)\sinh d(x,m)\cos \alpha_x \\
    & \leq \cosh d(m,z)\cosh d(x,m)  \end{align*}
Let $R_0 = B_0/2$.  The ``half-angle identity'' for hyperbolic cosine implies that $R_0$ satisfies 
$$ \cosh R_0 = \sqrt{\frac{1}{2}(\cosh B_0 + 1)} = \sqrt{\cosh (2R)} $$
If $d(x,y) \leq B_0$ then $d(x,m) = \frac{1}{2}d(x,y) < R_0$.  Since $d(x,z) \geq 2R$, combining expressions above yields:
$$ \cosh d(m,z) \geq \frac{\cosh d(x,z)}{\cosh d(x,m)} \geq \frac{\cosh (2R)}{\sqrt{\cosh (2R)}} = \cosh R_0 $$
Thus $m$ is at least as close to $x$ and $y$ as to $z$ and, since $z\in\cals$ is arbitrary, $m \in V_x\cap V_y$.

If $d(x,y) < B_0$, let $\eta=\frac{\cosh R_0}{\cosh d(x,m)}>1$.  The inequality above gives $\cosh d(m,z) \geq \eta\cdot \cosh R_0$ in this case.  Thus if $R_1 = \cosh^{-1}(\eta\cosh R_0)$ and $\delta = R_1 - R_0$, for $p\in B_{\delta/2}(m)$ the triangle inequality gives:
$$d(p,x) < R_0 + \delta/2 = R_1 -\delta/2 < d(p,z)$$  
Thus in this case $B_{\delta/2}(m)\subset V_x \cup V_y$, and if $\gamma_{xy}^{\perp}$ is the perpendicular bisector to $\gamma_{xy}$ then $\gamma_{xy}^{\perp} \cap B_{\delta/2}(m) \subset e = V_x\cap V_y$.  In particular, $m = \gamma_{xy}^{\perp}\cap \gamma_{xy}\in \mathit{int}(e)$, and $\mathit{int}(V_x) \cup B_{\delta/2}(m) \cup \mathit{int}(V_y)$ is an open neighborhood of $\gamma_{xy}$ in $V_x\cup V_y$.  \end{proof}

\begin{lemma}\label{centered poly Voronoi}  Let $P$ be a centered polygon in $\mathbb{H}^2$ with center  $v\in\mathit{int}\, P$.  For a vertex $x$ of $P$, let $Q_x \subset P$ be the quadrilateral with vertices $v$, $x$, and the midpoints of the edges of $P$ containing $x$.  Then $P = \bigcup Q_x$, taken over all vetices of $P$.  For $y\in P$, $y\in Q_x$ if and only if $d(y,x) \leq d(y,x')$ for each vertex $x'$ of $P$.  \end{lemma}

\begin{proof}  Let $J$ be the radius of $P$ --- ie, the distance from $v$ to the vertices of $P$ --- and let $x_{\pm}$ be the vertices adjacent to $x$ on $\partial P$.  The geodesic arc $e_+$ that joins $v$ to the midpoint $m_+$ of the edge $\gamma_+$ of $P$ containing $x$ and $x_+$ meets $\gamma_+$ at a right angle, since it is the fixed axis of a reflective involution of the isosceles triangle $\Delta_+$ with vertices $v$, $x$, and $x_+$.  Thus $e_+$ is contained in the perpendicular bisector $\gamma_+^{\perp}$.  The same holds true for the other edge $e_-$ of $Q_x$ containing $v$. and it follows that points of $Q_x$ are at least as close to $x$ as to either of $x_{\pm}$.

The center $v$ is in $Q_x$ and satisfies $d(v,x) = J = d(v,x')$ for all other vertices $x'$ of $P$, so the conclusion holds for $v$.  Fix $y\in Q_x - \{v\}$.  If $x'$ is a vertex of $P$ other than $x$, $x_+$ or $x_-$, then the geodesic arc from $y$ to $x'$ crosses one of $e_{\pm}$, say $e_+$.  Let $T$ be the triangle with vertices $v$, $y$, and $x$, and let $T'$ have vertices $v$, $y$, and $x'$.  Each of $T$ and $T'$ has an edge with length $d(y,v)$ and an edge with length $J = d(x,v) = d(x',v)$.  We consider two cases.

\begin{figure}
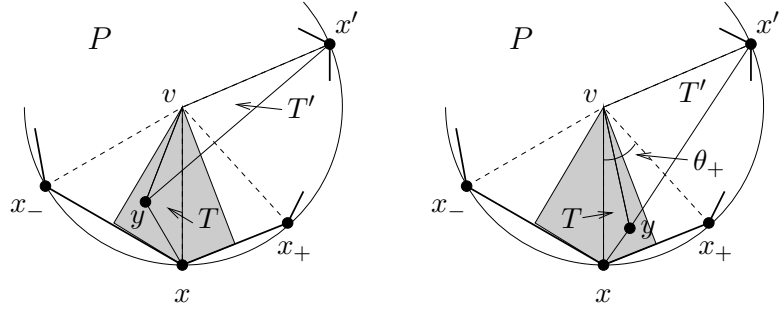
\caption{Two possibilities from Lemma \ref{centered poly Voronoi}.  $Q$ is shaded.}
\label{ctrd poly Vor}
\end{figure}

If the geodesic arc from $y$ to $x'$ crosses the arc from $v$ to $x$, as on the left-hand side of Figure \ref{ctrd poly Vor}, then clearly the angle $\theta'$ of $T'$ at $v$ is larger than the angle $\theta$ of $T$ at $v$.  Hence the hyperbolic law of cosines implies in this case that $d(x,y) < d(x',y)$.  

If not, then $\theta \leq \theta_+/2$, where $\theta_+$ is the angle at $v$ of the isosceles triangle $\Delta_+$ determined by $v$, $x$, and $x_+$.  This is because $y$ lies in $\Delta_+$ between its bisector $e_+$ and the edge joining $v$ to $x$ (see the right-hand side of Figure \ref{ctrd poly Vor}).  On the other hand, the arc from $y$ to $x'$ exits $\Delta_+$ at a point in the edge joining $v$ to $x_+$, since it crosses $e_+$.  Therefore $\theta' > \theta/2 \geq\theta$, and again by the hyperbolic law of cosines we have $d(x,y) < d(x',y)$.  

For $y\neq v$, the geodesic ray from $v$ in the direction of $y$ exits $P$ at a point in some edge $\gamma$.  By construction $\gamma\subset Q_x\cup Q_{x'}$, where $x$ and $x'$ are the endpoints of $\gamma$.  Thus $y$ is in $Q_x$ or $Q_{x'}$, say $Q_x$, since $v\in Q_x\cap Q_{x'}$ and these are convex quadrilaterals.  For another vertex $x''$, if $d(y,x'') \leq d(y,x)$ then $d(y,x'') = d(y,x)$, by the property of $Q_{x}$ that we showed above.  It follows that $x$ is adjacent to $x''$ on $\partial P$, since we showed that $d(x'',y) > d(x,y)$ otherwise.  In this case $Q_x\cap Q_{x''}$ is the intersection of $Q_{x}$ with the equidistant locus of $x$ and $x''$ by construction, so $y$ is in $Q_x\cap Q_{x''}$.  The lemma follows.
\end{proof}

\begin{lemma}\label{centered poly radius}  Let $P$ be a centered polygon in $\mathbb{H}^2$ with radius $J$ and center $v$.  For each $y\in P - \{v\}$ there is a vertex $x$ such that $d(x,y) < J$.\end{lemma}

\begin{proof}  Let $\{x_i\}_{i=0}^n$ be the set of vertices of $P$, enumerated so that for each $i$ there is an edge of $P$ containing $x_i$ and $x_{i+1}$ (with $i+1$ taken modulo $n$), and for each $i$ let $Q_i$ be the quadrilateral associated to $x_i$ by the construction of Lemma \ref{centered poly Voronoi}.  It is clear by construction that $P = \bigcup_{i=0}^{n-1} Q_i$, so for any $y\in P-\{v\}$ there exists $i\in\{0,\hdots,n-1\}$ so that $y\in Q_i$.  The conclusion of Lemma \ref{centered poly Voronoi} implies that $x_i$ is a closest vertex to $y$.

\begin{claim}  For any $i\in\{0,1,\hdots,n-1\}$ and $y\in Q_i- \{v\}$, $d(y,x_i) < J$. \end{claim} 

\begin{proof}  The geodesic ray from $x_i$ through $y$ intersects one of the edges of $Q_i$ containing $v$ at a point $y_0$.  If $y_0 = v$, then since $y\neq v$ is on the geodesic arc joining $x_i$ to $v$, the claim follows immediately.  Otherwise let us consider the right triangle determined by $x_i$, $y_0$, and the other endpoint of the edge containing $v$ and $y_0$, call it $m$.  The hyperbolic law of cosines gives: 
$$  \cosh d(x_i,y_0) = \cosh d(x_i,m)\cosh d(m,y_0) = \cosh R \cosh d(m,y_0) $$
Since $y_0$ is contained in the geodesic arc joining $m$ to $v$, and is not $v$, we have $d(m,y_0) < d(m,v)$.  When applied to the triangle determined by $x_i$, $v$, and $m$, the hyperbolic law of cosines gives $\cosh R \cosh d(m,v) = \cosh J$.  Thus by the above $d(x_i,y_0) < J$, and since $d(x_i,y) \leq d(x_i,y_0)$, the claim follows. \end{proof}

The lemma follows immediately.  \end{proof}

\begin{proposition}\label{shortish poly}  Let $\cals\subset \mathbb{H}^2$ have injectivity radius $R>0$.  If a centered $n$-gon $P$ in $\mathbb{H}^2$ has vertices in $\cals$, sides of length less than $B_0$ (from Lemma \ref{shortest edge centered}), and $\mathit{int}(P)\cap \cals = \emptyset$, then:\begin{enumerate}
\item $P = P_v$ as in Lemma \ref{vertex polygon}, where $v$ is the center of $P$; and
\item the Voronoi tessellation $V$ determined by $\cals$ satisfies $V^{(0)}\cap P = \{v\}$, and $V^{(1)}\cap P$ is the union of geodesic arcs joining $v$ to the midpoint of each side of $P$; and 
\end{enumerate}\end{proposition}

\begin{proof} Let the vertices of $P$ be cyclically ordered $x_0,\hdots,x_{n-1}$ in the sense of \CyclicallyOrderedV, and for each $i$ let $\gamma_i$ be the edge of $P$ joining $x_{i-1}$ to $x_i$ (taking $i-1$ modulo $n$).  By Lemma \ref{shortest edge centered}, $\gamma_i$ has an open neighborhood contained in $V_{x_{i-1}}\cup V_{x_i}$ for each $i$.  Since $\mathit{int}\,P\cap \cals = \emptyset$ and Voronoi cells are connected, it follows that $P\subset\bigcup_{i=0}^{n-1} V_{x_i}$.

For each $i \in \{0,\hdots,n-1\}$, Lemma \ref{centered poly Voronoi} implies that the quadrilateral $Q_{x_i}$ described there is contained in $V_{x_i}$, since its points are as close to $x_i$ as any $x_j$ for $j\neq i$.  Since $P = \bigcup_{i=0}^{n-1} Q_i$ it follows that $Q_i = P\cap V_{x_i}$ for each $i$.  The description of the Voronoi tessellation follows immediately; in particular, $Q_{x_{i-1}}\cap Q_{x_i}$ is contained in an edge $e_i = V_{x_{i-1}}\cap V_{x_i}$ containing $v$ for each $i$.  For each $i$, $\gamma_i$ is the geometric dual to $e_i$, so Lemma \ref{vertex polygon} implies that $P = P_v$.\end{proof}

\begin{corollary}\label{short symmetric poly}  Let $\cals\subset \mathbb{H}^2$ have injectivity radius $R>0$.  For $3\leq n \leq 6$, a cyclic $n$-gon $P$ with vertices in $\cals$ and all sides of length $2R$ satisfies the conclusions of Proposition \ref{shortish poly}.\end{corollary}

\begin{proof}  Since $B_0 > 2R$ by its definition in Lemma \ref{shortest edge centered}, the result will follow from Proposition \ref{shortish poly} once we show that $P$ is centered and $\mathit{int}\,P\cap\cals=\emptyset$. $P$ is represented by $(d,\hdots,d)\in\widetilde{\calAC}_n$.  \CenteredDefectBound\ implies that $(d,\hdots,d)\in\widetilde{\calc}_n$, and hence that $P$ is centered, and furthermore that its  radius $J = J_n(d)$ satisfies $\sinh J = \sinh(d/2)/\sin(\pi/n)$.  If $n\leq 6$ then $\sin(\pi/n) \geq 1/2$, so $\sinh J \leq 2\sinh R < \sinh (2R)$.  Since the hyperbolic sine is increasing on $\mathbb{R}^+$ it follows that $J < 2R$.   Thus since the points of $P$ have distance at most $J$ from $x$ by Lemma \ref{centered poly radius}, and $J < 2R$ in this case, $P$ has no points of $\cals$ but the $x_i$.  \end{proof}

\begin{corollary}  Let  $F_{\alpha}$ and $x_{\alpha}\in F_{\alpha}$ be as in Example \ref{sharp Boroczky}.  The Delaunay tessellation of $F_{\alpha}$ determined by $\cals = \{x_{\alpha}\}$ is the triangulation by the projections of $T_1,\hdots,T_6$ described there, and each edge intersects the interior of its geometric dual.  \end{corollary}

\begin{proof}  Let $O_{\alpha} = T_1\cup\hdots\cup T_6\subset\mathbb{H}^2$, as described in Example \ref{sharp Boroczky}.  Given a scheme for isometrically pairing edges of $O_{\alpha}$ to produce $F_{\alpha}$ as in Example \ref{sharp Boroczky}, for each pair of edges $e$ and $e'$, there is an orientation--preserving isometry $f$ of $\mathbb{H}^2$ with $f(e) = e'$ and $f(O_{\alpha})\cap O_{\alpha} = e'$.  The Poincar\`e polyhedron theorem asserts that the set of these edge pairings generates a discrete group $\Pi$ of isometries with fundamental domain $O_{\alpha}$, and that the quotient map $\mathbb{H}^2 \to \mathbb{H}^2/\Pi = F_{\alpha}$ is a locally isometric universal covering.  In particular, $\Pi$-translates of $O_{\alpha}$ tessellate $\mathbb{H}^2$.  

Since $O_{\alpha}$ is itself tessellated by the $T_i$, $\mathbb{H}^2$ is tessellated by $\Pi$-translates of these six triangles.  The preimage $\widetilde{\cals}$ of $\cals$ in $\mathbb{H}^2$ is the set of vertices of $\Pi$-translates of $O_{\alpha}$, so the vertices of any $\Pi$-translate of any $T_{i}$ are in $\widetilde{\cals}$.  $F_{\alpha}$ has injectivity radius $r_{\alpha}$ at $x_{\alpha}$, so $\widetilde{\cals}$ also has injectivity radius $r_{\alpha}$.  Thus since $T_i$ has edge length $d_{\alpha}=2r_{\alpha}$ for each $i$, Corollary \ref{short symmetric poly} implies that each translate of each $T_i$ is a two-cell of the Delaunay tessellation of $\mathbb{H}^2$ determined by $\widetilde{\cals}$.  Lemma \ref{shortest edge centered} further implies that each edge of each $T_i$ intersects the interior of its geometric dual, and the conclusion for $F_{\alpha}$ follows from Definition \ref{Delaunay down}. \end{proof}

The Voronoi tessellation from Example \ref{sharp Boroczky} is easily described.  A similar proof establishes:

\begin{corollary}\label{sharp DeBlois Delaunay}  Let $T_{\beta}$, $Q$, $F_{\beta}$, and $x_{\beta}\in F_{\beta}$ be as in Example \ref{sharp DeBlois}.  The Delaunay tessellation of $F_{\beta}$ determined by $\{x_{\beta}\}$ is the decomposition described there, into $T_1,\hdots,T_4$ and $Q$, all with side lengths $d_{\beta}$. Each edge intersects the interior of its geometric dual.  \end{corollary}

In particular, the Delaunay tessellation of $F_{\beta}$ is not a triangulation.  The example below shows that the conclusion of Theorem \ref{main} fails upon slightly relaxing the hypothesis $r \geq r_{\beta}$.

\begin{example}\label{dull DeBlois}  We will produce a family of surfaces $F_t$ by perturbing the surface $F_{\beta}$ from Example \ref{sharp DeBlois}.  By \HalfSymmQuadDefect, the equation determining $d_{\beta}$ can be rewritten as
$$ 4\pi = 4\cdot D_{0,3}(d_{\beta}) + 2\cdot D_0(b_{\beta},d_{\beta},d_{\beta}), $$
where $b_{\beta} = \cosh^{-1}(2\cosh d_{\beta} -1) = b_0(d_{\beta},d_{\beta})$.  This reflects the geometric observation, also in \HalfSymmQuadDefect, that the quadrilateral $Q$ from Example \ref{sharp DeBlois} has a diagonal that contains its center and divides it into triangles $T_0$ and $T_0'$, each with side length collection $(b_{\beta},d_{\beta},d_{\beta})$.  Re-naming if necessary, we will  assume that $T_0'$ shares an edge with $T_1$ from Example \ref{sharp DeBlois}.  

For $t$ near $0$ let $d_t = d_{\beta}+t$, $b_t = \cosh^{-1}(2\cosh d_t -1)$, and let $d_1(t)$ satisfy $d_1(0) = d_{\beta}$ and $f(t,d_1(t))\equiv 4\pi$, where:\begin{align}\label{dull constant area}
  f(t,d) =  3\cdot D_{0,3}(d_t) + D_0(d,d_t,d_t) + D_0(b_t,d_t,d) + D_0(b_t,d_t,d_t) \end{align}
We note that $d_0 = d_{\beta}$ and $b_0 = b_{\beta}$, and by comparing with the equation above one finds that $f(0,d_{\beta}) = 4\pi$.  To produce $d_1(t)$ we note that \DefectDerivative\ implies:\begin{align}\label{dull partial}
  \frac{\partial f}{\partial d} (0,d_{\beta})  = \sqrt{\frac{1}{\cosh^2(d_{\beta}/2)} - \frac{1}{\cosh^2 J_3(d_{\beta})}} > 0,\end{align}
since $(d_{\beta},d_{\beta},d_{\beta})\in \widetilde{\calc}_3$ by \CenteredDefectBound, and $J_3(d_{\beta}) < d_{\beta}/2$.  Therefore the implicit function theorem yields  $\epsilon > 0$ and a function $d_1$ on $(-\epsilon,\epsilon)$ with $d_1(0)=d_{\beta}$ and $f(t,d_1(t)) = 4\pi$ for each $t\in(\epsilon,\epsilon)$.  With a computation analogous to the above it is possible to show that $\partial f/\partial t$ is also positive at $(0,d_{\beta})$, and this further implies that $d_1$ decreases in $t$.

By \CenteredDefectBound, $(d_t,d_t,d_t)\in\widetilde{\calc}_3$ for each $t$ such that $d_t >0$; we may as well assume this is all of $(-\epsilon,\epsilon)$.  Moreover, $(b_t,d_t,d_t)\in\widetilde{\calBC}_3$ for each such $t$ by the definition of $b_t$ and \TriCenter.  We may also assume that for each $t\in(-\epsilon,\epsilon)$, each of $(b_t,d_t,d_1(t))$, and $(d_1(t),d_t,d_t)$ is in $\widetilde{\calAC}_3$, since this set is open in $\mathbb{R}^3$. 

For each $t$, let $T_2(t)$, $T_3(t)$, and $T_4(t)$ be centered triangles with all side lengths $d_t$.  Let $T_0(t)$ be a cyclic triangle with cyclically ordered side length collection $(b_t,d_t,d_t)$, let $T_0'(t)$ have side length collection $(b_t,d_t,d_1(t))$, and let $T_1(t)$ have side length collection $(d_1(t),d_t,d_t)$.  That these exist follows from \CenteredSpace\ and \UniqueParameters.

Note that $T_i(0) = T_i$ for $0\leq i \leq 4$, and $T_0'(0) = T_0'$.  Arranging the triangles in $\mathbb{H}^2$ so that at time $0$ their union is $O_{\beta}$, and they have the same combinatorial pattern of intersection for all time, their union at each time $t$ is an octagon $O_t$ with all side lengths $d_t$ and area $4\pi$ by construction.  The total angle defect of $O_t$ is thus $2\pi$, so an isometric edge-pairing scheme that is combinatorially identical to that for $O_{\beta}$ produces a surface $F_t$.  Let $x_t\in F_t$ be the quotient of the vertices of $O_t$.  One can show as in the previous examples that for each $t<0$, $F_t$ has injectivity radius $r_t = d_t/2$ at $x_t$.  Since $d_1$ decreases in $t$, $d_1(t) < d_t$ for $t > 0$, so $F_t$ has injectivity radius at most $d_1(t)/2$ for such $t$. \end{example}

We prove in the lemma below that for $t < 0$, the conclusion of Theorem \ref{main} does not apply to the surfaces $F_t$ from Example \ref{dull DeBlois}, though the Delaunay tessellation is a triangulation.  Recall from the example that for such $t$, $F_t$ has injectivity radius $r(t)< r_{\beta}$ at $x_t$.

\begin{lemma}\label{dull DeBlois Delaunay}  For $O_t$, $F_t$, and $x_t\in F_t$ as in Example \ref{dull DeBlois}, and $t<0$ but near to it, the triangulation that $F_t$ inherits from $O_t$ is its Delaunay tessellation determined by $\{x_t\}$.  Each edge of this triangulation intersects the interior of its geometric dual edge except for $T_0(t) \cap T_0'(t)$, which intersects an endpoint of its geometric dual.  \end{lemma}

\begin{proof}  For each $t$, the Poincar\`e polyhedron theorem implies that $\Pi_t$-translates of $O_t$ tessellate $\mathbb{H}^2$, where $\Pi_t$ is the group generated by the edge-pairing isometries of $O_t$ yielding $F_t$.  It follows that $\mathbb{H}^2$ is triangulated by $\Pi_t$-translates of $T_0'(t)$ and the $T_i(t)$ for $0\leq i\leq 4$.  The preimage $\widetilde{\cals}_t$ of $x_t$ in $\mathbb{H}^2$ is the set of vertices of $\Pi_t$-translates of $O_t$.  

For $i= 2$, $3$, or $4$, since $T_i(t)$ has all side lengths equal to $d_t = 2r_t$, Corollary \ref{short symmetric poly} implies that each of its $\Pi_t$-translates is a $2$-cell of the Delaunay tessellation $\widetilde{P}$ of $\mathbb{H}^2$ determined by $\widetilde{\cals}$.  Lemma \ref{shortest edge centered} applies to the edges of these polygons, and asserting in particular that each intersects the interior of its geometric dual.

For reference we have depicted $T_0(t) \cup T_0'(t) \cup T_1(t)$ in Figure \ref{non-centered Voronoi}, and labeled its vertices.  By construction, each $x_i$ is in $\widetilde{\cals}$, and each frontier edge has length $d_t$.  Lemma \ref{shortest edge centered} thus implies that each frontier edge is in $\widetilde{P}^{(1)}$, intersects the interior of its geometric dual, and has an open neighborhood contained in the union of the Voronoi cells determined by its endpoints.

\begin{figure}
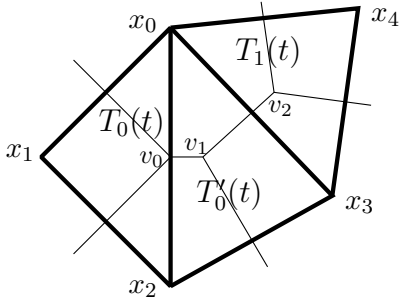
\caption{$P = T_0\cup T_1(\delta)\cup T_2(\delta)$ and its intersection with $\widetilde{V}^{(1)}$.}
\label{non-centered Voronoi}
\end{figure}

The edge $T_0'(t)\cap T_1(t)$, joining $x_0$ to $x_3$ in the figure, has length $d_1(t)$.  Since $d_1(0) = d_{\beta} < b_{\beta} = b_0$, the inequality $d_1(t) < b_t$ holds for $t$ near to $0$.  For such $t$, since $d_1(t) < b_t = \cosh^{-1}(2\cosh d_t -1)$, Lemma \ref{shortest edge centered} asserts that $T_0'(t)\cap T_1(t)$ is in $\widetilde{P}^{(1)}$ and intersects the interior of its geometric dual, and \TriCenter\ implies that $(d_1(t),d_t,d_t)\in\widetilde{\calc}_3$.  Therefore $T_1(t)$ is centered, and since it intersects $\widetilde{\cals}$ in its vertex set, Proposition \ref{shortish poly} implies it is a Delaunay $2$-cell intersecting $\widetilde{V}^{(1)}$ as illustrated in Figure \ref{non-centered Voronoi}.

It remains to consider $T_0(t)\cup T_0'(t)$.  We have already showed that its frontier is in $P^{(1)}$, and moreover has an open neighborhood contained in $V_{x_0}\cup V_{x_1}\cup V_{x_2}\cup V_{x_3}$.  Since $T_0(t)\cup T_0'(t)\cap \widetilde{\cals} = \{x_0,x_1,x_2,x_3\}$ and Voronoi cells are connected, $T_0(t)\cup T_0'(t)$ is entirely contained in $V_{x_0}\cup V_{x_1}\cup V_{x_2}\cup V_{x_3}$.  We claim that $V_{x_0}$ intersects $V_{x_2}$ in an edge $e$, whose geometric dual $T_0(t)\cap T_0'(t)$ is thus in $\widetilde{P}^{(1)}$ and furthermore intersects $e$ in an endpoint.  It will follow immediately that $T_0(t)$ and $T_0'(t)$ are each Delaunay $2$-cells.

Since $\cosh b_t = 2\cosh d_t -1$ by construction, \TriCenter\ implies that $(b_t,d_t,d_t)\in\widetilde{\calBC}_3$.  Therefore $T_0(t)$ has its center $v_0$ at the midpoint of its longest edge $T_0(t)\cap T_0'(t)$ by \InCenteredClosure.  On the other hand, we pointed out below (\ref{dull partial}) that $d_1$ is decreasing in $t$, so $d_1(t) > d_1(0) = d_{\beta} > d_t$ for $t < 0$.  Therefore $\cosh b_t < \cosh d_t + \cosh d_1(t) -1$, so $(b_t,d_t,d_1(t))\in\widetilde{\calc}_3$ by \TriCenter\ again. Thus $T_0'(t)$ has its center $v_1$ in its interior.

By definition $v_0$ is equidistant from $x_0$, $x_1$ and $x_2$, and $v_1$ is equidistant from $x_0$, $x_2$, and $x_3$.  The quadrilaterals $Q_{x_0}$ and $Q_{x_2}$ in $T_0'(t)$ supplied by Lemma \ref{centered poly Voronoi} contain $v_0$ in their intersection, so since $v_0\neq v_1$ it is not contained in $Q_{x_3}$.  Lemma \ref{centered poly Voronoi} thus implies that $d(v_0,x_3) > d(v_0,x_0)$, and hence that $v_0 = V_{x_0}\cap V_{x_1}\cap V_{x_2}$.  

Since $T_0(t)$ is isosceles, the geodesic arc $\gamma$ from $x_1$ to $v_0$ intersects $T_0(t)\cap T_0'(t)$ at a right angle.  Since the geodesic arc $\gamma'$ from $v_1$ to $v_0$ also meets $T_0(t)\cap T_0'(t)$, their union is geodesic.  It follows that $d(x_1,v_1) = d(x_1,v_0) + d(v_0,v_1)$.  On the other hand, $d(x_2,v_1)$ satisfies 
$$\cosh d(x_2,v_1) = \cosh d(x_2,v_0)\cosh d(v_0,v_1) = \cosh d(x_1,v_0)\cosh d(v_0,v_1)$$
by the hyperbolic law of cosines.  The angle addition formula for hyperbolic cosine therefore implies that $d(x_1,v_1) > d(x_2,v_1)$, so $v_1 = V_{x_0}\cap V_{x_2}\cap V_{x_3}$, and it follows that $V_{x_0}\cap V_{x_2}$ contains an edge joining $v_0$ to $v_1$.  The lemma follows.
\end{proof}

\begin{remark}  The construction of Example \ref{dull DeBlois} may be modified, by increasing $d_0$ and reducing $\delta$, to produce deformations of $F_{\beta}$ in which an edge of the Delaunay tessellation does not intersect its geometric dual at all. \end{remark}

\section{The centered dual to the Voronoi tessellation}\label{dirichlet dual}

Our task in this section is to understand the ``pathology'' described in Lemma \ref{dull DeBlois Delaunay}, in which an edge of $V$ does not intersect the interior of its geometric dual.  We will say that such an edge of $V$ is ``non-centered,'' and relate (non-)centeredness of edges to (non-)centeredness of vertex polygons in Lemma \ref{vertex polygon centered}.  The set of non-centered edges has restricted combinatorics: its components are sub-trees of $P^{(1)}$, each with a canonical root vertex (Lemma \ref{tree components}).  We organize the Delaunay polygons corresponding to vertices of such a component into a $2$-cell of the ``centered dual decomposition'' $P_c$, in Definition \ref{centered dual}.

\begin{definition}\label{centered edge}  Let $V$ be the Voronoi tessellation determined by $\cals\subset\mathbb{H}^2$ closed and discrete.  We will say an edge $e$ of $V$ is \textit{centered} if $e$ intersects its geometric dual $\gamma_{xy}$ at a point in $\mathit{int}\,e$.  If $e$ is not centered, orient it pointing away from $\gamma_{xy}$. 

If $V$ is the Voronoi tessellation of a closed surface $F$ determined by a finite set $\cals$, we say an edge $e$ of $V$ is \textit{centered} if and only if one (and hence all) of its lifts to $\widetilde{V}$ is centered, where $\widetilde{V}=\pi^{-1}(V)\subset\mathbb{H}^2$ is the Voronoi tessellation of $\mathbb{H}^2$ determined by $\widetilde{\cals} = \pi^{-1}(\widetilde{\cals})$.  If $e$ is not centered, let it inherit an orientation from a lift $\tilde{e}$.\end{definition}

As indicated above, the action of $\pi_1 F$ on $\widetilde{V}^{(1)}$ preserves (non-)centeredness of edges, and also the orientation of non-centered edges.

\begin{fact}  Let $V$ be the Voronoi tessellation determined by $\cals \subset\mathbb{H}^2$ closed and discrete.  For $x\in\cals$, an edge $e$ of $V_x$ is non-centered with initial vertex $v$  if and only if the angle $\alpha$ at $v$, measured in $V_x$ between $e$ and the geodesic segment joining $v$ to $x$, is at least $\pi/2$.\end{fact}

This is because there is a right triangle with vertices at $x$ and $v$ and edges contained in $\gamma_{xy}$ and $\gamma_{xy}^{\perp}$, where $\gamma_{xy}$ is the geometric dual to $e$; ie, $e = V_x\cap V_y$.  This triangle has angle equal to either $\alpha$ or $\pi-\alpha$ at $v$, depending on the case above; see Figure \ref{non-centered edge}.

\begin{figure}
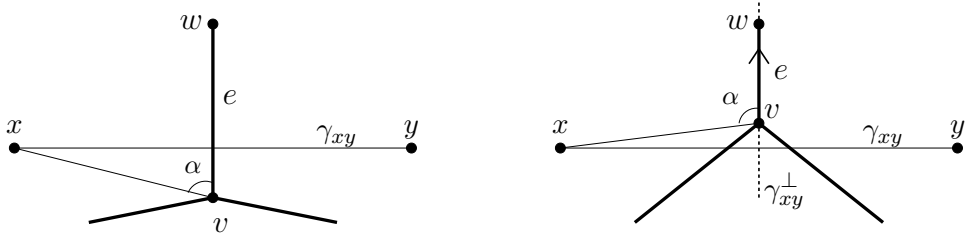
\caption{Centered and non-centered edges.}
\label{non-centered edge}
\end{figure}

If $w$ is the other endpoint of $e$ then since $x\in P_v\cap P_w$, the fact above and the hyperbolic law of cosines imply:\begin{align}\label{radius relation}  \cosh J_w = \cosh\ell(e)\cosh J_v - \sinh\ell(e)\sinh J_v\cos\alpha \end{align}
Because $\cos\alpha \leq 0$ if $\alpha \geq \pi/2$, we have:

\begin{lemma}\label{increasing radius} Let $V$ be the Voronoi tessellation determined by $\cals\subset\mathbb{H}^2$ closed and discrete.  For $x\in\cals$, if $e$ is a non-centered edge of $V_x$ oriented as prescribed in Definition \ref{centered edge}, with initial vertex $v$ and terminal vertex $w$, then $J_{v} < J_{w}$.\end{lemma}

Below we relate centeredness of edges of $V$ to that of $2$-cells of the Delaunay tessellation.

\begin{lemma}\label{vertex polygon centered} Let $V$ be the Voronoi tessellation determined by $\cals\subset\mathbb{H}^2$ closed and discrete.  For $v\in V^{(0)}$, $P_v$ is non-centered if and only if $v$ is the initial vertex of a non-centered edge $e$ of $V$.  If this is so, the geometric dual $\gamma$ to $e$ is the unique longest edge of $P_v$, and $P\cup T(e,v)$ is a convex polygon, where $T(e,v)$ is the triangle determined by $v$ and $\partial\gamma$.\end{lemma}

\begin{proof}  Suppose first that $v$ is the initial vertex of a non-centered edge $e = V_x\cap V_y$, and let $\calh'$ be the half-space containing $e$ and bounded by the geodesic containing $x$ and $y$.  The circle $C$ with radius $J_v$ and center $v$ intersects $\partial\calh'$ in $\{x,y\}$.  If $\alpha$ is the angle at $v$ between $e$ and the geodesic arc to $x$, then by the Fact above, $\alpha \geq \pi/2$.  The hyperbolic law of cosines implies that $z\in C$ is in the interior of $\calh'$ if and only if the angle $\alpha'$ at $v$ between $e$ and the geodesic arc to $z$ is less than $\alpha$.  Thus if $w$ is the other endpoint of $e$ for such $z$:
$$\cosh d(z,w) = \cosh\ell(e)\cosh J_v - \sinh\ell(e)\sinh J_v\cos\alpha' $$
Since $\alpha' < \alpha$, comparing with (\ref{radius relation}) we find that $d(z,w) < J_w$, so the intersection of $C$ with the interior of $\calh$ is entirely contained in $B_{J_w}(w)$.  Therefore by Lemma \ref{vertex radius} it contains no points of $\cals$.  Since all vertices of $P_v$ are on $C$, it follows that $P_v$ is contained in the half-plane $\calh$ opposite $\calh'$, and hence that $v\notin\mathit{int}\,P_v$.  Thus $P_v$ is non-centered by \CenteredPoly.

Assume now that $P_v$ is not centered and apply \OneSide.  This produces an edge $\gamma$ of $P_v$ and a half-space $\calh$ containing $P_v$ and bounded by the geodesic containing $\gamma$, such that $v$ is in the half-space $\calh'$ opposite $\calh$.  \OneSide\ further asserts that $P\cup T(e,v)$ is a convex polygon; also, $\gamma$ is the unique longest edge of $P_v$, by \LongestSide.  We claim that the other endpoint $w$ of the geometric dual $e$ to $\gamma$ is further from $\calh$ than $v$, and hence that $e$ is non-centered with initial vertex $v$.

If $w$ is closer to $\calh$ than $v$ (this includes the possibility $w\in\calh$), then $e$ intersects the geodesic joining $v$ to $x$ in an angle of $\alpha\leq\pi/2$.  $J_w$ again satisfies (\ref{radius relation}), and if $C$ is the circle of radius $J_v$ centered at $v$, the hyperbolic law of cosines again implies that $z\in C$ is in $\mathit{int}\,\calh$ if and only if the angle at $v$ between $e$ and the geodesic joining $v$ to $z$ is less than $\alpha$.  As in the previous case, this implies that the distance from the other vertices of $P_v$ to $w$ is less than $J_w$, contradicting Lemma \ref{vertex radius}.  Therefore $w$ is further from $\calh$ than $v$.\end{proof}

If $v\in V^{(0)}$ is the initial vertex of a non-centered edge $e$, the fact that the geometric dual to $e$ is the \textit{unique} longest edge of $P_v$ immediately implies the following.

\begin{corollary}\label{one direction}  Let $V$ be the Voronoi tessellation determined by $\cals \subset\mathbb{H}^2$ closed and discrete.  No $v\in V^{(0)}$ is the initial vertex of more than one non-centered edge.  \end{corollary}

\begin{definition}  If $V$ is the Voronoi tessellation determined by $\cals\subset\mathbb{H}^2$ closed and discrete, let $V^{(1)}_n\subset V^{(1)}$ be the union of the non-centered edges.  If $F$ is a closed surface, we define $V^{(1)}_n$ in the same way for the Voronoi tessellation $V$ determined by $\cals\subset F$ finite.\end{definition}

Below, given a graph $G$ we will say that $\gamma = e_0\cup e_1\cup \hdots\cup e_{n-1}$ is an \textit{edge path} if $e_i$ is an edge of $G$ for each $i$ and $e_i\cap e_{i-1}\neq \emptyset$ for $i>0$.  An edge path $\gamma$ as above is \textit{reduced} if $e_i\neq e_{i-1}$ for each $i>0$, and $\gamma$ is \textit{closed} if $e_0\cap e_{n-1} \neq\emptyset$.

\begin{lemma}\label{tree components}  Let $\widetilde{V}$ be the Voronoi tessellation determined by $\widetilde{\cals}\subset\mathbb{H}^2$ closed and discrete.  Each component $T$ of $\widetilde{V}^{(1)}_n$ is a tree.  If $\widetilde{\cals} = \pi^{-1}(\cals)$, where $\pi\co\mathbb{H}^2\to F$ is the universal cover to a closed surface $F$, and $\cals\subset F$ is finite, then $T$ is finite, with a unique vertex $v_T$ such that $J_{v_T} > J_v$ for all $v\in T^{(0)}-\{v_T\}$, and $T$ projects homeomorphically to $F$.  \end{lemma}

\begin{proof}  Suppose that a component $T$ of $\widetilde{V}^{(1)}_n$ admits closed, reduced edge paths, and let $\gamma = e_0\cup e_1\cup \hdots \cup e_{n-1}$ be a shortest such.  Orienting the $e_i$ as in Definition \ref{centered edge}, we may assume (after re-numbering if necessary) that $e_0$ points toward $e_0 \cap e_{n-1}$.  We claim that then $e_i$ points to $e_i\cap e_{i-1}$ for each $i>0$ as well.
Otherwise, for the minimal $i>0$ such that $e_i$ points toward $e_{i+1}$ it would follow that the vertex $e_i \cap e_{i-1}$ was the initial vertex of both $e_i$ and $e_{i-1}$, contradicting Corollary \ref{one direction}.

Let $v_0 = e_0\cap e_{n-1}\in V^{(0)}$, and for $i>1$ take $v_i = e_i\cap e_{i-1}$.  Applying Lemma \ref{increasing radius} to $e_i$ for each $i$, we find that $J_{v_i}> J_{v_{i+1}}$.  By induction this gives $J_{v_0} > J_{v_{n-1}}$; but since $e_{n-1}$ points to $v_{n-1}$ Lemma \ref{increasing radius} implies that $J_{v_{n-1}}$ must exceed $J_{v_0}$, a contradiction.  Thus no component of $\widetilde{V}^{(1)}_n$ admits closed, reduced edge paths, so each is a tree.

If $\widetilde{\cals} = \pi^{-1}(\cals)$ as in the hypotheses then the set $\{J_v\,|\,v\in T^{(0)}\}$ has only finitely many distinct elements, since $J_v = J_{v'}$ if $v$ and $v'$ project to the same point of $\cals$.  Thus take $v_T\in T^{(0)}$ with $J_{v_T}$ maximal.  We claim that $J_v < J_{v_T}$ for each $v\in T^{(0)}-\{v_T\}$.

If there exists $v\in T^{(0)}-\{v_T\}$ with $J_{v} = J_{v_T}$, let $\gamma = e_0\cup \hdots \cup e_{n-1}$ be a reduced edge path joining $v_T$ to $v$.  We may assume that $v_T$ is the endpoint of $e_0$ not in $e_1$, and $v\in e_{n-1} - e_{n-2}$.  Lemma \ref{increasing radius} implies that $e_0$ points toward $v_T$ and $e_{n-1}$ towards $v$.  Thus if $i>0$ is minimal such that $e_i$ does not point toward $e_{i-1}$, $v_i = e_i\cap e_{i-1}$ is the initial endpoint of $e_i$ and $e_{i-1}$, contradicting Corollary \ref{one direction}.  This proves the claim.

Since covering transformations exchange components of $\widetilde{V}^{(1)}_n$, if $\gamma.T\cap T \neq \emptyset$ for some $\gamma\in\pi_1 F - \{1\}$ then $\gamma.T = T$.  Since $J_{\gamma.v} = J_v$ for each $v\in T^{(0)}$, the claim above would imply that $\gamma.v_T = v_T$ for such $\gamma$, contradicting freeness of the $\pi_1 F$-action.  Therefore $T$ does not intersect its $\pi_1 F$-translates and thus projects homeomorphically to $F$.  It follows that $T^{(0)}$, and hence also $T$, is finite.\end{proof}

It is a basic fact that any two distinct points in a tree are joined by a unique reduced edge path, and that each such path is homeomorphic to an embedded interval.

\begin{lemma}\label{to the root}  Let $V$ be the Voronoi tessellation of a closed hyperbolic surface $F$ determined by $\cals\subset F$ finite, and let $T$ be a component of $V^{(1)}_n$.  For $v_T\in T^{(0)}$ as in Lemma \ref{tree components}, $P_{v_T}$ is centered.  For $v\in T^{(0)} - \{v_T\}$, $P_v$ is not centered, and the reduced edge path joining $v$ to $v_T$ inherits an orientation from each of its constituent edges, pointing from $v$ to $v_T$.  \end{lemma}

\begin{proof}  Since $v_T$ has maximal radius among $v\in T^{(0)}$, Lemma \ref{increasing radius} implies that it is the terminal endpoint of each edge of $T$ containing it.  Since every other edge of $V$ containing $v_T$ is centered, Lemma \ref{vertex polygon centered} implies that $P_{v_T}$ is centered.

For $v\in T^{(0)}-\{v_T\}$, let $e_0\cup\hdots\cup e_{n-1}$ be the reduced edge path joining $v$ to $v_T$, take $v_i = e_i\cap e_{i-1}$ for $1\leq i \leq n-1$, and let $v_0$ and $v_n$ be the endpoints of $e_0$ and $e_{n-1}$ not equal to $v_1$ and $v_{n-1}$, respectively.  Re-numbering if necessary, we may assume that $v_0 = v$ and $v_n = v_T$.  Then $e_{n-1}$ has terminal endpoint $v_n$.  If $e_i$ does not point toward $v_{i+1}$ for some $i < n-1$ then for the maximal such $i$, $v_i$ is the initial vertex of $e_i$ and $e_{i+1}$, contradicting Corollary \ref{one direction}.  The $e_i$ thus agree with the orientation on the edge path that points toward $v_T$.  In particular, $v$ is the initial vertex of $e_0$, so $P_v$ is not centered by Lemma \ref{vertex polygon centered}.  \end{proof}

\begin{definition}\label{centered dual graph}  Let $V$ be the Voronoi tessellation of a closed hyperbolic surface $F$ determined by $\cals\subset F$ finite.  Define the \textit{centered dual graph} $P^{(1)}_c$ to $V^{(1)}$ as: 
$$P^{(1)}_c = \bigcup\ \{\gamma_{xy}\,|\,\gamma_{xy}\ \mbox{is the geometric dual to a centered edge}\ e \subset V^{(1)} \}\subset P^{(1)}$$
Let $\widetilde{P}^{(1)}_c\subset\widetilde{P}^{(1)}$ be the preimage of $P^{(1)}_c$ in the universal cover.\end{definition}

It is easy to see that $\widetilde{P}^{(1)}_c$ is the union of geometric duals to centered edges of $\widetilde{V}$ (cf.~Definition \ref{centered edge}).  The centered dual graph has the structure of a subgraph of the one-skeleton $P^{(1)}$ of the Delaunay tessellation.  It exhibits the behavior that one expects from a dual:

\begin{lemma}\label{duality}  Let $V$ be the Voronoi tessellation of a closed hyperbolic surface $F$ determined by a finite set $\cals\subset F$.  If $e$ is a centered edge of $V$ and $\gamma \subset P^{(1)}_c$ its geometric dual, then $\gamma\cap V^{(1)} = \gamma\cap e$ is a single point.  Furthermore, $e\cap P^{(1)}_c = e\cap\gamma$. \end{lemma}

\begin{proof}  Let $\tilde{e}$ be a lift of $e$ to $\mathbb{H}^2$, and let $x$ and $y\in\widetilde{\cals}$ be such that $\tilde{e} = V_x\cap V_y$.  Then the geodesic arc $\gamma_{xy}$ joining $x$ to $y$ projects to $\gamma$.  Let $p = e\cap\gamma_{xy}$, and let $[x,p]$ and $[y,p]$ be the sub-arcs of $\gamma_{xy}$ joining $x$ and $y$, respectively, to $p$.  Since $V_x$ is convex with $x$ in its interior, it contains $[x,p]$, and $[x,p] \cap\partial V_x = \{p\}$.  The analogous assertion holds for $V_y$ and $[y,p]$, and so $\gamma_{xy} \cap\widetilde{V}^{(1)} = \{p\}$.  Since this holds for any lift of $e$, the first claim follows. 

The second claim follows from the first, since the set of centered edges of $V^{(1)}$ is in bijective correspondence with the edge set of $P^{(1)}_c$ by associating a centered edge to its dual.
\end{proof}

Figure \ref{non-centered edge} shows that the conclusion of Lemma \ref{duality} does not hold for non-centered edges.

Lemma \ref{tree components} implies that each cell $V_x$ of $V$ has at least one centered edge $e$, for otherwise some component of $V^{(1)}_n$ would contain the closed loop $\partial V_x$.  Since the geometric dual of such an edge $e$ is of the form $\gamma_{xy}$, it follows that the vertex set of $P^{(1)}_c$ is all of $\cals$.

The interior of each Voronoi cell $V_x$ is isometric to the interior of a compact, convex polygon in $\mathbb{H}^2$.  Therefore there is a ``geometric'' deformation retract $V_x -\{x\} \to \partial V_x \subset V^{(1)}$ along geodesic arcs connecting $x$ to points on $\partial V_x$.  Since $P^{(1)}_c$ intersects each $V_x$ in a collection of such arcs, we have:

\begin{lemma}  Let $V$ be the Voronoi tessellation of a closed hyperbolic surface $F$ determined by $\cals \subset F$ finite, and let $P^{(1)}_c$ be the centered dual graph to $V^{(1)}$.  There is a deformation retract $\rho_{\cals}\co F - P^{(1)}_c \to V^{(1)} - (V^{(1)}\cap P^{(1)}_c)$ such that for each $x\in\cals$, $\rho_{\cals}$ restricts on $V_x - P^{(1)}_c$ to the restriction of the corresponding geometric deformation retract.  \end{lemma}

Since $\rho_{\cals}$ is a deformation retract, it determines a one-to-one correspondence between the set of components of $F-P^{(1)}_c$ and the set of components of $V^{(1)} - (V^{(1)}\cap\Gamma)$.  We use this to give an initial description of the components of $F - \Gamma$.  It will be convenient to first introduce another definition.

\begin{definition}\label{frontier}  If $V$ is a graph and $T$ a subgraph, we define the \textit{frontier} $\calf$ of $T$ in $V$ to be the set of pairs $(e,v)$ where $e$ is an edge of $V$ that is not in $T$ and $v\in e \cap T$.\end{definition}

We may refer just to ``an edge'' of the frontier of $T$, without reference to its vertices, but note that $\calf$ has two elements for each $e$ not in $T$ with both endpoints in $T$.

\begin{lemma}\label{centered or not}  Let $V$ be the Voronoi tessellation of a closed hyperbolic surface $F$ determined by $\cals\subset F$ finite, and let $P^{(1)}_c$ be the centered dual graph to $V^{(1)}$.  Each component $U$ of $F- P^{(1)}_c$ is homeomorphic to an open disk, and either:\begin{enumerate}
\item  there is a unique $v\in V^{(0)}\cap U$, each edge of $V$ containing $v$ is centered, and the universal cover maps $P_{\tilde{v}}$ to $\overline{U}$ for (any) $\tilde{v}\in\pi^{-1}(V)$; or
\item  $U$ contains a unique component $T$ of $V^{(1)}_n$, and $V^{(0)} \cap U = T^{(0)}$.  
\end{enumerate}\end{lemma}

\begin{proof}  Suppose for $v\in V^{(0)}\cap U$ that every edge of $V$ containing $v$ is centered.  The same holds for $\tilde{v}\in\pi^{-1}(v)$, so Lemma \ref{vertex polygon centered} implies that the vertex polygon $P_{\tilde{v}}$ is centered, and hence that $\tilde{v}\in\mathit{int}\,P_{v}$.  By Lemma \ref{vertex polygon}, each edge of $P_{\tilde{v}}$ is the geometric dual to an edge containing $\tilde{v}$.  Thus $\partial P_{\tilde{v}}\subset \widetilde{P}^{(1)}_c$, so $\mathit{int}\,P_{\tilde{v}}$ is a component of $\mathbb{H}^2-\widetilde{P}^{(1)}_c$ containing $\tilde{v}$.  Since $\pi$ maps $\partial P_{\tilde{v}}$ into $P^{(1)}_c$ and $\mathit{int}\,P_{\tilde{v}}$ homeomorphically, $U = \pi(\mathit{int}\,P_{\tilde{v}})$.

In particular, since $\mathit{int}\,P_v$ is homeomorphic to an open disk, the same holds for $U$.  Let the edges of $\widetilde{V}$ containing $\tilde{v}$ by cyclically ordered $e_0,\hdots,e_{n-1}$ as in Lemma \ref{vertex polygon}, with geometric dual $\gamma_i$ for each $i$, and let $m_i = e_i\cap\gamma_i\in \mathit{int}\,e_i$.  For each the quadrilateral $Q_{x_i}$ constructed in Lemma \ref{centered poly Voronoi} is contained in $V_{x_i}$ since its vertices are (cf.~Definition \ref{cyclically ordered V}), so the conclusion there implies that $P_{\tilde{v}}\subset \bigcup_{i=0}^{n-1} Q_{x_i} \subset \bigcup_{i=0}^{n-1} V_{x_i}$.  Since $Q_{x_i}\cap\partial V_{x_i} = (e_i\cap P_{\tilde{v}})\cup (e_{i+1}\cap P_{\tilde{v}})$ it follows that $v = V^{(0)}\cap U$.

If $v\in V^{(0)}\cap U$ is a vertex of a non-centered edge, then $U$ contains the entire component $T$ of $V^{(1)}_n$ containing $v$, since $T$ does not intersect $P^{(1)}_c$.  Let $\widetilde{T}$ be a lift of $T$ to the universal cover.  For each $(e,v)$ in the frontier $\calf$ of $\widetilde{T}$, $e$ is centered, so its geometric dual $\gamma$ intersects it and lies in $\widetilde{P}^{(1)}_c$.  For $v\in e\cap \widetilde{T}$, let $[v,e\cap\gamma)$ refer to the component of $e-\gamma$ containing $v$. Then
$$ S_{\widetilde{T}} = T \cup \left(\ \bigcup\, \left\{[v,e\cap\gamma)\,|\,(e,v)\in\calf\right\}\, \right)  $$
is a connected open subset of $\widetilde{V}^{(1)} - \widetilde{P}^{(1)}_c$ with frontier in $\widetilde{P^{(1)}_c}$, and hence is a component of $\widetilde{V}^{(1)} - \widetilde{P}^{(1)}_c$.  Again we find that by construction, $S_{\widetilde{T}}$ deformation retracts to $\widetilde{T}$ (thus in particular is simply connected) and that $S_{\widetilde{T}} \cap \widetilde{V}^{(0)} = \widetilde{T}^{(0)}$.  Therefore $S_{\widetilde{T}}$ projects homeomorphically to the component $S_T$ of $V^{(1)} - P^{(1)}_c$ containing $T$.

Since $S_T$ is a component of $V^{(1)}-P^{(1)}_c$ contained in $U$, and $U$ is a component of $F-P^{(1)}_c$, $U = \rho_{\cals}^{-1}(S_T)$; in particular, $T^{(0)} = S_T\cap V^{(0)} = U\cap V^{(0)}$.  Since $\rho_{\cals}$ is a deformation retract, $U$ is simply connected and hence lifts homeomorphically to the component $\widetilde{U}$ of $\mathbb{H}^2-\widetilde{P}^{(1)}_c$ containing $S_{\widetilde{T}}$.  This is homeomorphic to a disk by, say, the Riemann mapping theorem, and therefore so is $U$.\end{proof}

To better understand the structure of complementary components to $P^{(1)}_c$ that contain points of $V^{(1)}_n$, we introduce a new tool.

\begin{definition}\label{edge-vertex tri dfn}  Let $V$ be the Voronoi tessellation of $\mathbb{H}^2$ determined by $\cals\subset F$ closed and discrete.  For $v\in V^{(0)}$, an edge $e$ of $V$ containing $v$, and $x,y\in \cals$ such that $e = V_x \cap V_y$, let $T(e,v)$ be the isosceles triangle with vertices $v$, $x$ and $y$.  

If $V$ is the Voronoi tessellation of a closed hyperbolic surface $F$ determined by $\cals\subset F$ finite, let $T(e,v) = \pi(T(e,\tilde{v}))$, where $\tilde{v}\in\pi^{-1}(v)$.  \end{definition}

The edges of $T(e,v)$ that join $v$ to $x$ and $y$, respectively, each have length $J_v$, and the third edge of $(e,v)$ is the geometric dual $\gamma_{xy}$ to $e$.  If $v$ and $w$ are opposite endpoints of $e$, then $T(e,v)$ and $T(e,w)$ share the edge $\gamma_{xy}$.  Whether their intersection is larger than this depends on whether $e$ is centered --- see Figure \ref{isosceles dichotomy}.  In particular:

\begin{figure}
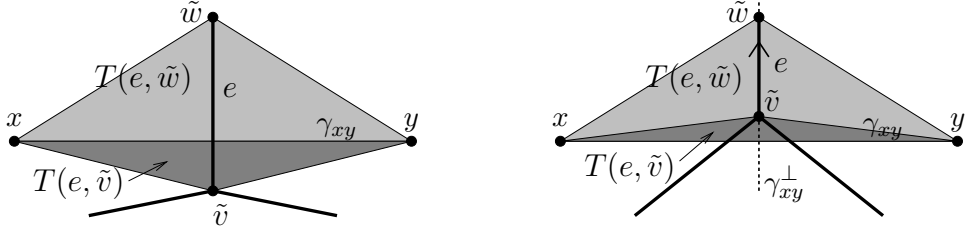
\caption{Triangles $T(e,\tilde{v})$ and $T(e,\tilde{w})$ when $e$ is centered (on the left) and not centered.}
\label{isosceles dichotomy}
\end{figure}

\begin{lemma}\label{edge-vertex tri}  Let $V$ be the Voronoi tessellation determined by $\cals \subset\mathbb{H}^2$ closed and discrete.  If $e$ is a non-centered edge of $V$ with initial vertex $v$ and terminal vertex $w$, then $T(e,v) \subset T(e,w)$ and $T(e,v)\cap \partial T(e,w)$ is the geometric dual of $e$.  \end{lemma}

\begin{proof}  Since $e$ is non-centered it is contained on one side of the geodesic in $\mathbb{H}^2$ containing its geometric dual $\gamma$, so the nearer vertex $v$ on $e$ to $\gamma$ is in the interior of $T(e,w)$.  The result now follows from convexity.\end{proof}

\begin{lemma}\label{vertex polygon decomp}  Let $V$ be the Voronoi tessellation determined by $\cals\subset\mathbb{H}^2$ closed and discrete.  For $v\in V^{(0)}$:\begin{enumerate} 
\item  If $P_v$ is centered then $P_{v} = \bigcup\, \{ T(e,v)\,|\,v\in e\}$.
\item  Otherwise, $P_v\cap T(e_v,v) = \gamma_v$ and $P_{v}\cup T(e_v,v) = \bigcup\,\{ T(e',v)\,|\, e\neq e', v\in e'\}$, where $e_v$ is the non-centered edge of $V$ with initial vertex $v$ and $\gamma_v$ is its geometric dual.\end{enumerate}
Let $\gamma$ and $\gamma'$ be the respective geometric duals to $e$ and $e'\ne e$ containing $v$.  In case (1), $T(e,v)\cap T(e',v) = \{v\}$ if $\gamma\cap\gamma'=\emptyset$, and otherwise is an edge joining $v$ to $\gamma\cap\gamma'$.  This holds in case (2) for $e,e'\neq e_v$.\end{lemma}

\begin{proof}  If the edges of $P_v$ are enumerated $\gamma_0,\hdots,\gamma_i$, for each $i$ the triangle $T_i$ described in the hypothesis of \IsoscelesDecomp\ is identical to $T(e_i,v)$, where $e_i$ is the geometric dual to $\gamma_i$.  If $P_v$ is non-centered then Lemma \ref{vertex polygon centered} implies that $\gamma_v$ as defined above is its unique longest edge, so \LongestSide\ and \OneSide\ imply that $T(e_v,v)\cap P_v = \gamma_v$.  The decompositions of $P_v$ and $P_v\cup T(e,v)$ described above follow directly from \IsoscelesDecomp.\end{proof}

\begin{proposition}\label{dual decomp}  Let $V$ be the Voronoi decomposition of a closed hyperbolic surface $F$ determined by $\cals \subset F$ finite, and let $P^{(1)}_c$ be the centered dual graph to $V^{(1)}$.  For each component $U$ of $F-P^{(1)}_c$, $\overline{U} = \bigcup_{v\in Q\cap V^{(0)}} P_v$.  \end{proposition}

\begin{proof}  In case (1) of Lemma \ref{centered or not}, the desired conclusion is proved there, so we will assume that $U$ contains a component $T$ of $V^{(1)}_n$.  Let $\widetilde{T}$ be a lift of $T$ to $\mathbb{H}^2$, let $\widetilde{U}$ be the component of $\mathbb{H}^2 - \widetilde{\Gamma}$ containing $\widetilde{T}$, and let $\widetilde{Q}$ be the closure of $\widetilde{U}$ in $\mathbb{H}^2$.  Let $\calf$ be the frontier of $\widetilde{T}$ in $\widetilde{V}^{(1)}$.  For $(e,v)\in\calf$, we claim that $T(e,v)$ is contained in $\widetilde{Q}$.

Let $x$ and $y\in \widetilde{\cals}$ be such that $e = V_x \cap V_y$.  Then the side of $T(e,v)$ opposite $v$ is $\gamma_{xy} \subset \widetilde{P}^{(1)}_c$, since $e$ is centered, and $T(e,v)$ is equal to the union of its intersections with $V_x$ and $V_y$.  Let $[v,\gamma_{xy}]$ be the sub-arc of $e$ running from $v$ to $e\cap\gamma_{xy}$.  Lemma \ref{duality} implies that $[v,\gamma_{xy}] \cap \widetilde{P}^{(1)}_c = \{e\cap\gamma_{xy}\}$, so since $v\in\widetilde{T}\subset\widetilde{U}$ it follows that $[v,\gamma_{xy}]\subset \widetilde{Q}$.

$T(e,v)\cap V_x$ is the union of geodesic arcs in $V_x$ joining $x$ to points on $[v,\gamma_{xy}]$.  Since $\widetilde{P}^{(1)}_c\cap V_x$ is a union of geodesic arcs joining $x$ to points of $\widetilde{P}^{(1)}_c\cap\partial V_x$, and the only such point in $[v,\gamma_{xy}]$ is $e\cap\gamma_{xy}$, it follows that $T(e,v) \cap V_x$ intersects $\widetilde{P}^{(1)}_c$ only in $\gamma_{xy} \cap V_x$.  Since $\widetilde{U}$ contains a neighborhood of $v$ and has its frontier in $\widetilde{P}^{(1)}_c$, this implies that $T(e,v) \cap V_x\subset\widetilde{Q}$.  The analogous argument gives the analogous result for $T(e,v)\cap V_y$, and the claim is proved.

Lemma \ref{vertex polygon decomp} implies that the interior of $P_{v}$ intersects that of $T(e,v)$, so the claim above implies that the interior of $P_{v}$ intersects $\widetilde{U}$.  By Lemma \ref{vertex polygon}, the interior of $P_{v}$ is a component of $\mathbb{H}^2-\widetilde{P}^{(1)}$.  Therefore since $\widetilde{P}^{(1)}_c \subset \widetilde{P}^{(1)}$, $\mathrm{int}\,P_{v} \subset \widetilde{U}$, and hence $P_{v} \subset \widetilde{Q}$.

If $v$ and $w$ are adjacent in $\widetilde{T}$, then the corresponding vertex polygons $P_{v}$ and $P_{w}$ share an edge $\gamma$ of $\widetilde{P}$, the geometric dual to the edge of $T$ joining $v$ to $w$.  Since this edge is non-centered, $\gamma$ intersects $\widetilde{P}^{(1)}_c$ only at its endpoints.  Since $\mathit{int}\,P_{v}$ and $\mathit{int}\,P_{w}$ are each components of $\mathbb{H}^2 - \widetilde{P}^{(1)}_c$, it follows that $P_{v}\subset \widetilde{Q}$ if and only if $P_{w}\subset\widetilde{Q}$. 

We have already proved that $P_v\subset\widetilde{Q}$ for any $v\in\widetilde{T}^{(0)}$ such that $(e,v)\in\calf$ for some edge $e$, so since $\widetilde{T}$ is connected, the previous paragraph and an inductive argument show that $\bigcup_{v\in \widetilde{T}^{(0)}} P_{v}\subset \widetilde{Q}$.  Projecting to $F$ it follows that $\bigcup_{v\in T^{(0)}} P_v\subset \overline{U}$.

It remains to show that $\bigcup_{v\in T^{(0)}} P_{\tilde{v}}$ is not properly contained in $\overline{U}$.  If it were, then there would exist $v' \notin T^{(0)}$ such that $P_{v'} \subset \overline{U}$.  But then Lemma \ref{centered or not} implies that $v'$ is contained in a different component $U'$ of $F-P^{(1)}_c$, so by the above $P_{v'}\subset\overline{U}'$.  But since $P_{v'}$ has non-empty interior, this is a contradiction.\end{proof}

\begin{corollary}\label{centered dual polys}  Let $V$ be the Voronoi decomposition of a closed hyperbolic surface $F$ determined by $\cals \subset F$ finite, and let $P^{(1)}_c$ be the centered dual graph to $V^{(1)}$.  For each component $U$ of $F-P^{(1)}_c$, the completion $\bar{Q}$ of the induced path metric on $U$ is homeomorphic to a closed disk.  If $U$ contains a component $T$ of $V^{(1)}_n$ with frontier $\calf$ in $V^{(1)}$, then:
$$ \bar{Q} - U = \bigcup\ \{\gamma_{(e,v)}\,|\,(e,v)\in\calf\}, $$
where $\gamma_{(e,v)}$ is isometric to the geometric dual to $e$ for each $(e,v)\in\calf$.\end{corollary}

A brief proof sketch: $\bar{Q}$ is homeomorphic to the complement in $U$ of a small neighborhood of its frontier, itself a closed disk.  If $e$ is an edge of $V$ that is not in $T$ but has both endpoints in it, then its geometric dual $\gamma$ contributes two edges to $\bar{Q}$ --- one for each side --- but only one to the closure $Q$ of $U$.  This is why we use the induced path metric.  It holds even in $\mathbb{H}^2$: a lift of $U$ to $\widetilde{U}\subset\mathbb{H}^2-\widetilde{P}^{(1)}_c$ determines a map from $\bar{Q}$ to the closure $\widetilde{Q}$ of $\widetilde{U}$ that is two-to-one over each edge of $\widetilde{P}^{(1)}_c$ geometrically dual to a lift of $e$ as above, and injective elsewhere.

\begin{definition}\label{centered dual}  Let $V$ be the Voronoi tessellation of a closed hyperbolic surface $F$ determined by $\cals\subset F$ finite.  Define the \textit{centered dual decomposition} $P_c$ of $F$ to be the cell complex with $P^{(0)}_c = \cals$, $P^{(1)}_c$ as described in Definition \ref{centered dual graph}, and $2$-cells as in Corollary \ref{centered dual polys}.

If $U$ is a component of $F-P^{(1)}_c$ containing a component $T$ of $V^{(1)}_n$, we will refer to its closure $Q$ as a $2$-cell of $P_c$ with \textit{vertex set} $Q\cap\cals = \bigcup_{v\in T^{(0}} P_v\cap T^{(0)}$, and \textit{edge set} $Q - U = \bigcup_{e\in\calf} \{\gamma\ \mbox{geometrically dual to}\ e\}$, where $\calf$ is the frontier of $T$ in $V^{(1)}$ and ``$e\in\calf$'' means that $e$ has an endpoint $v$ such that $(e,v)\in\calf$.\end{definition}

\section{The centered dual versus a disk packing}\label{centered packing}

We show in this section that the centered dual decomposition determined by $\cals$ interacts well with a set of disjoint open hyperbolic disks of equal radius isometrically embedded about the points of $\cals$.  Recall from the beginning of \cite[\S 5]{DeB_cyclic_geom} that a polygon $P$ determines a \textit{sector} of a disk $U$ centered at one of its vertices $x$, with angle measure equal to $\angle_x P$, and that this sector contains $P\cap U$.  Figure \ref{more bad packing} illustrates an instance in which containment is proper, with the ``bad'' region shaded.

\begin{figure}
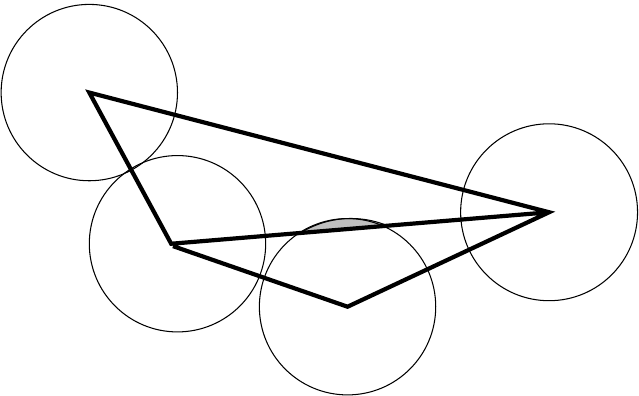
\caption{When a triangle determines a sector that it does not entirely contain.}
\label{more bad packing}
\end{figure}

The \textit{radius-$R$ defect} $D_R(d_0,\hdots,d_{n-1})$, as defined in \Defect, describes the area of the region of a cyclic $n$-gon $P$ represented by $(d_0,\hdots,d_{n-1})\in\widetilde{\calAC}_n$ outside the union of a collection of disjoint radius-$R$ disks centered at its vertices, if each disk intersects $P$ in a full sector. If $P$ is centered then by \PackingVsDecomp, the full sectors hypothesis holds, and by \DefectFunction\ the area in question is $D_R(P)$.  For non-centered cyclic polygons the pathology of Figure \ref{more bad packing} may occur, but we show here that it does not for centered dual $2$-cells.

\begin{definition}\label{inj rad}  For $\cals\subset F$ finite, where $F$ is a closed hyperbolic surface, define the \textit{injectivity radius $i(\cals)$ of $F$ at $\cals$} to be the injectivity radius of the preimage $\widetilde{\cals}\subset\mathbb{H}^2$, as defined above Lemma \ref{Voronoi poly}, of $\cals$ under the universal cover $\mathbb{H}^2\to F$.\end{definition}

It is easy to see that $i(\cals)$ is the maximal $R$ such that a collection of open, radius-$R$ hyperbolic disks may be isometrically embedded in $F$ without overlapping, centered at the points of $\cals$.  In particular, if $\cals = \{x\}$ is a singleton, then $i(\cals)$ is the usual injectivity radius of $F$ at $x$.

\begin{proposition}\label{2-cell defect}  Let $V$ be the Voronoi tessellation and $P_c$ the centered dual decomposition determined by $\cals\subset F$ finite, where $F$ is a closed hyperbolic surface.  If $\{U_x\}$ is a set of open hyperbolic disks of radius $R\leq i(\cals)$ centered at the points of $\cals$, then for a $2$-cell $Q$ of $P_c$:
$$ \mathrm{area}\left(Q- \bigcup_{x\in\cals} (U_x\cap Q)\right) = D_R(Q) \doteq \sum_{v\in Q\cap V^{(0)}} D_R(\bd_v), $$
where $\bd_v$ represents $P_v$ in $\widetilde{\calAC}_{n_v}$ for each $v$ (with $n_v$ the valence of  $v$ in $V^{(1)}$).\end{proposition}

As we remarked above, this does not necessarily hold for Delaunay $2$-cells that are non-centered; however, for those that are it follows directly from \DefectFunction.  Lemma \ref{centered or not} implies that each $2$-cell of $P_c$ is either a centered Delaunay polygon or contains a component $T$ of $V^{(1)}_n$, so it is this latter case that we will address in the remainder of the section.

A centered dual $2$-cell is by definition equal to the union of vertex polygons $P_v$ for $v\in Q\cap V^{(0)}$.  It will be convenient for our purposes to re-tile $Q$ by a new set of ``polygons.''

\begin{definition}\label{vertex poly prime} Let $\widetilde{V}$ be the Voronoi tessellation determined by $\widetilde{\cals}\subset \mathbb{H}^2$ closed and discrete, and let $T$ be a component of $\widetilde{V}^{(1)}_n$.  For $v\in T^{(0)} - \{v_T\}$, let $e_v$ be the edge of $T$ with initial vertex $v$.  For $v\in T^{(0)}$, define $v+1$ to be the set of $w\in T^{(0)}$ such that $v$ is the terminal vertex of $e_w$ (oriented as in Definition \ref{centered edge}). 

Define $P_{v_T}' = \overline{P_v - \left(\bigcup_{w\in v+1} T(e_w,w)\right)}$, where $T(e_w,w)$ is as in Definition \ref{edge-vertex tri dfn}, and if $v \ne v_T$ let $P_v' = \overline{(P_v\cup T(e_v,v)) - \left(\bigcup_{w\in v+1} T(e_w,w)\right)}$ (here the overline denotes the closure in $\mathbb{H}^2$).\end{definition}

Although $P_v'$ is not necessarily convex, its angle at a vertex $x$ of $P_v$ is clearly at most that of $P_v\cup T(e_v,v)$, so since $P_v\cup T(e_v,v)$ is convex (cf.~Lemma \ref{vertex polygon centered}) it makes sense to talk about ``the sector determined by $P_v'$'' of a disk $U$ centered at $x$.  The key advantage of the $P_v'$ is that they behave well with respect to such disks.

\begin{lemma}\label{star sector} Let $V$ be the Voronoi tessellation and $P_c$ the centered dual decomposition determined by $\widetilde{\cals} = \pi^{-1}(\cals)$, where $\pi\co\mathbb{H}^2\to F$ is the universal cover of a closed hyperbolic surface and $\cals\subset F$ is finite.  Fix a component $T$ of $V^{(1)}_n$, $v\in T^{(0)}$, and $R\leq i(\cals)$.  A disk $U_x$ of radius $R$ centered at a vertex $x$ of $P$ intersects $P_v'$ in the sector determined by $P_v'$.  For $w\in v+1$, $U_x\cap T(e_w,w)\neq\emptyset$ if and only if $x$ is in the geometric dual $\gamma_w$ to $e_w$.\end{lemma}

Let us recall that for $v$ as above and $w\in v+1$, $T(e_w,w)\subset T(e_w,v)$ and $T(e_w,v)\cap \partial T(e_w,v)$ is the edge $\gamma_w$ geometrically dual to $e_w$, by Lemma \ref{edge-vertex tri}.  Thus Lemma \ref{vertex polygon decomp} implies that $T(e_w,w)$ is entirely contained in $P_v\cup T(e_v,v)$, for $v\neq v_T$, or in $P_v$ if $v = v_T$; and furthermore that $T(e_w,w)\cap \partial (P_v\cup T(e_v,v)) = \gamma_w$ (or that $T(e_w,w)\cap \partial P_v = \gamma_w$ if $v=v_T$).

\begin{proof}  For now take $v = v_T$.  Lemma \ref{vertex polygon decomp} implies that a vertex $x$ of $P$ is contained in $T(e,v)$, for some edge $e$ containing $v$, if and only if $x$ is an endpoint of the edge $\gamma$ of $P_v$ that is geometrically dual to $e$.  Thus a small-enough disk $U$ around $x$ has the property that $U\cap P_v = (U\cap T(e,v))\cup (U\cap T(e',v))$, where $e$ and $e'$ are the edges containing $v$ with geometric duals $\gamma$ and $\gamma'$ meeting at $x$ (this also uses Lemma \ref{vertex polygon decomp}).  

For $U_x$ as described above, \FullSector\ implies that $U_x\cap T(e,v)$ is the sector determined by $T(e,v)$, and likewise for $T(e',v)$.  Since $U_x\cap P_v$ is contained in the sector determined by $P_v$, and this is the union of those determined by $T(e,v)$ and $T(e',v)$ by the above, it follows that $U_x\cap P_v = (U_x\cap T(e,v))\cup (U_x\cap T(e',v))$.  For $w\in v+1$ such that $e_w \neq e,e'$, since $T(e_w,w)\subset T(e_w,v)$ it follows that $T(e_w,w)\cap U_x = \emptyset$.  On the other hand, if $e_w = e $ (say), then $\gamma_w = \gamma$ and $U_x$ clearly intersects $T(e_w,w)$, in the sector that it determines (by \FullSector\ again).  The final assertion of the lemma follows.

Since $v = v_T$, the definition of $P_v'$ implies that $U_x\cap P_v' = \overline{(U_x\cap P_v) - (\bigcup_{w\in v+1} (U_x\cap T(e_w,w)))}$.  By the above, $U_x\cap P_v$ is a sector, and for each $w\in v+1$, $U_x\cap T(e_w,w)$ is empty unless $x\in \gamma_w$, in which case it is a sector containing the boundary edge $U_x\cap \gamma_w$ of $U_x\cap P_v$.  It thus easily follows from the description above that $U_x\cap P_v'$ is also a sector.

We have proved the lemma for $v = v_T$.  The proof for $v\in T^{(0)}-\{v_T\}$ is similar, but with two important differences.  First, $P_v$ should be replaced above by $P_v\cup T(e_v,v)$, and second, the case of $x\in \partial \gamma_v$ must be treated separately.  For such $x$ it turns out that $U_x\cap (P_v\cup T(e_v,v)) = U_x\cap T(e',v)$, where $e'$ is the geometric dual to the edge of $P_v$ meeting $\gamma_v$ at $x$.  This follows from \PackingVsDecomp, which implies that $\angle_x (P_v\cup T(e_v,v)) = \angle_x T(e',v)$ after chasing through some definitions.\end{proof}

\begin{fact}  For each $v\in T^{(0)}$ and $w\in v+1$, $P_v'\cap T(e_w,w)$ is the union of edges of $T(e_w,w)$ containing $w$.  For $v\in T^{(0)} - \{v_T\}$, if $\gamma_v$ is the geometric dual to $e_v$ then: 
$$\mathit{int}\, P_v' = \left(\mathit{int}\, P_v \cup \mathit{int}\,T(e_v,v) \cup \mathit{int}\,\gamma_v\right) -  \left(\bigcup_{w\in v_T+1} T(e_w,w)\right)$$
Similarly, $\mathit{int}\, P_{v_T}' = \mathit{int}\, P_{v_T} - \left(\bigcup_{w\in v_T+1} T(e_w,w)\right)$.\end{fact}

\begin{lemma}\label{re-tile}  Let $V$ be the Voronoi tessellation and $P_c$ the centered dual decomposition determined by $\widetilde{\cals} = \pi^{-1}(\cals)$, where $\pi\co\mathbb{H}^2\to F$ is the universal cover of a closed hyperbolic surface and $\cals\subset F$ is finite.  If $Q$ is a $2$-cell of $P_c$ containing a component $T$ of $V^{(1)}_n$ then $Q = \bigcup_{v\in T^{(0)}} P_v'$, for $P_v'$ as in Definition \ref{vertex poly prime}.  For distinct $v$ and $w$ in $T^{(0)}$, $\mathit{int}\,P_v' \cap P_w'=\emptyset$.\end{lemma}

\begin{proof}  For $x\in Q$ let $v_0$ be such that $x\in P_{v_0}$.  If $x\notin T(e_v,v)$ for any $v\in v_0+1$, then $x\in P_{v_0}'$.  If there exists $v_1 \in v_0+1$ such that $x\in T(e_{v_1},v_1)$, we choose $v\in T^{(0)}$ to satisfy three criteria: $x\in T(e_v,v)$, the reduced edge path from $v_0$ to $v$ contains $v_1$, and this edge path is longest among all of those joining $v_0$ to vertices satisfying the first two criteria.  Since $v_1$ satisfies the first two criteria there is some such $v$, and by construction $x\in T(e_v,v)$ but not in $T(e_w,w)$ for any $w\in v+1$ (the reduced edge path from $v_0$ to such a $w$ is the union of the reduced edge path to $v$ with $e_w$).  Therefore $x\in P_v'$.  

Our goal in the remainder is to prove that $\mathit{int}\,P_v' \cap P_w'=\emptyset$ for distinct $v$ and $w$ in $T^{(0)}$.  For $v\in T^{(0)}-\{v_T\}$, let $e_0\cup \hdots\cup e_{n-1}$ be the reduced edge path in $T$ joining $v_T$ to $v$, numbered so that $v_T$ is a vertex of $e_0$ and $v$ is a vertex of $e_{n-1}$.   Upon orienting this path and its edges as described in Lemma \ref{to the root}, $v_T$ is the terminal vertex of $e_0$ and $v$ the initial vertex of $e_{n-1}$.  For $0\leq i\leq n-1$ let $v_i$ be the terminal vertex of $e_i$, so in particular $v_T = v_0$, and let $v_n = v$.  Then $v_i$ is the initial vertex of $e_{i-1}$ for $i >0$, so  for each $i<n$, $e_i = e_{v_{i+1}}$ as defined in \ref{vertex poly prime}.  

\begin{claim}\label{T vs P} For $0<i\leq n$, $T(e_{i-1},v_i)\subset \bigcup_{j=0}^{i-1} P_{v_j}$; and if $\gamma_j$ is the geometric dual to $e_j$ for each $j$, we have $T(e_{i-1},v_i)\cap\partial P_{v_0}\subset \gamma_0$, and $T(e_{i-1},v_i)\cap \partial P_{v_j} \subset \gamma_{j-1}\cup \gamma_{j}$ for $0< j\leq i-1$.\end{claim}

First take $i=1$.  Lemma \ref{edge-vertex tri} implies that $T(e_0,v_1)\subset T(e_0,v_0)$, and Lemma \ref{vertex polygon decomp} implies that $T(e_0,v_0)\subset P_{v_0}$.  (Recall from Lemma \ref{to the root} that $P_{v_T} = P_{v_0}$ is centered.)  By Definition \ref{edge-vertex tri dfn}, $\gamma_0$ is the edge of $T(e_0,v_0)$ opposite $v_0$.  Any other edge $\gamma'$ of $P_{v_0}$ is contained in $T(e',v_0)$, where $e'$ is the geometric dual to $\gamma'$, and Lemma \ref{vertex polygon decomp} implies that $T(e_0,v_0)\cap \gamma' \subset T(e_0,v_0)\cap T(e',v_0)$ is at most the endpoint $\gamma_0\cap\gamma'$.  Thus $T(e_0,v_0)\cap\partial P_{v_0}\subset \gamma_0$.

For $1 < i\leq n$, the combination of Lemmas \ref{edge-vertex tri} and \ref{vertex polygon decomp} implies:\begin{align}\label{combo}
  T(e_{i-1},v_i)\subset T(e_{i-1},v_{i-1})\subset P_{v_{i-1}}\cup T(e_{i-2},v_{i-1})\end{align}
We assume by induction that $T(e_{i-2},v_{i-1})\subset \bigcup_{j=0}^{i-2} P_{v_j}$, so by the above $T(e_{i-1},v_i)\subset\bigcup_{j=0}^{i-1} P_{v_j}$.  

To prove the claim it remains to show that $T(e_{i-1},v_i)$ has reasonable intersections with the $P_{v_j}$.  We assume by induction that $T(e_{i-2},v_{i-1})\cap \partial P_{v_0}\subset\gamma_0$ and $T(e_{i-2},v_{i-1})\cap\partial P_{v_j}\subset \gamma_{j-1}\cup\gamma_j$ for $0<j\leq i-2$.  Thus by (\ref{combo}) it suffices to show that $T(e_{i-1},v_{i-1})\cap\partial P_{v_{i-1}}\subset \gamma_{i-2}\cup\gamma_{i-1}$.  This follows as in the base case, but using the non-centered case of Lemma \ref{vertex polygon decomp}.

Below we will obtain different information from essentially the same sequence of observations.

\begin{claim}\label{T vs P cup T}  For $n >1$ and $0< j <i \leq n$, $T(e_{i-1},v_i) \subset T(e_{j-1},v_j)\cup \left(\bigcup_{k=j}^{i-1} P_{v_k}\right)$.\end{claim}

We again prove the claim by induction, this time on $n$.  For $n=2$ the only case above is with $i = n = 2$ and $j =1$.  Applying (\ref{combo}) immediately implies the conclusion in this case.  Let us now take $n > 2$ and suppose that the claim holds for $n-1$.  The only new cases to consider have $i = n$, since for $i\leq n-1$ the conclusion follows from the induction hypothesis.

Fixing $i = n$, the conclusion for $j = n-1$ is a direct application of (\ref{combo}).  For $j < n-1$, (\ref{combo}) gives $T(e_{n-1},v_n)\subset T(e_{n-2},v_{n-1}) \cup P_{v_{n-1}}$, so induction produces:\begin{align*}
   T(e_{n-1},v_n) \subset T(e_{j-1},v_j)\cup \left(\bigcup_{k=j}^{n-2} P_{v_k}\right)\cup P_{v_{n-1}} = T(e_{j-1},v_j)\cup \left(\bigcup_{k=j}^{n-1} P_{v_k}\right)\end{align*}
This proves the claim.

We first fix $w = v_T$ and $v\in T^{(0)} - \{v_T\}$, and prove $\mathit{int}\, P_v'\cap P_w = P_v' \cap \mathit{int}\,P_w = \emptyset$.  Let $e_0\cup\hdots e_{n-1}$ be the reduced edge path joining $v_T$ to $v$ in $T$, numbered and oriented as above, so that $w = v_T$ is the terminal vertex of $e_0$ and $v$ is the initial vertex of $e_{n-1}$.  With the $v_j$ also numbered as above for $0\leq j \leq n$, we apply Claim \ref{T vs P cup T} with $i = n$ and $j = 1$.  Taking a union with $P_v$ on each side of the result yields:\begin{align}\label{i equals n}
  P_v\cup T(e_{n-1},v) \subset T(e_0,v_1)\cup \left(\bigcup_{k=1}^{n} P_{v_k}\right)\end{align}
This is because $v = v_n$.  Since $v_1\in v_T + 1$ and $e_0 = e_{v_1}$, the fact above the lemma implies that $\mathit{int}\, P_{v_T}' \cap T(e_0,v_1) = \emptyset$.  By Lemma \ref{vertex polygon}, $\mathit{int}\,P_{v_T}\cap P_{v_k} = \emptyset$ for $1\leq k \leq n$.  Since $P_v' \subset P_v\cup T(e_{n-1},v)$, it follows from (\ref{i equals n}) that $P_v'\cap\mathit{int}\, P_{v_T}' = \emptyset$.

Since $P_v'\subset P_v\cup T(e_{n-1},v)$, $\mathit{int}\,P_v'\subset \mathit{int}\left(P_v\cup T(e_{n-1},v)\right) = \mathit{int}\, P_v\cup\mathit{int}\, T(e_{n-1},v)\cup\mathit{int}\, \gamma_{n-1}$.  The latter equality here follows from Lemma \ref{vertex polygon decomp}, which asserts that $T(e_{n -1},v)$ intersects $P_v$ in precisely the edge $\gamma_{n-1}$ geometrically dual to $e_{n-1}$.  Similarly:

\begin{claim}\label{union int} For $0< i\leq n$, 
$$ \mathit{int}\left( P_v\cup T(e_{n-1},v_n)\right) \subset \mathit{int}\,T(e_{i-1},v_i)\cup\left(\bigcup_{k=i}^n \mathit{int}\,P_{v_k}\right)\cup\left(\bigcup_{k=i-1}^{n-1}\mathit{int}\,\gamma_k\right),$$
where $\gamma_k$ is the geometric dual to $e_k$ for each $k$.\end{claim}

This uses Claim \ref{T vs P cup T}, also noting that $T(e_{i-1},v)\cap \partial P_{v_k}\subset \gamma_{k-1}\cup \gamma_k$ for $i\leq k< n$, by Claim \ref{T vs P}.  Thus by the fact above the lemma, $\mathit{int}\, P_v'\cap P_{v_T}' = \emptyset$.

For distinct $v$ and $w$ in $T^{(0)} - \{v_T\}$, let $e_0\cup\hdots\cup e_{n-1}$ and $f_0\cup\hdots f_{m-1}$ be reduced edge paths joining $v_T$ to $v$ and $w$, respectively.  Let the $e_i$ be numbered and oriented as in the case $w= v_T$ above, and number the vertices $v_i$ accordingly.  Similarly, number the $f_i$ so that $v_T$ is an endpoint of $f_0$ and $w$ an endpoint of $f_{m-1}$, and orient them as in Lemma \ref{to the root}.  Let $w_i$ be the terminal vertex of $f_i$ for each $i$, so $v_T = w_0$ in particular, and take $w = w_m$. 

Because $T$ is a tree, there are three cases to consider: either these paths have no edges in common,  meeting only at $v_T$; or one is an initial segment of the other; or they share an initial segment that is proper in each.  We will address the third case in detail.  The others are similar, and we will indicate afterwards how to approach them.

If $e_0\cup\hdots\cup e_{n-1}$ and $f_0\cup\hdots\cup f_{m-1}$ share an initial segment that is proper in each, then $m$ and $n$ are each at least $2$.  Let $i_0 >0$ but less than $\min\{m,n\}$ be such that $e_i = f_i$ for $i< i_0$ but $e_{i_0}\neq f_{i_0}$.  It follows that $v_i = w_i$ for $i\leq i_0$, but that $\{v_{i_0+1},\hdots,v_n\}$ is disjoint from $\{w_{i_0+1},\hdots,w_m\}$, since $T$ is a tree.

Applying Claim \ref{T vs P cup T} to $e_0\cup\hdots\cup e_{n-1}$ with $i = n$ and $j = i_0+1$, then taking the union with $P_v = P_{v_n}$ on both sides, we have:\begin{align}\label{i equals n part deux}
  P_v\cup T(e_{n-1},v) \subset T(e_{i_0},v_{i_0+1})\cup \left(\bigcup_{k=i_0+1}^{n} P_{v_k}\right)\end{align}
Using Claim \ref{T vs P cup T} on $f_0\cup\hdots\cup f_{m-1}$ with $i = m$ and $j = i_0+1$, arguing as above yields:\begin{align}\label{j equals m part un}
  P_w \cup T(f_{m-1},w)\subset T(f_{i_0},w_{i_0+1}) \cup \left(\ \bigcup_{k=i_0+1}^{m} P_{w_k}\right)\end{align}
We will use Claim \ref{union int} with $i=i_0+1$ to show that $\mathit{int}(P_v\cup T(e_{n-1},v_n))$ is disjoint from $P_w\cup T(f_{m-1},w)$, from which it immediately follows that $\mathit{int}\,P_v' \cap P_w'=\emptyset$.  

Lemma \ref{vertex polygon} implies that $\mathit{int}\,P_{v_k}\cap P_{w_l} = \emptyset$ for each $k\in \{i_0+1,\hdots,n\}$ and $l\in\{i_0+1,\hdots,m\}$, and also that $\mathit{int}\,\gamma_k\cap P_{w_l}=\emptyset$ for $k\in\{i_0,\hdots,n-1\}$ and the same $l$.  This is because $\gamma_k = P_{v_k}\cap P_{v_{k+1}}$ for each such $k$, and so its interior is disjoint from all $P_v$ but these two.  In the particular case $k = i_0$, $\gamma_{i_0}$ is a different edge of $P_{v_{i_0}}$ than its edge of intersection with $P_{w_{i_0+1}}$, since $w_{i_0+1} \neq v_{i_0+1}$, so it is still true that $\mathit{int}\,\gamma_{i_0}\cap P_{w_{i_0+1}} = \emptyset$.

Lemma \ref{edge-vertex tri} implies that $T(e_{i_0},v_{i_0+1})\subset T(e_{i_0},v_{i_0})$ and $T(f_{i_0},w_{i_0+1}) = \subset T(f_{i_0},v_{i_0})$ (recall that $w_{i_0}=v_{i_0}$), and by Lemma \ref{vertex polygon decomp} these are each contained in $P_{v_{i_0}}\cup T(e_{i_0-1},v_{i_0})$.  It further implies that $T(e_{i_0},v_{i_0})\cap T(f_{i_0},v_{i_0})$ is at most an edge of each containing $v_{i_0}$; hence in particular $\mathit{int}\, T(e_{i_0},v_{i_0+1})\cap T(f_{i_0},w_{i_0+1})=\emptyset$.

We will finally show that $\mathit{int}\,T(e_{i_0},v_{i_0+1})\cap P_{w_l}=\emptyset$ for each $l\in\{i_0+1,\hdots,m\}$.  Claim \ref{T vs P} implies that $T(e_{i_0},v_{i_0+1})\subset \bigcup_{j=0}^{i_0} P_{v_j}$, and the second part of that claim implies that its interior is contained in $\left(\,\bigcup_{j=0}^{i_0} \mathit{int}\,P_{v_j}\right)\cup\left(\,\bigcup_{j=0}^{i_0-1}\mathit{int}\,\gamma_j\right)$.  It thus follows as above that $\mathit{int}\,T(e_{i_0},v_{i_0+1})\cap P_{w_l}=\emptyset$ for each $l$ under consideration.  This completes the proof that $\mathit{int}\, P_v'\cap P_w' = \emptyset$ when the edge paths joining each to $v_T$ share a proper initial segment.

The case when the edge paths meet at $v_T$ and nowhere else is similar to the above but with $i_0 =0$. Then $T(e_0,v_1)$ and $T(f_0,w_1)$ are each contained in $P_{v_T}$, and it is immediate that $\mathit{int}\,T(e_0,v_1)$ is disjoint from $P_{w_l}$ for each $l>0$.  The case when $e_0\cup\hdots\cup e_{n-1}$ (say) is a proper initial segment of $f_0\cup\hdots\cup f_{m-1}$ is similar to the case $v = v_T$ that we first addressed.  An extra argument is required in this case to show that $T(e_v,v) = T(e_{n-1},v_n)$ has interior disjoint from the $P_{w_l}$; this proceeds as in the paragraph directly above.
\end{proof}

\begin{definition} Let $V$ be the Voronoi tessellation and $P_c$ the centered dual decomposition determined by $\widetilde{\cals} = \pi^{-1}(\cals)$, where $\pi\co\mathbb{H}^2\to F$ is the universal cover of a closed hyperbolic surface and $\cals\subset F$ is finite, and fix a $2$-cell $Q$ of $P_c$ containing a component $T$ of $V^{(1)}_n$.  Define $T^{(0)}_0 = \{v_T\}$ and $Q_0 = P_{v_T}$, and for $k > 0$ let $T^{(0)}_k$ consist of vertices of $T$ joined to $v_T$ by a path of at most $k$ edges, and $Q_k' = \bigcup_{v\in T^{(0)}_k} P_v'$.

If $x$ is a vertex of $Q$, for $k\geq 0$ define the \textit{restriction} of $U_x$ to $Q_k'$ as the union of sectors:
$$U_x|_{Q_k'} \doteq\bigcup\ \{U_x\cap P_v'\,|\,v\in T^{(0)}_k, x\in P_v\}$$
Let the restriction $U_x|_Q$ of $U_x$ to $Q$ be the union of its restrictions to $Q_k'$ over all $k \geq 0$.\end{definition}

The point of defining the restriction is to exclude an incidental component of intersection with $Q_k'$ as in Figure \ref{more bad packing}, where a disk $U_x$ protrudes from a polygon that does not entirely contain the sector that it determines, intersecting one that does not have $x$ as a vertex.

\begin{lemma}\label{restriction defect} Let $V$ be the Voronoi tessellation and $P_c$ the centered dual decomposition determined by $\widetilde{\cals} = \pi^{-1}(\cals)$, where $\pi\co\mathbb{H}^2\to F$ is the universal cover of a closed hyperbolic surface and $\cals\subset F$ is finite.  If $\{U_x\}$ is a set of open hyperbolic disks of radius $R\leq i(\cals)$ centered at the vertices of a $2$-cell $Q$ of $P_c$ containing a component $T$ of $V^{(1)}_n$, for each $k\geq 0$:
$$ \mathrm{area}\left(Q_k'- \bigcup_{x\in\cals} U_x|_{Q_k'}\right) =\sum_{v\in T^{(0)}_k} D_R(\bd_v) -\sum_{w\in T^{(0)}_{k+1} - T^{(0)}_k} D_R(d_{e_w},J_w), $$
where $\bd_v$ represents $P_v$ in $\widetilde{\calAC}_{n_v}$ for each $v$ (with $n_v$ the valence of  $v$ in $V^{(1)}$), and $D_R(d_{e_w},J_w)$ is as in \FullSector\ (with $d_{e_w}$ the length of the geometric dual to $e_w$).\end{lemma}

\begin{proof}  We prove this by induction on $k$.  By definition, $P_{v_T} = P_{v_T}'\cup \left(\bigcup_{w\in v_T+1} T(e_w,w)\right)$.  For each vertex $x$ of $P_{v_T}$, $U_x\cap P_{v_T}$ is a sector (by \PackingVsDecomp) that is a non-overlapping union of sectors $(U_x\cap P_{v_T}')\cup (U_x\cap T(e_w,w))$ (by Lemma \ref{star sector} and \FullSector, respectively).  By Lemma \ref{star sector}, $T(e_w,w)$ does not intersect $U_x$ unless $x$ is one of its vertices.  Thus:
$$ \mathrm{area}\left(P_{v_T}'- \bigcup_{x\in\cals} U_x|_{Q_0'}\right) = D_R(\bd_{v_T}) -\sum_{w\in v_T+1} D_R(d_{e_w},J_w), $$
This is because \DefectFunction\ implies that $D_R(\bd_{v_T})$ is the area of the region in $P_{v_T}$ outside the $U_x$, and \FullSector\ implies that $D_R(d_{e_w},J_w)$ is the area of the region of $T(e_w,w)$ outside the disks centered at its endpoint.  Since $T^{(0)} = \{v_T\}$ and $v_T+1=T^{(0)}_1 - T^{(0)}_0$, the $k=0$ case follows.

For $k>0$, we note that for any $v\in T^{(0)}_k$, $\mathit{int} P_v'$ intersects $U_x|_{Q_k'}$ if and only if $x$ is a vertex of $P_v$.  This by definition of the $U_x|_{Q_k'}$, since by Lemma \ref{star sector} for a vertex $x'$ of $P_w'$ for some $w\neq v$, $P_w'$ contains the sector of $U_{x'}$ that it determines, and $\mathit{int}\, P_v' \cap P_w' = \emptyset$ by Lemma \ref{re-tile}. 

For $k>0$, assume the conclusion holds for $k-1$.  Writing $Q_k' = Q_{k-1}'\cup\left(\bigcup_{v\in T^{(0)}_k-T^{(0)}_{k-1}} P_v'\right)$ yields:\begin{align*}
  \mathrm{area}\left(Q_k'- \bigcup_{x\in\cals} U_x|_{Q_k'}\right) & =\sum_{v\in T^{(0)}_{k-1}} D_R(\bd_v) -\sum_{w\in T^{(0)}_k-T^{(0)}_{k-1}} D_R(d_{e_w},J_w), \\
  &\qquad + \sum_{v\in T^{(0)}_k - T^{(0)}_{k-1}} \left(D_R(\bd_v) + D_R(d_{e_v},J_v) - \sum_{w\in v+1} D_R(d_{e_w},J_w)\right) \end{align*}
The first line above follows from the inductive hypothesis, and the second by an argument analogous to the base case.  The sum above telescopes, and since $T^{(0)}_{k+1} - T^{(0)}_k = \bigcup\,\left\{v+1\,|\,v\in T^{(0)}_k - T^{(0)}_{k-1}\right\}$, the lemma follows by induction.\end{proof}

The main result of the section follows quickly.

\begin{proof}[Proof of Proposition \ref{2-cell defect}]  The result follows from Lemma \ref{restriction defect} and two observations.  

First, for a $2$-cell $Q$ of $P_c$ containing a component $T$ of $V^{(1)}_n$, let $\widetilde{Q}$ be a lift to $\mathbb{H}^2$ and $\widetilde{T}$ the lift of $T$ that it contains.  Because $\widetilde{T}$ is finite, there exists $k >0$ such that $\widetilde{T} = \widetilde{T}^{(0)}_k$, and hence $\widetilde{Q} = \widetilde{Q}_k'$ (by Lemma \ref{re-tile}) and $\widetilde{T}^{(0)}_{k+1}-\widetilde{T}^{(0)}_k = \emptyset$.  Thus Lemma \ref{restriction defect} implies:
$$ \mathrm{area}\left(\widetilde{Q}- \bigcup_{x\in\widetilde{\cals}} U_x|_{\widetilde{Q}}\right) =\sum_{v\in \widetilde{T}^{(0)}} D_R(\bd_v) $$

The first observation above, combined with Lemma \ref{star sector}, implies for each $2$-cell $\widetilde{Q}$ of $\widetilde{P}_c$ that contains a component $\widetilde{T}$ of $\widetilde{V}^1_n$, and each $x\in\widetilde{\cals}$ that is the vertex of $P_v$ for some $v\in\widetilde{T}^{(0)}$, that $\widetilde{Q}$ contains the union of sectors of $U_x$ determined by the $P_v$ over all $v\in T^{(0)}$.  The same holds for $2$-cells $\widetilde{Q}$ that are centered polygons, by \PackingVsDecomp.  It follows that  for each $x\in\widetilde{\cals}$, $U_x \subset \bigcup_{i=1}^{n} Q_i$, where $\{Q_i\}_{i=1}^{n}$ is the collection of $2$-cells of $\widetilde{P}_c$ containing $x$.  This in turn implies the second observation: that $U_x\cap\widetilde{Q} = U_x|_{\widetilde{Q}}$ for each $x\in\widetilde{\cals}$.\end{proof}

\section{Admissible spaces}\label{moduli}

This section is devoted to abstracting the data provided by a $2$-cell $Q$ of the centered dual and lower bounds for its edge lengths, turning this into a parameter space and a function on it whose minimum bounds the defect below.  We will show that this defect function attains a minimum on the closure of the parameter space, and in the second half of the section restrict the location of such a minimum for low-complexity cells.

\begin{definition}\label{partial order}  Let $T\subset V$ be finite graphs such that $T$ is a rooted tree with root vertex $v_T$.    Partially order $T^{(0)}$ by setting $v < v_T$ for each $v\in T^{(0)} - \{v_T\}$ and $w < v$ if the edge arc in $T$ joining $w \in T^{(0)}- \{v_T,v\}$ to $v_T$ runs through $v$.  A vertex $v$ is \textit{minimal} if there is no $w \in T^{(0)}$ such that $w < v$; ie, so that $v+1$ (as in Defintion \ref{vertex poly prime}) is empty.  For $v\in T^{(0)}-\{v_T\}$, say ``$e\to v$'' for each edge $e\neq e_v$ of $V$ containing $v$, where $e_v$ is as in Definition \ref{vertex poly prime}.\end{definition}

\begin{definition}\label{admissible criteria}  Let $T\subset V$ be finite graphs such that $T$ is a rooted tree with root vertex $v_T$ and each vertex of $V$ has valence at least three.    Let $\mathcal{E}$ be the set of edges of $T$ and $\mathcal{F}$ the frontier of $T$ in $V$, fix an ordering on $\mathcal{E}\cup\calf$ and for some choice of $d_e > 0$ for each $e\in\calf$, let $\bd_{\calf} = (d_e)\in\mathbb{R}^{\calf}$.  For any $\bd_{\cale} = (d_e)\in \mathbb{R}^{\mathcal{E}}$, let $\bd = (\bd_{\cale},\bd_{\calf})$ and $P_{v}(\bd) = (d_{e_0},\hdots,d_{e_{n-1}})$ for $v\in T^{(0)}$, where the edges of $V$ containing $v$ are cyclically ordered as $e_0,\hdots,e_{n-1}$.  We will say $\bd_{\cale}$ is in the \textit{admissible set} $\mathit{Ad}(\bd_{\calf})$ determined by $\bd_{\calf}$ if the following criteria hold: \begin{enumerate}
\item\label{not centered} For $v\in T^{(0)}-\{v_T\}$ with valence $n_v$ in $V$, $P_v(\bd) \in \mathcal{AC}_{n_v} - \calc_{n_v}$ has largest entry $d_{e_v}$. 
\item\label{centered} $P_{v_T}(\bd) \in \calc_{n_T}$, where $v_T$ has valence $n_T$ in $V$.
\item\label{radius order} $J(P_{v}(\bd)) > J(P_w(\bd))$ for each $w\in v- 1$, where $J(P_v(\bd))$ and $J(P_w(\bd))$ are the respective radii of $P_v(\bd)$ and  $P_w(\bd)$.  \end{enumerate}  
(Note that the final condition above is vacuous for minimal $v$.)  \end{definition}

\begin{definition}\label{tree defect} Let $T \subset V$ be finite graphs such that $T$ is a rooted tree with root vertex $v_T$ and each vertex of $V$ has valence at least three.  Let $\mathcal{E}$ be the set of edges of $T$ and $\calf$ the frontier of $T$ in $V$, and fix $\bd_{\calf} = (d_e\,|\,e\in\calf)\in \mathbb{R}^{\calf}$ such that $\mathit{Ad}(\bd_{\calf})\neq \emptyset$.  For each $\bd_{\cale}\in \mathit{Ad}(\bd_{\calf})$, let $\bd = (\bd_{\cale},\bd_{\calf})$ and for $R\leq\min\{d_e/2\,|\,e\in\mathcal{F}\}$, define:
$$ D_R(T,\bd) = \sum_{v\in T^{(0)}} D_R(P_v(\bd)), $$
where $P_v(\bd)$ is as in Definition \ref{admissible criteria} and $D_R(P)$ is as defined in \DefectFunction.  \end{definition}

\begin{lemma}\label{admissible Voronoi}  Let $V$ be the Voronoi decomposition of a closed hyperbolic surface $F$ determined by a finite set $\cals\subset F$, and let $Q$ be a centered dual $2$-cell containing a component $T$ of $V^{(1)}_n$.  Let $\mathcal{E}$ be the edge set of $T$ and $\mathcal{F}$ its frontier in $V$, and for each $e\in\mathcal{E}$ or such that $(e,v)\in\mathcal{F}$ for some $v$, let $d_e$ be the length of the geometric dual to $e$.  Then $\bd_{\cale} \in \mathit{Ad}(\bd_{\calf})$ and $D_R(Q) = D_R(T,\bd)$ for $\bd = (\bd_{\cale},\bd_{\calf})$ and $R\leq \min\{d_e/2\,|\,e\in\mathcal{F}\}$.  \end{lemma}

\begin{proof}  For each $v\in T^{(0)}$, Lemma \ref{vertex polygon} implies that the vertex polygon $P_v$ is a cyclic polygon with cyclically ordered side length collection $(d_{e_0},\hdots,d_{e_{n-1}})$, where $e_0,\hdots,e_{n-1}$ is the cyclically ordered collection of edges of $V$ containing $v$. 
Recall from Definition \ref{centered edge} that each edge of $T$ is oriented.  Lemma \ref{tree components} asserts that the root vertex $v_T$ of $T$ is the terminal endpoint of every edge of $T$ that contains it.  Since every other edge of $V$ containing $v_T$ is centered, Lemma \ref{vertex polygon centered} implies that $P_{v_T} \in \calc_n$.  This establishes (\ref{centered}) from Definition \ref{admissible criteria}.  

Lemma \ref{vertex polygon centered} also implies that for $v\in T^{(0)} - \{v_T\}$, $P_v$ is non-centered with longest side length $d_{e_v}$, yielding (\ref{not centered}) from Definition \ref{admissible criteria}.  For $v\in T^{(0)}$ and $w \in v -1$, Lemma \ref{to the root} implies that $w$ is the initial vertex of $e_w$, and the definition (in \ref{partial order} above) implies that $v$ is its terminal vertex.  Therefore Lemma \ref{increasing radius} yields $J_v > J_w$ and hence, by Lemma \ref{vertex polygon}, that $J(P_v) > J(P_w)$.  This establishes (\ref{radius order}) from Definition \ref{admissible criteria}.  

That $D_R(Q) = D_R(T,\bd)$ is a direct consequence of Definition \ref{tree defect} and Proposition \ref{2-cell defect}.  \end{proof}

The lemma below implies in particular that the admissible set of $\bd_{\cale}$ is a bounded subset of $\mathbb{R}^{\calf}$, so that it has compact closure.

\begin{lemma}\label{precompact}  Let $T\subset V$ be finite graphs such that $T$ is a rooted tree with root vertex $v_T$ and each vertex of $V$ has valence at least three.    Let $\mathcal{E}$ be the set of edges of $T$ and $\mathcal{F}$ the frontier of $T$ in $V$.  There exist collections $\{b_e\co (\mathbb{R}^+)^{\calf}\to\mathbb{R}^+\,|\,e\in\cale\}$ and $\{h_e\co (\mathbb{R}^+)^{\calf}\to\mathbb{R}^+\,|\,e\in\cale\}$ characterized by the following properties:\begin{itemize}
  \item  $P_v(\bd)\in\widetilde{\calBC}_{n_v}$, with largest entry $d_{e_v}$, for each $v\in T^{(0)}-\{v_T\}$, where $v$ has valence $n_v$ in $V$ and $\bd = (\bd_{\cale},\bd_{\calf})$, if and only if $d_e = b_e(\bd_{\calf})$ for each $e\in\cale$.
  \item  $P_v(\bd)\in\widetilde{\calHC}_{n_v}$, with largest entry $d_{e_v}$, for each $v\in T^{(0)}-\{v_T\}$ if and only if $d_e = h_e(\bd_{\calf})$ for each $e\in\cale$.\end{itemize}
For fixed $\bd_{\calf}$, if $\bd_{\cale} = (d_e\,|\, e\in\cale) \in \mathit{Ad}(\bd_{\calf})$ then $b_e(\bd_{\calf}) \leq d_e < h_e(\bd_{\calf})$ for each $e\in\mathcal{E}$; and for $\bd_{\calf}' = (d_e'\,|\,e\in\calf)$ such that $d_e' \geq d_e$ for each $e\in\calf$, $b_e(\bd_{\calf}')\geq b_e(\bd_{\calf})$ for each $e\in\cale$.\end{lemma}

\begin{proof}  The proof is by induction, the key point being that for $v\in T^{(0)}-\{v_T\}$, $b_{e_v}(\bd_{\calf})$ is determined by the set of $b_{e_w}(\bd_{\calf})$ for $w<v$, and similarly for $h_{e_v}(\bd_{\calf})$.  Fix $\bd_{\calf}\in(\mathbb{R}^+)^{\calf}$.

Suppose first that $v\in T^{(0)}$ is minimal, so each $e\to v$ is in $\mathcal{F}$.  Cyclically enumerate the edges of $V$ containing $v$ as $e_0,\hdots,e_{n-1}$ so that $e_0 = e_v$, and for each $i >0$ let $d_i = d_{e_i}$.  Then \BCn\ implies that $b_{e_v}(\bd_{\calf})\doteq b_0(d_1,\hdots,d_{n-1})$ is the unique number such that $(b_{e_v}(\bd_{\calf}),d_1,\hdots,d_{n-1})\in\overline{U}_n\cap\widetilde{\calBC}_n$, where $U_n$ is as in \FundamentalDomain.  That is, $b_{e_v}(\bd_{\calf})$ is unique with the property that the tuple above is in $\widetilde{\calBC}_n$ and has its largest entry first.  \HCn\ implies the analogous conclusion for $h_{e_v}(\bd_{\calf})\doteq h_0(d_1,\hdots,d_{n-1})$ and $\widetilde{\calHC}_n$. 

Let us also note that by \VerticalRay, $(d_0,d_1,\hdots,d_{n-1})\in U_n$ is in $\widetilde{\calAC}_n - \widetilde{\calc}_n$ if and only if $b_{e_v}(\bd_{\calf})\leq d_0 < h_{e_v}(\bd_{\calf})$.  If $\bd_{\cale}\in \mathit{Ad}(\bd_{\calf})$ then for $\bd = (\bd_{\cale},\bd_{\calf})$ Definition \ref{admissible criteria} (\ref{not centered}) implies that $P_v(\bd) = (d_{e_v},d_1,\hdots,d_{n-1})\in U_n\cap (\widetilde{\calAC}_n - \widetilde{\calc}_n)$, so $b_{e_v}(\bd_{\calf})\leq d_{e_v} < h_{e_v}(\bd_{\calf})$.

Now fix $v\in T^{(0)}-\{v_T\}$ non-minimal, and suppose that we $b_{e_w}(\bd_{\calf})$ and $h_{e_w}(\bd_{\calf})$ are defined for each $w<v$, satisfying the following inductive hypotheses:\begin{itemize}
\item $P_w(\bd)\in\widetilde{\calBC}_{n_w}$, with largest entry $d_{e_w}$, for each $w<v$, where $\bd=(\bd_{\cale},\bd_{\calf})$, if and only if $d_{e_w} = b_{e_w}(\bd_{\calf})$ for each $w < v$.
\item $P_w(\bd)\in\widetilde{\calBC}_{n_w}$, with largest entry $d_{e_w}$, for each $w<v$, where $\bd=(\bd_{\cale},\bd_{\calf})$, if and only if $d_{e_w} = b_{e_w}(\bd_{\calf})$ for each $w < v$.
\item For each $\bd_{\cale}\in\mathit{Ad}(\bd_{\calf})$, $b_{e_w}(\bd_{\calf})\leq d_{e_w} < h_{e_w}(\bd_{\calf})$ for each $w < v$.\end{itemize}
Cyclically enumerate the edges containing $v$ as $e_0,\hdots,e_{n-1}$ so that $e_0 = e_v$, and for $i>0$ define:\begin{align*}
  & d_i = \left\{\begin{array}{ll} d_{e_i} & e_i\in\calf \\ b_{e_i}(\bd_{\calf}) & e_i\in\cale \end{array}\right. &
  & d_i' = \left\{\begin{array}{ll} d_{e_i} & e_i\in\calf \\ h_{e_i}(\bd_{\calf}) & e_i\in\cale\end{array}\right.\end{align*}
Then \BCn\ again implies that $b_{e_v}(\bd_{\calf})\doteq b_0(d_1,\hdots,d_{n-1})$ is the unique number such that $(b_{e_v}(\bd_{\calf}),d_1,\hdots,d_{n-1})\in\overline{U}_n\cap\widetilde{\calBC}_n$, and \HCn\ gives the analogous conclusion for $h_{e_v}(\bd_{\calf}) = h_0(d_1',\hdots,d_{n-1}')$.

For $\bd_{\cale}\in\bd_{\calf}$, by hypothesis $d_i \leq d_{e_i}$ for each $i\in\{1,\hdots,n-1\}$, so \BCn implies that $b_{e_v}(\bd_{\calf})\leq b_0(d_{e_1},\hdots,d_{e_{n-1}})$.  We also have $d_{e_i} < d_i$ for each $i$ such that $e_i\in\cale$ by hypothesis (and $d_{e_i}=d_i$ otherwise), so $h_{e_v}(\bd_{\calf}) > h_0(d_{e_1},\hdots,d_{e_{n-1}})$ by \HCn.  By Definition \ref{admissible criteria} (\ref{not centered}) and \VerticalRay,
$$ b_0(d_{e_1},\hdots,d_{e_{n-1}}) \leq d_{e_v} < h_0(d_{e_1},\hdots,d_{e_{n-1}}),$$
and it follows that $b_{e_v}(\bd_{\calf})\leq d_{e_v} < h_{e_v}(\bd_{\calf})$.  We have thus proved the three hypotheses above for $\{v\}\cup\{w<v\}$, so it follows by induction that they hold on all of $T^{(0)}-\{v_T\}$.  (Recall in particular that there is a unique $e_v$ for each $v\in T^{(0)}-\{v_T\}$, and that $\cale$ is the set of all such $e_v$.)

The final claim of the lemma, that $b_e$ is ``increasing'' in $\bd_{\calf}$ for each $e$, follows from an inductive argument and \BCn, which asserts that $b_0(d_1,\hdots,d_{n-1})\leq b_0(d_1',\hdots,d_{n-1}')$ when $d_i\leq d_i'$ for each $i$.\end{proof}

\begin{remark}\label{outside-in}  For any given tree $T$ with frontier $\calf$, the proof of Lemma \ref{precompact} is easily adapted (using formulas from \cite{DeB_cyclic_geom}) to produce a recursive algorithm that takes $\bd_{\calf}\in(\mathbb{R}^+)^{\calf}$ and computes the values $b_e(\bd_{\calf})$ or $h_e(\bd_{\calf})$ from the ``outside in.''\end{remark}

\begin{lemma}\label{admissible closure}  Let $T\subset V$ be finite graphs such that $T$ is a rooted tree with root vertex $v_T$ and each vertex of $V$ has valence at least three.    Let $\mathcal{E}$ be the set of edges of $T$ and $\mathcal{F}$ the frontier of $T$ in $V$, and fix $\bd_{\calf} = (d_e\,|\,e\in\calf)\in\mathbb{R}^{\calf}$.  If $\mathit{Ad}(\bd_{\calf})\neq \emptyset$, then for each $\bd_{\cale}$ in its closure $\overline{\mathit{Ad}}(\bd_{\calf})$, $\bd = (\bd_{\cale},\bd_{\calf})$ satisfies: \begin{enumerate}
\item\label{closure not centered} For $v\in T^{(0)}-\{v_T\}$ with valence $n_v$ in $V$, $P_v(\bd) \in\widetilde{\mathcal{AC}}_{n_v} - \widetilde{\calc}_{n_v}$ has largest entry $d_{e_v}$. 
\item\label{closure centered} $P_{v_T}(\bd) \in \widetilde{\calc}_{n_T}\cup \widetilde{\calBC}_{n_T}$, where $v_T$ has valence $n_T$ in $V$.
\item\label{closure radius order} $J(P_{v}(\bd)) \geq J(P_w(\bd))$ for each $w\in v- 1$, where $J(P_v(\bd))$ and $J(P_w(\bd))$ are the respective radii of $P_v(\bd)$ and  $P_w(\bd)$.  \end{enumerate}  
\end{lemma}

\begin{proof}  Lemma \ref{precompact} implies that $\mathit{Ad}(\bd_{\calf})$ is bounded in $\mathbb{R}^{\cale}$ and therefore has compact closure.  Since $P_{v_T}(\bd)\in\calc_n$ for each $\bd\in\mathit{Ad}(\bd_{\calf})\times\{\bd_{\calf}\}$, for $\bd$ in the closure it must be the case that $P_{v_T}(\bd)\in\overline{\calc}_n$, establishing (\ref{closure centered}).  \SmoothAnglesRadius\  implies that for each $v\in T^{(0)}$, $J(P_v(\bd))$ varies continuously with $\bd$ on $\mathit{Ad}(\bd_{\calf})\times\{\bd_{\calf}\}$.  Thus since $J(P_v(\bd)) > J(P_w(\bd))$ for each such $\bd$ and $w\in v-1$, $J(P_v(\bd))\geq J(P_w(\bd))$ for each $\bd\in\overline{\mathit{Ad}}(\bd_{\calf})\times\{\bd_{\calf}\}$, so (\ref{closure radius order}) holds.

Now suppose that (\ref{closure not centered}) does not hold, so there exist $\bd\in\overline{\mathit{Ad}}(\bd_{\calf})\times\{\bd_{\calf}\}$ and $v\in T^{(0)}-\{v_T\}$ such that $P_v(\bd) \in \overline{\mathcal{AC}}_n - \mathcal{AC}_n$.  Let us take $v$ to be maximal with this property, so that in particular $P_w(\bd) \in \mathcal{AC}_n$ for the endpoint $w$ of $e_v$.  If $\{\bd_n\}$ is a sequence in $\mathit{Ad}(\bd_{\calf})\times\{\bd_{\calf}\}$ approaching $\bd$, then $P_w(\bd_n) \to P_w(\bd)$, and so there is a universal upper bound on $J(P_w(\bd_n))$.  On the other hand, \RadiusUpNDown\  implies that $J(P_v(\bd_n)) \to \infty$, contradicting criterion (\ref{radius order}) of Definition \ref{admissible criteria} for some $\bd_n$.  Therefore (\ref{closure not centered}) holds.  \end{proof}

\begin{lemma}  Let $T \subset V$ be finite graphs such that $T$ is a rooted tree with root vertex $v_T$ and each vertex of $V$ has valence at least three.  Let $\mathcal{E}$ be the set of edges of $T$ and $\calf$ the frontier of $T$ in $V$, and fix $\bd_{\calf} = (d_e\,|\,e\in\calf)\in \mathbb{R}^{\calf}$ such that $\mathit{Ad}(\bd_{\calf})\neq \emptyset$.  Then $D_R(T,\bd)$ is continuous on $\overline{\mathit{Ad}}(\bd_{\calf})\times\{\bd_{\calf}\}$ and attains a minimum there.  \end{lemma}

\begin{proof}  Since $P\mapsto D_R(P)$ is continuous on $\mathcal{AC}_n$ (by \DefectDerivative), and by the above $P_v(\bd) \in \mathcal{AC}_n$ for each $\bd \in \overline{\mathit{Ad}}(\bd_{\calf})\times\{\bd_{\calf}\}$, $D_R(T,\bd)$ is continuous on $\overline{\mathit{Ad}}(\bd_{\calf})\times\{\bd_{\calf}\}$.  This set is closed and, by Lemma \ref{precompact}, bounded, so it is compact and $D_R(T,\bd)$ attains a minimum on it.
\end{proof}


For an arbitrary finite tree $T$ and $\bd_{\calf}\in\mathbb{R}^{\calf}$ as above, it seems difficult to precisely describe $\mathit{Ad}(\bd_{\calf})$ or determine the point in $\overline{\mathit{Ad}}(\bd_{\calf})$ at which $D_R(T,\bd)$ attains its minimum.  Here we will identify an alternative that such a minimum point must satisfy, at least for very simple $T$: those with one or two edges.  In the second half of the section, we will turn this into an algorithm that produces lower bounds on the minimum of $D_R(T,\bd)$, given lower bounds on the entries of $\bd_{\calf}$.

We first address the case that $T$ has a single edge.  In this case, uniquely, we are able to describe the topology of $\mathit{Ad}(\bd_{\calf})$ and locate the minimum of $D_R$.

\begin{lemma}\label{admissible quad} Let $V$ be a graph and $T$ a subgraph with one edge $e_T$ and root vertex $v_T$, and let $\calf$ be the frontier of $T$ in $V$.  For $\bd_{\calf}\in\mathbb{R}^{\calf}$, if $\mathit{Ad}(\bd_{\calf})\neq \emptyset$ it is an interval: $(d^-,d^+)$ or $[d^-,d^+)$.  For $R\geq 0$, $D_R(T,\bd)$ attains its minimum at $\bd= (d^-,\bd_{\calf})$, which satisfies one of:\begin{enumerate}
 \item $P_v(\bd)\in\widetilde{\calBC}_n$, where $n$ is the valence in $V$ of the initial vertex $v$ of $e_T$; or
 \item $P_{v_T}(\bd)\in\widetilde{\calBC}_{n_T}$, where $n_T$ is the valence of $v_T$ in $V$.
\end{enumerate}  
In case (2) above, $d^-$ is not the largest side length of $P_{v_T}(\bd)$.\end{lemma}

\begin{proof}  By Definition \ref{admissible criteria} $\mathit{Ad}(\bd_{\calf})$ is contained in a subset of $\mathbb{R}^+$ consisting of possible values for $d_{e_T}$.  By Lemma \ref{precompact}, if criterion (\ref{not centered}) is satisfied then $d_{e_T} \in [b_{e_T},h_{e_T})$, where $b_{e_T} = b_{e_T}(\bd_{\calf})$ and $h_{e_T}=h_{e_T}(\bd_{\calf})$ as in Lemma \ref{precompact}.  In fact, $P_v(\bd)\in\widetilde{\calAC}_n - \widetilde{\calc}_n$ if and only if $d_{e_T}\in[b_{e_T},h_{e_T})$, where $\bd = (d_{e_T},\bd_{\calf})$.  This follows from \VerticalRay, as pointed out in the base case of the proof of Lemma \ref{precompact}, and it follows that (\ref{not centered}) is satisfied if and only if $d_{e_T}\in[b_{e_T},h_{e_T})$.

Now consider criterion (\ref{centered}).  Let the edges of $V$ containing $v_T$ be cyclically enumerated $e_0,e_1,\hdots,e_{n_T}-1$ so that $e_0 = e_T$.  Then $e_i\in\calf$ for $1\leq i < n_T$.  Let $d = d_{e_T}$ and $d_i = d_{e_i}$ for $i>0$, and for $\bd = (d,\bd_{\calf})$ recall that from Definition \ref{admissible criteria} that:
$$ P_{v_T}(\bd) \doteq (d,d_1,\hdots,d_{n_T}-1) \in (\mathbb{R}^+)^n $$

Let $M = \max\{d_i\}_{i=1}^{n_T-1}$.  The inequality of \CenteredSpace, determining whether $P_{v_T}(\bd)\in\widetilde{\calc}_n$, takes different forms depending on the relation of $d$ to $M$.  For $d\geq M$, $P_{v_T}(\bd)\in\calc_{n_T}$ if and only if $A_d(d/2)+\sum_{i=1}^{n_T-1} A_{d_i}(d/2) > 2\pi$, where $A_d(J)$ is from \ParameterFunction.  By \MonotonicParameterFunction, $A_d(d/2) = \pi$ and $\sum_{i=1}^{n_T-1} A_{d_i}(J)$ decreases in $J$ to a horizontal asymptote of $0$, so the criterion of \CenteredSpace\ is satisfied at $d=M$ and there exists $D_0^+>M$ such that $A_d(d/2)+\sum_{i=1}^{n_T-1} A_{d_i}(d/2) > 2\pi$ if and only if $d<D_0^+$.  

For $d \leq M$, \CenteredSpace\ requires $A_{d}(M/2) + \sum_{i=1}^{n_T-1} A_{d_i}(M/2) > 2\pi$ for $P_v(\bd)\in\widetilde{\calc}_n$.  Since $A_{d}(J)$ is continuous and increases in $d$ there is an open interval of positive $d < M$, with left endpoint $D_0^- \geq 0$, on which this inequality holds.  Thus $\{d\,|\,P_{v_T}(\bd)\in\calc_n\} = (D_0^-,D_0^+)$.  If $D_0^- >0$ then $A_{D_0^-}(M/2) + \sum_{i=1}^{n_T-1} A_{d_i}(M/2) = 2\pi$.  In this case, for $\bd=(D_0^-,\bd_{\calf})$ \CenteredClosure\ implies $P_{v_T}(\bd)$ is in the closure of $\widetilde{\calc}_{n_T}$.  It is not in $\widetilde{\calc}_{n_T}$, so $P_{v_T}(\bd)\in\widetilde{\calBC}_{n_T}$ by \BCn.  Futhermore, its longest side length is $M >D_0^-$, since $M=d_i$ for some $i > 0$.

If $\mathit{Ad}(\bd_{\calf})\neq \emptyset$ then $[b_{e_T},h_{e_T}) \cap (D_0^-,D_0^+)$ is non-empty.  We claim that $f(d) = J(P_{v_T}(\bd)) - J(P_v(\bd))$ decreases in $d$ on $[b_{e_T},h_{e_T}) \cap (D_0^-,D_0^+)$, where $J\co\widetilde{\calAC}_n\to\mathbb{R}^+$ is as in \RadiusFunction.  This follows directly from \RadiusDeriv, which implies that on this interval:
$$ \left|\frac{\partial}{\partial d} J(P_{v_T}(\bd)) \right| < \frac{1}{2} < \left|\frac{\partial}{\partial d} J(P_{v}(\bd))\right| $$
By criterion (\ref{radius order}) of Definition \ref{admissible criteria}, $\mathit{Ad}(\bd_{\calf}) = f^{-1}(\mathbb{R}^+)\cap \left([b_{e_T},h_{e_T}) \cap (D_0^-,D_0^+)\right)$, so since $f$ is decreasing $\mathit{Ad}(\bd_{\calf})$ is a subinterval containing the left endpoint of $[b_{e_T},h_{e_T}) \cap (D_0^-,D_0^+)$, if it is non-empty.  If $b_{e_t}\leq D_0^-$ then $\mathit{Ad}(\bd_{\calf}) = (d^-,d^+)$ for $d^- = D_0^-$; otherwise $\mathit{Ad}(\bd_{\calf}) = [d^-,d^+)$ for $d^- = b_{e_T}$.  By \DefectDerivative, for $d$ in this interval the derivative $\frac{\partial}{\partial d} D_R(T,\bd)$ is:
$$ \cosh R \left[ \sqrt{\frac{1}{\cosh^2(d/2)} - \frac{1}{\cosh^2 J(P_{v_T}(\bd))}} - \sqrt{\frac{1}{\cosh^2 (d/2)} - \frac{1}{\cosh^2 J(P_{v}(\bd))}} \right]  $$
Since $J(P_{v_T}(\bd)) > J(P_{v}(\bd))$, this quantity is positive, and it follows that the defect sum increases with $d$.  Therefore its minimum is at $d^-$.

If $d^-= b_e$ then $P_{v}(\bd)\in\widetilde{\calBC}_n$ for $\bd = (d^-,\bd_{\calf})$, by Lemma \ref{precompact}, and condition (2) above holds.  If $d^- = D_0^-$ then since $D_0^- \geq b_{e_T} >0$ in this case, $P_{v_T}(\bd)\in\widetilde{\calBC}_{n_T}$ for $\bd = (d^-,\bd_{\calf})$ as we observed above, and condition (1) holds.  We also noted above that $d^- = D_0^-$ is not the longest edge of $P_{v_T}(\bd)$ in this case.\end{proof}

A two-edged tree is homeomorphic to an interval, but up to symmetry there are two possibilities for a root vertex: the intersection of the two edges, or one of the two boundary vertices.  Although these two possibilities have different admissible spaces, the locations at which the associated defect function may be minimized satisfy the same criteria.

\begin{proposition}\label{two-edge}  Let $V$ be a graph and $T$ a subtree of $V$ with two edges and root vertex $v_T$.  Let $\calf$ be the frontier of $T$ in $V$ and fix $\bd_{\calf}\in\mathbb{R}^{\calf}$ with $\mathit{Ad}(\bd_{\calf})\neq\emptyset$.  For $R\geq 0$, $D_R(T,\bd)$ attains a minimum at $\bd = (\bd_{\cale},\bd_{\calf})$ satisfying one of:\begin{enumerate}
  \item\label{non-roots near centered} $P_{v}(\bd)\in\widetilde{\calBC}_{n_v}$ for each $v\in T^{(0)}-\{v_T\}$, where $v$ has valence $n_v$ in $V$; or
  \item\label{root not centered} $P_{v_T}(\bd) \in\widetilde{\calBC}_{n_T}$, where $v_T$ has valence $n_T$ in $V$; or
  \item\label{equal radii} $J(P_{v}(\bd)) = J(P_{v_T}(\bd))$ for each $v\in T^{(0)}-\{v_T\}$. \end{enumerate}\end{proposition}

Proposition \ref{two-edge} follows directly from the two lemmas below, which separately address the  possible locations for the root vertex of $T$.

\begin{lemma}\label{two-edge inward}  Let $V$ be a graph and $T$ a subtree of $V$ with two edges that share its root vertex $v_T$.  With $\calf$, $\bd_{\calf}\in\mathbb{R}^{\calf}$, and $R\geq 0$ as in Proposition \ref{two-edge}, its conclusions hold.  \end{lemma}

\begin{proof}  Lemma \ref{admissible closure} describes $\overline{\mathit{Ad}}(\bd_{\calf})$ and asserts that $\sum_{v\in T^{(0)}} D_R(P_v(\bd))$ attains a minimum somewhere on $\overline{\mathit{Ad}}(\bd_{\calf})\times\{\bd_{\calf}\}$.  We will show that if $\bd = (\bd_{\cale},\bd_{\calf})$ satisfies none of the criteria of Proposition \ref{two-edge}, then $\bd_{\cale}$ may be deformed in $\overline{\mathit{Ad}}(\bd_{\calf})$ to reduce $\sum_{v\in T^{(0)}} D_R(P_v(\bd))$.

Let $e_1= e_{v_1}$ and $e_2 = e_{v_2}$, and let $d_1$ and $d_2$ be the respective lengths of their geometric duals.  Then $\bd_{\cale} = (d_1,d_2)$.  We note that as long as $\bd = (\bd_{\cale},\bd_{\calf}) \in\overline{\mathit{Ad}}(\bd_{\calf})$, then reducing either of $d_1$ or $d_2$ does not increase the defect sum, since (say) $\frac{\partial}{\partial d_1} \sum_{v\in T^{(0)}} D_R(P_v(\bd))$ is:
$$ \cosh R \left[ \sqrt{\frac{1}{\cosh^2(d_1/2)} - \frac{1}{\cosh^2 J(P_{v_T}(\bd))}} - \sqrt{\frac{1}{\cosh^2 (d_1/2)} - \frac{1}{\cosh^2 J(P_{v_1}(\bd))}} \right]  \geq 0 $$
This follows from \DefectDerivative\  because $\bd_{\cale} \in \overline{\mathit{Ad}}(\bd_{\calf})$ implies that $P_{v_T}(\bd) \in \overline{\calc}_{n_T}$, $d_1$ is the largest side length of $P_{v_1}(\bd) \in\mathcal{AC}_{n_1} - \calc_{n_1}$, and $J(P_{v_T}(\bd)) \geq J(P_{v_1}(\bd))$.  

Now assume that $\bd_{\cale}=(d_1,d_2)$ does not satisfy any of criteria (1) -- (3) from Proposition \ref{two-edge}.  Thus $P_{v_T}(\bd)\in\calc_n$ by (\ref{root not centered}), and by (\ref{equal radii}) we may assume that (say) $J(P_{v_T}(\bd)) > J(P_{v_1}(\bd))$.  If $P_{v_2}(\bd) \in \overline{\calc}_{n_2}$, then $P_{v_1}(\bd)\notin\overline{\calc}_{n_1}$ by (\ref{non-roots near centered}).  In this case, addressed in the paragraph below, we also have $J(P_{v_T}(\bd)) > d_2/2 = J(P_{v_2}(\bd))$, by \InCenteredClosure.

Since the radius varies continuously with $\bd$ (see \SmoothAnglesRadius), and $P_{v_1}(\bd)$ is in the open set $\mathcal{AC}_{n-1}-\overline{\calc}_{n-1}$, and $P_{v_T}(\bd)$ is in the open set $\calc_n$ there exists $\epsilon > 0$ such that for $d_1 -\epsilon < d_1' < d_1$ and $\bd_{\cale}' = (d_1',d_2)$, $P_{v_1}(\bd') \in \mathcal{AC}_{n_1} - \overline{\calc}_{n_1}$, $P_{v_T}(\bd')\in\calc_n$, and $J(P_{v_T}(\bd')) > J(P_{v_i}(\bd'))$ for $i = 1$ or $2$, where $\bd' = (\bd_{\cale}',\bd_{\calf})$.  Note that $P_{v_2}(\bd') = P_{v_2}(\bd) \in \mathcal{AC}_{n_2} - \calc_{n_2}$.  Therefore each $\bd_{\cale}'\in\overline{\mathit{Ad}}(\bd_{\calf})$, and the defect computation above gives: 
$$ \sum_{v\in T^{(0)}} D_R(P_v(\bd')) < \sum_{v\in T^{(0)}} D_R(P_v(\bd))$$ 
(In particular, since $J(P_{v_T}(\bd)) > J(P_{v_1}(\bd)) > d_1/2$, the inequality is strict.) 

Continuing to assume that (1) -- (3) do not hold, and in particular that $J(P_{v_T}(\bd)) > J(P_{v_1}(\bd))$, let us now suppose that $P_{v_2}(\bd) \notin \overline{\calc}_{n_2}$.  In this case it is possible that $J(P_{v_T}(\bd)) = J(P_{v_2}(\bd))$.  We will reduce $d_2$ instead of $d_1$, yielding $\bd'_{\cale} = (d_1,d'_2)$ for $d_2' < d_2$, and $\bd' = (\bd'_{\cale},\bd_{\calf})$.  Note that \RadiusDeriv\  implies that $\frac{\partial}{\partial d_2}J(P_{v_2}(\bd)) > \frac{1}{2} > \frac{\partial}{\partial d_2} J(P_{v_T}(\bd))$, and indeed this estimate holds at $\bd'$ for as long as $P_{v_T}(\bd')\in\calc_{n_T}$ and $P_{v_2}(\bd') \in \mathcal{AC}_{n_2} - \overline{\calc}_{n_2}$.

Let $\epsilon > 0$ be small enough that if $d_2 -\epsilon < d_2' < d_2$ and $\bd_{\cale}' = (d_1,d_2')$, then $P_{v_T}(\bd') \in \calc_{n_T}$, $P_{v_2}(\bd') \in \mathcal{AC}_{n_2} - \overline{\calc}_{n_2}$ and $J(P_{v_T}(\bd')) > J(P_{v_1}(\bd'))$, where $\bd' = (\bd_{\cale}',\bd_{\calf})$.  By the paragraph above, $J(P_{v_T}(\bd')) > J(P_{v_2}(\bd'))$ for such $\bd'$, and since $P_{v_1}(\bd') = P_{v_1}(\bd) \in \mathcal{AC}_{n_1} - \calc_{n_1}$ it follows that $\bd'\in\overline{\mathit{Ad}}(\bd_{\calf})$.  Furthermore, the change-of-defect computation using \DefectDerivative\  again implies a strict decrease in defect.\end{proof}

\begin{lemma}\label{two-edge sideways}  Let $V$ be a graph and $T$ a rooted subtree with two edges, only one containing the root vertex $v_T$, and other vertices $v_1$ and $v_2$.  With $\calf$, $\bd_{\calf}\in\mathbb{R}^{\calf}$, and $R\leq\min\{d_e/2\,|\,e\in\calf\}$ as in Proposition \ref{two-edge}, the conclusions of the proposition hold. \end{lemma}

\begin{proof}  Lemma \ref{admissible closure} describes $\overline{\mathit{Ad}}(\bd_{\calf})$ and asserts that $\sum_{v\in T^{(0)}} D_R(P_v(\bd))$ attains a minimum somewhere on $\overline{\mathit{Ad}}(\bd_{\calf})\times\{\bd_{\calf}\}$.  We will show that if $\bd = (\bd_{\cale},\bd_{\calf})$ satisfies none of the criteria above, then $\bd_{\cale}$ may be deformed in $\overline{\mathit{Ad}}(\bd_{\calf})$ to reduce $\sum_{v\in T^{(0)}} D_R(P_v(\bd))$.

Take $v_2$ to be the opposite endpoint of the edge $e_2 = e_{v_2}$ containing $v_T$, let $v_1$ be the far endpoint of the other edge $e_1 = e_{v_1}$, and let $d_1$ and $d_2$ be the lengths of the geometric duals to $e_1$ and $e_2$, respectively, so that $\bd_{\cale} = (d_1,d_2)$.  Assume now that $\bd = (\bd_{\cale},\bd_{\calf}) \in \overline{\mathit{Ad}}(\bd_{\calf}) \times \{\bd_{\calf}\}$ does not satisfy any of (1) -- (3) from Proposition \ref{two-edge}.  

Since $\bd_{\cale} \in \overline{\mathit{Ad}}(\bd_{\calf})$ we have $J(P_{v_1}(\bd)) \leq J(P_{v_2}(\bd)) \leq J(P_{v_T}(\bd))$.  Since $\bd$ does not satisfy (\ref{equal radii}), at least one of these inequalities is strict.  Let us suppose first that $J(P_{v_2}(\bd)) < J(P_{v_T}(\bd))$.  If $P_{v_1}(\bd) \in \overline{\calc}_{n_1}$, then ``not (\ref{non-roots near centered})'' implies that $P_{v_2}(\bd)\notin \overline{\calc}_{n_2}$, and furthermore:
$$ J(P_{v_2}(\bd)) > d_2/2 > d_1/2 = J(P_{v_1}(\bd))  $$
Therefore there exists $\epsilon > 0$ such that for all $d_2'$ with $d_2-\epsilon < d_2' < d_2$, taking $\bd_{\cale}' = (d_1,d_2')$ and $\bd' = (\bd_{\cale}',\bd_{\calf})$ we have $P_{v_2}(\bd') \notin \overline{\calc}_{n_2}$, $P_{v_T}(\bd')\in \calc_n$, and $J(P_{v_T}(\bd')) > J(P_{v_2}(\bd'))$.  We note that $P_{v_1}(\bd') = P_{v_1}(\bd)$ for all such $\bd'$.  \DefectDerivative\  implies that $\frac{\partial}{\partial d_2} \sum_{v\in T^{(0)}} D_R(P_v(\bd))$ is:
$$ \cosh R \left[ \sqrt{\frac{1}{\cosh^2(d_2/2)} - \frac{1}{\cosh^2 J(P_{v_T}(\bd))}} - \sqrt{\frac{1}{\cosh^2 (d_2/2)} - \frac{1}{\cosh^2 J(P_{v_2}(\bd))}} \right]  $$
As long as $J(P_{v_T}(\bd)) > J(P_{v_2}(\bd))$, this quantity is positive, so decreasing $d_2$ decreases the defect sum.  Thus with $\bd'$ as above we have $\sum_{v\in T^{(0)}} D_R(P_v(\bd')) <  \sum_{v\in T^{(0)}} D_R(P_v(\bd))$.

Continuing to suppose that $J(P_{v_2}(\bd)) < J(P_{v_T}(\bd))$, let us now also assume that $P_{v_1}(\bd) \notin \overline{\calc}_{n_1}$.  \RadiusDeriv\  implies that decreasing $d_1$ has the effect of decreasing $J(P_{v_1}(\bd))$ but increasing $J(P_{v_2}(\bd))$, since $P_{v_2}(\bd) \in\mathcal{AC}_{n_2} - \calc_{n_2}$ has longest side $d_2$.  Thus there exists $\epsilon > 0$ such that for $d_1 - \epsilon < d_1' < d_1$, taking $\bd_{\cale}' = (d_1',d_2)$ and $\bd' = (\bd_{\cale}',\bd_{\calf})$ we have $P_{v_1}(\bd') \notin\overline{\calc}_{n_1}$, $P_{v_2}(\bd') \in\mathcal{AC}_{n_2} - \calc_{n_2}$, and $J(P_{v_2}(\bd')) < J(P_{v_T}(\bd'))$.  Furthermore, $P_{v_T}(\bd') = P_{v_T}(\bd) \in\calc_n$ for all such $\bd'$, so $\frac{\partial}{\partial d_1} \sum_{v\in T^{(0)}} D_R(P_v(\bd'))$ is:
$$ \cosh R \left[ \sqrt{\frac{1}{\cosh^2(d_1'/2)} - \frac{1}{\cosh^2 J(P_{v_2}(\bd'))}} - \sqrt{\frac{1}{\cosh^2 (d_1'/2)} - \frac{1}{\cosh^2 J(P_{v_1}(\bd'))}} >0 \right]  $$
Thus we again find that $\sum_{v\in T^{(0)}} D_R(P_v(\bd')) < \sum_{v\in T^{(0)}} D_R(P_v(\bd))$ for $d_1' < d_1$.  (Note that even if $J(P_{v_1}(\bd)) = J(P_{v_2}(\bd))$, strict inequality holds for $\bd'$ by the above, and so the strict inequality of defect sums is also accurate.)

Let us finally suppose that $J(P_{v_2}(\bd)) = J(P_{v_T}(\bd))$.  Then since (\ref{equal radii}) does not hold, $J(P_{v_1}(\bd)) < J(P_{v_2}(\bd))$.  Since (\ref{root not centered}) does not hold we have $P_{v_T}(\bd)\in\calc_n$, so $J(P_{v_T}(\bd)) = J(P_{v_2}(\bd)) > d_2/2$ and so also $P_{v_2}(\bd))\notin\overline{\calc}_{n_2}$.  \RadiusDeriv\  implies that reducing $d_2$ reduces the radius of $P_{v_2}(\bd)$ faster than that of $P_{v_T}(\bd)$, and it follows that $d_2$ may be reduced slightly keeping $\bd_{\cale}\in\overline{\mathit{Ad}}(\bd_{\calf})$.  A derivative computation as above shows that this reduces the defect.\end{proof}

\section{Defect bounds from side length bounds}\label{bounds}

\begin{definition}  For $\calf$ finite and $\bb_{\calf},\bd_{\calf}\in\mathbb{R}^{\calf}$, say $\bd_{\calf} \geq \bb_{\calf}$ if $d_f \geq b_f$ for each $f\in\calf$, where $\bb_{\calf} = (b_f\,|\,f\in\calf)$ and $\bd_{\calf} = (d_{\calf}\,|\,f\in\calf)$.\end{definition}

This section describes an algorithm with the following:\begin{description}
\item[Input]  A rooted tree $T$ with frontier $\calf$, $R\geq 0$, and $\bb_{\calf}\in(\mathbb{R}^+)^{\calf}$.
\item[Output]  $D>0$ such that $D_R(T,(\bd_{\cale},\bd_{\calf})) \geq D$ for all $\bd_{\calf} \geq \bb_{\calf}$ and $\bd_{\cale}\in\overline{\mathit{Ad}}(\bd_{\calf})$.\end{description}

We begin with some \textit{a priori} bounds.

\begin{lemma}\label{horocyclic vertex poly}  Let $T$ be a rooted tree with root vertex $v_T$, edge set $\cale$, and frontier $\calf$ such that each $v\in T^{(0)}$ is at least three-valent in $T\cup\bigcup\,\{e\,|\,(e,v)\in \calf\}$.   Fix $\bb_{\calf} \in\mathbb{R}^{\calf}$.  For $v\in T^{(0)} - \{v_T\}$ let $e_0,\hdots, e_{n-1}$ be the set of edges containing $v$, with $e_v = e_0$, and define $P_v^h(\bb_{\calf})= (h_0(b_1,\hdots,b_{n-1}),b_1,\hdots,b_{n-1})\in\widetilde{\calHC}_n$, where $h_0$ is as in \HCn\ and:
$$ b_i = \left\{\begin{array}{ll} b_{e_i}(\bb_{\calf}) & e_i\in\cale \\ b_{e_i} & e_i\in\calf \end{array}\right.$$
Then for $R\geq 0$, $\bd_{\calf}\geq \bb_{\calf}$, $\bd_{\cale}\in\overline{\mathit{Ad}}(\bd_{\calf})$ and $\bd = (\bd_{\cale},\bd_{\calf})$, $D_R(P_v(\bd)) > D_R(P_v^h(\bb_{\calf}))$, where $P_v(\bd)$ is as in Definition \ref{admissible criteria}.\end{lemma}

\begin{proof}  Fix $\bd_{\calf}\geq\bb_{\calf}$ and $\bd_{\cale} = (d_e\,|\,e\in\cale) \in \mathit{Ad}(\bd_{\calf})$.  By Lemma \ref{precompact}, $d_e \geq b_e(\bd_{\calf})\geq b_e(\bb_{\calf})$ for each $e\in\cale$.  Thus for $v\in T^{(0)}-\{v_T\}$ and $e_0,\hdots,e_{n-1}$ as described in the hypotheses above, $b_i \leq d_{e_i}$ for each $i > 0$.  Criterion (\ref{closure not centered}) of Lemma \ref{admissible closure} implies that $P_v(\bd)\in \mathcal{AC}_n - \calc_n$ has longest edge $d_{e_v}$, so \HorocyclicBound\ implies $D_R(P_v(\bd)) > D_R(P_v^h(\bd_{\calf}))$.\end{proof}

\begin{proposition}\label{horocyclic tree}  Let $T$ be a rooted tree with root vertex $v_T$, edge set $\cale$, and frontier $\calf$ such that each $v\in T^{(0)}$ is at least three-valent in $T\cup\bigcup\,\{e\,|\,(e,v)\in \calf\}$.  Fix $\bb_{\calf}\in(\mathbb{R}^+)^{\calf}$.  For a subtree $T_0$ of $T$ with $v_T\in T_0$, let $\cale_0$ and $\calf_0$ be the edge set and frontier (in $T\cup\,\bigcup\,\{e\,|\,(e,v)\in\calf\}$) of $T_0$, respectively, and define $\bb_{\calf_0} = (b_e\,|\,e\in\calf_0)$ by:
$$ b_e = \left\{\begin{array}{ll} b_e & e\in\calf \\ b_e(\bb_{\calf}) & e\in\cale\end{array}\right. $$
Let $D_0 = \inf\{ D_R(T_0,(\bd_{\cale_0},\bd_{\calf_0}))\,|\,\bd_{\cale_0}\in\overline{\mathit{Ad}}(\bd_{\calf_0}),\ \bd_{\calf_0}\geq \bb_{\calf_0}\}$.  Then for any $\bd_{\calf}\geq \bb_{\calf}$ and $\bd_{\cale}\in\overline{\mathit{Ad}}(\bd_{\calf})$,
$$ D_R(T,\bd_{\calf}) > D_0 + \sum_{v\in T^{(0)}-T^{(0)}_0} D_R(P_v^h(\bb_{\calf})), $$
where $P_v^h(\bb_{\calf})$ is as in Lemma \ref{horocyclic vertex poly}.\end{proposition}

\begin{proof}  A fixed pair $\bd_{\calf} = (d_e\,|e\in\calf) \geq\bb_{\calf}$ and $\bd_{\cale}=(d_e\,|\,e\in\cale)\in\overline{\mathit{Ad}}(\bd_{\calf})$ determines $\bd_{\calf_0}$ and $\bd_{\cale_0}$ simply by taking the appropriate entries of $\bd_{\calf}$ or $\bd_{\cale}$.  Lemma \ref{precompact} and the construction of $\bb_{\calf_0}$ then imply that $\bd_{\calf_0}\geq\bb_{\calf_0}$.

Taking $\bd_0 = (\bd_{\cale_0},\bd_{\calf_0})$, it is clear by Definition \ref{admissible criteria} and the construction of $\bd_{\cale_0}$ and $\bd_{\calf_0}$ that for each $v\in T_0$, $P_v(T_0,\bd_0) = P_v(T,\bd)$, where $\bd=(\bd_{\cale},\bd_{\calf})$.  Definition \ref{admissible criteria} thus implies:
$$ D_R(T,\bd) = D_R(T_0,\bd_0) + \sum_{v\in T^{(0)}-T^{(0)}_0} D_R(P_v(\bd)) $$
Lemma \ref{admissible closure} implies that $\bd_{\cale_0}\in\overline{\mathit{Ad}}(\bd_{\calf_0})$, since $\bd_{\cale}\in\bd_{\calf}$ by hypothesis, so $D_R(T_0,\bd_0)\geq d_0$ and the result follows from Lemma \ref{horocyclic vertex poly}.  \end{proof}

Proposition \ref{horocyclic tree} can be used in conjunction with the result below to give \textit{a priori} bounds.

\begin{lemma}\label{root defect} Let $T$ be a rooted tree with root vertex $v_T$, edge set $\cale$, and frontier $\calf$ such that each $v\in T^{(0)}$ is at least three-valent in $T\cup\bigcup\,\{e\,|\,(e,v)\in \calf\}$.  For $\bb_{\calf}\in(\mathbb{R}^+)^{\calf}$, let $\bb_{\cale} = (b_e(\bb_{\calf})\,|\,e\in\cale) \in \mathbb{R}^{\cale}$ and take $\bb=(\bb_{\cale},\bb_{\calf})$.  Enumerate the edges of $V$ containing $v_T$ as $e_0,\hdots,e_{n-1}$ so that $b_{e_0}$ is largest.  With $b_0\co (\mathbb{R}^+)^{n-1}\to\mathbb{R}^+$ as in \BCn, define\begin{align*}
  & B_{e_0} = \left\{\begin{array}{ll} b_0(b_{e_1},\hdots,b_{e_{n-1}}) & \mbox{if}\ b_{e_0} > b_0(b_{e_1},\hdots,b_{e_{n-1}}) \\ b_{e_0} & \mbox{otherwise}\end{array}\right.\end{align*}
For $R\geq 0$, let $M_R(v_T,\bb_{\calf}) = D_R(B_{e_0},b_{e_1},\hdots,d_{e_{n-1}})$.  Then $D_R(P_{v_T}(\bd_{\cale},\bd_{\calf})) \geq M_R(v_T,\bb_{\calf})$ for each $\bd_{\calf}\geq\bb_{\calf}$, and $\bd_{\cale}\in\overline{\mathit{Ad}}(\bd_{\calf})$.\end{lemma}

\begin{proof}  Given $\bd_{\cale}$ and $\bd_{\calf}$ as above, Lemma \ref{precompact} implies that $d_{e_i} \geq b_{e_i}$ for each $i\in\{0,\hdots,n-1\}$, and therefore also that $d_{e_0} \geq B_{e_0}$.  Since $P_{v_T}(\bd_{\cale},\bd_{\calf})\in \widetilde{\calc}_n\cup\widetilde{\calBC}_n$ by Lemma \ref{admissible closure}, and $(B_{e_0},b_{e_1},\hdots,b_{e_{n-1}})\in\widetilde{\calc}_n\cup\widetilde{\calBC}_n$ by construction, \CenteredDefectBound\ implies the result.\end{proof}

\begin{remark}\label{tri root defect}  With the hypotheses of Lemma \ref{root defect}, if $v_T$ is three-valent in $V$ and $b_{e_0} > b_0(b_{e_1},b_{e_2})$, the conclusion may be improved, using \CenteredBoundaryBound, taking:
$$ M_R(v_T,\bb_{\calf}) = \min\{D_R(b_{e_0},b_{e_1},b_2'),D_R(b_{e_0},b_1',b_{e_2})\}, $$
where $\cosh b_{e_0} = \cosh b_1'+\cosh b_{e_2} -1 = \cosh b_{e_1}+\cosh b_2' -1$.\end{remark}

Given a rooted tree $T$ with root vertex $v_T$, edge set $\cale$, and frontier $\calf$ such that each vertex of $T$ is at least three-valent in $T\cup\bigcup_{e\in\calf} e$, for $\bb_{\calf}\in(\mathbb{R}^+)^{\calf}$ and $R\geq 0$, the procedure below can be algorithmically implemented:\begin{enumerate}
  \item For each $e\in\cale$, compute $b_e(\bb_{\calf})$ as in Lemma \ref{precompact} (see Remark \ref{outside-in}).
  \item For each $v\in T^{(0)}-\{v_T\}$, compute $D_R(P_v^h(\bb_{\calf}))$ for $P_v^h(\bb_{\calf})$ as in Lemma \ref{horocyclic vertex poly}.
  \item Compute $M_R(v_T,\bb_{\calf})$ with Lemma \ref{root defect}, or if $v_T$ has valence three, with Remark \ref{tri root defect}.
  \item Let $D = M_R(v_T,\bb_{\calf}) + \sum_{v\in T^{(0)}-\{v_T\}} D_R(P_v^h(\bb_{\calf}))$.\end{enumerate}
By Proposition \ref{horocyclic tree} (taking $T_0 = \{v_T\}$) and Lemma \ref{root defect}, $D$ as defined above is a lower bound on $D_R(T,(\bd_{\cale},\bd_{\calf}))$ for any $\bd_{\calf}\geq \bb_{\calf}$, and $\bd_{\cale}\in \overline{\mathit{Ad}}(\bd_{\calf})$

Below we will describe how to improve the procedure above under the assumption that $D_R(T,\bd)$ attains its minimum at a point of $\mathit{Ad}(\bd_{\calf})$ satisfying one of the three criteria of Proposition \ref{two-edge}.  We will treat these cases separately.

\subsection{Case (\ref{non-roots near centered})\label{case non-roots near centered}: $P_v(\bd)\in\widetilde{\calBC}_{n_v}$ for all $v\in T^{(0)}-\{v_T\}$.}  Lemma \ref{precompact} implies that each $\bd_{\calf}\in\mathbb{R}^{\calf}$ determines a unique $\bd_{\cale}\in\overline{\mathit{Ad}}(\bd_{\calf})$ such that $\bd=(\bd_{\cale},\bd_{\calf})$ falls into this case.  For such $\bd$, \Monotonicity\ implies:

\begin{lemma}\label{non-roots near centered bound} Let $T$ be a rooted tree with root vertex $v_T$, edge set $\cale$, and frontier $\calf$ such that each $v\in T^{(0)}$ is at least three-valent in $T\cup\bigcup\,\{e\,|\,(e,v)\in \calf\}$.  Fix $\bb_{\calf} \in\mathbb{R}^{\calf}$, let $\bb_{\cale} = (b_e(\bb_{\calf})\,|\,e\in\cale)$, and for $v\in T^{(0)} - \{v_T\}$ define $P_v^b(\bb_{\calf})=P_v(\bb_{\cale},\bb_{\calf})\in\widetilde{\calBC}_{n_v}$.  Then for $R\geq 0$, $\bd_{\calf}\geq \bb_{\calf}$, $\bd_{\cale}\in\overline{\mathit{Ad}}(\bd_{\calf})$ and $\bd = (\bd_{\cale},\bd_{\calf})$ such that $P_v(\bd)\in\widetilde{\calBC}_{n_v}$ for all $v\in T^{(0)}-\{v_T\}$, $D_R(P_v(\bd)) > D_R(P_v^b(\bb_{\calf}))$ for each such $v$.\end{lemma}

Using Lemma \ref{non-roots near centered bound}, we may improve the basic algorithm in this case by replacing the computation of $D_R(P_v^h(\bb_{\calf}))$ in step (2) with that of $D_R(P_v^b(\bb_{\calf}))$, and in step (4) taking $D = M_R(v_T,\bb_{\calf})+\sum_{v\in T^{(0)}-\{v_T\}} D_R(P_v^b(\bb_{\calf}))$.

\subsection{Case (\ref{root not centered})\label{case root not centered}: $P_{v_T}(\bd)\in\widetilde{\calBC}_{n_T}$.}

Our main advantage in this case is the following improved version of Lemma \ref{root defect}.

\begin{lemma}\label{better root defect out tree} Let $T$ be a rooted tree with root vertex $v_T$, edge set $\cale$, and frontier $\calf$ such that each $v\in T^{(0)}$ is at least three-valent in $T\cup\bigcup\,\{e\,|\,(e,v)\in \calf\}$.  For $\bb_{\calf}\in(\mathbb{R}^+)^{\calf}$, let $\bb_{\cale} = (b_e(\bb_{\calf})\,|\,e\in\cale) \in \mathbb{R}^{\cale}$ and take $\bb=(\bb_{\cale},\bb_{\calf})$, and enumerate the edges of $V$ containing $v_T$ as $e_0,\hdots,e_{n-1}$.  With $b_0\co (\mathbb{R}^+)^{n-1}\to\mathbb{R}^+$ as in \BCn, for each $i$ define
$$ B_{e_i} = b_0(b_{e_0},\hdots,\hat{b}_{e_i},\hdots,b_{e_{n-1}})  $$
Then for $R\geq 0$, $\bd_{\calf}\geq\bb_{\calf}$, and $\bd_{\cale}\in\overline{\mathit{Ad}}(\bd_{\calf})$ such that $P_{v_T}(\bd_{\cale},\bd_{\calf})\in\widetilde{\calBC}_{n_T}$ has longest side dual to $e_i$, $D_R(P_{v_T}(\bd_{\cale},\bd_{\calf})) \geq D_R(b_{e_0},\hdots,B_{e_i},\hdots,b_{e_{n-1}})$.\end{lemma}

This follows from Lemma \ref{precompact} and \Monotonicity\ as in the proof of Lemma \ref{root defect}.  If the longest edge of $P_{v_T}(\bd)$ is dual to an element of $\cale$, we may further augment Lemma \ref{root defect}:

\begin{lemma}\label{better root defect in tree} Let $T$ be a rooted tree with root vertex $v_T$, edge set $\cale$, and frontier $\calf$ such that each $v\in T^{(0)}$ is at least three-valent in $T\cup\bigcup\,\{e\,|\,(e,v)\in \calf\}$.  For $\bd_{\calf}\in(\mathbb{R}^+)^{\calf}$ and $\bd_{\cale}\in\overline{\mathit{Ad}}(\bd_{\calf})$ such that $P_{v_T}(\bd_{\cale},\bd_{\calf})\in\widetilde{\calBC}_{n_T}$ has longest side dual to $e\in\cale$, $P_{v}(\bd_{\cale},\bd_{\calf})\in\widetilde{\calBC}_{n_{v}}$, where $v$ is the initial vertex of $e$.\end{lemma}

\begin{proof}  Let $\bd = (\bd_{\cale},\bd_{\calf})$.  Since $P_{v_T}(\bd)\in\widetilde{\calBC}_{n_T}$ by hypothesis, \InCenteredClosure\ implies that $J(P_{v_T}(\bd)) = d_i/2$, where $d_i$ is the length of the geometric dual to $e_i$.  On the other hand, $P_{v_i}(\bd)$ also has longest side dual to $e_i$ by Lemma \ref{admissible closure}, which further implies that $J(P_{v_i}(\bd))\leq J(P_{v_T}(\bd))$.  Since $J(P_{v_0}(\bd))\geq d_i/2$ by \RadiusFunction, it is equal to $d_i/2$, and therefore the Lemma holds by \InCenteredClosure.\end{proof}

From Lemma \ref{precompact} and \Monotonicity\ we thus directly obtain:

\begin{corollary}\label{almost centered subtree}  With the hypotheses of Lemma \ref{better root defect out tree}, for $R\geq 0$, $\bd_{\calf}\geq \bb_{\calf}$, and $\bd_{\cale}\in\overline{\mathit{Ad}}(\bd_{\calf})$ such that $P_{v_T}(\bd_{\cale},\bd_{\calf})\in\widetilde{\calBC}_{n_T}$ has longest side dual to $e_i\in\cale$, $D_R(P_{v_i}(\bd))\geq D_R(P_{v_i}^b(\bb_{\calf}))$, where $\bd=(\bd_{\cale},\bd_{\calf})$, $v_i$ is the initial vertex of $e_i$, and $P_{v_i}^b(\bb_{\calf})$ is as in Lemma \ref{non-roots near centered bound}.\end{corollary}

\begin{remark}\label{better almost centered subtree}  With the hypotheses of Corollary \ref{almost centered subtree}, if $v_i$ is trivalent in $T\cup\bigcup\{e\,|\,(e,v)\in\calf\}$ and $b_{e_i}(\bb_{\calf}) < B_{e_i}$, then using \CenteredBoundaryBound\ as in Remark \ref{tri root defect} we have:
$$ D_R(P_{v_i}(\bd)) \geq\min\{D_R(B_{e_i},b_{f_1},b_2'),D_R(b_{e_0},b_1',b_{f_2})\} $$
Here $f_1$ and $f_2$ are the other edges containing $v_i$, $b_{f_1}$ and $b_{f_2}$ are as in Lemma \ref{non-roots near centered bound}, and $\cosh B_{e_i} = \cosh b_1'+\cosh b_{f_2} -1 = \cosh b_{f_1}+\cosh b_2' -1$.\end{remark}

In order to improve the basic algorithm in this case, enumerate the edges containing $v_T$ as $e_0,\hdots,e_{n-1}$, and in step (3) of the basic algorithm replace the computation of $M_R(v_T,\bb_{\calf})$ with those of $M_R^{(i)}(v_T,\bb_{\calf}) =D_R(b_{e_0},\hdots,B_{e_i},\hdots,b_{e_{n-1}})$ for each $i$ , where $B_{e_i}$ is as in Lemma \ref{better root defect out tree}.  It is useful now to divide into two subcases:\begin{description}
  \item[Case (\ref{root not centered})(A)]  In step (4) of the basic algorithm, replace $D$ with:
$$ D_A = \min\,\left\{M_R^{(i)}(v_T,\bb_{\calf})\,|\,e_i\in\calf\right\} + \sum_{v\in T^{(0)}-\{v_T\}} D_R(P_v^h(\bb_{\calf})) $$
  \item[Case (\ref{root not centered})(B)]  In step (2) of the basic algorithm, also compute $D_R(P_{v_i}^b(\bb_{\calf}))$ for each $i$ such that $e_i\in\cale$, where $v_i$ is the initial vertex of $e_i$, and in step (4) replace $D$ with:
$$ D_B = \min_{e_i\in\cale} \left\{M_R^{(i)}(v_T,\bb_{\calf}) + D_R(P_{v_i}^b(\bb_{\calf})) + \sum_{v\in T^{(0)} - \{v_T,v_i\}} D_R(P_v^h(\bb_{\calf}))\right\} $$
For each $i$ such that $v_i$ is trivalent, $D_R(P_{v_i}^b(\bb_{\calf})$ can be replaced by the computation from Remark \ref{better almost centered subtree} if $b_{e_i}(\bb_{\calf})< B_{e_i}$.\end{description}
By the results above, $D = \min\{D_A,D_B\}$ bounds $D_R(T,(\bd_{\cale},\bd_{\calf}))$ below for any $\bd_{\calf}\geq\bb_{\calf}$ and $\bd_{\cale}\in\overline{\mathit{Ad}}(\bd_{\calf})$.

\subsection{Case (\ref{equal radii})\label{case equal radii}: $J(P_v(\bd)) = J(P_{v_T}(\bd))$ for all $v\in T^{(0)}-\{v_T\}$.}  Here we have:

\begin{lemma}\label{centered equal radii}  Let $T$ be a rooted tree with root vertex $v_T$, edge set $\cale$, and frontier $\calf$ such that each $v\in T^{(0)}$ is at least three-valent in $T\cup\bigcup\,\{e\,|\,(e,v)\in \calf\}$.  For $\bd_{\calf}\in(\mathbb{R}^+)^{\calf}$ and $\bd_{\cale}\in\overline{\mathit{Ad}}(\bd_{\calf})$ such that $J(P_v(\bd)) = J(P_{v_T}(\bd))$ for each $v\in T^{(0)}-\{v_T\}$, where $\bd = (\bd_{\cale},\bd_{\calf})$, there exists $P\in\widetilde{\calc}_{|\calf|}\cup\widetilde{\calBC}_{|\calf|}$ such that $D_R(T,\bd) = D_R(P)$ for any $R\geq 0$.\end{lemma}

\begin{proof}  For $\bd = (\bd_{\cale},\bd_{\calf})$, note that $P_{v_T}(\bd)\in\widetilde{\calc}_{n_T}\cup\widetilde{\calBC}_{n_T}$ by Lemma \ref{admissible closure}.  Now fix a subtree $T_0$ of $T$ with $v_T \in T_0$ and $T = T_0\cup e_0$ for some $e_0\in\cale$, and assume that the following holds for $T_0$:  there is a cyclic polygon $P_0$ in $\mathbb{H}^2$ that contains a copy of $P_v(\bd)$ for each $v\in T^{(0)}_0$, such that $P_0 = \bigcup_{v\in T^{(0)}_0} P_v(\bd)$ and $P_v(\bd)\cap P_w(\bd)$ contains more than one point if and only if $v$ and $w$ bound an edge $e$ of $T_0$, in which case $P_v(\bd)\cap P_w(\bd)$ is the geometric dual to $e$.

The edge set of $P_0$ is in one-to-one correspondence with the frontier $\calf_0$ of $T_0$ in $V$, and we will assume that $P_0$ has the same center $x_0$ and radius $J$ as $P_v(\bd)\subset P_0$ for each $v\in T^{(0)}_0$.  Thus in particular, $P_0\in \widetilde{\calc}_{|\calf_0|}\cup\widetilde{\calBC}_{|\calf_0|}$ by \InCenteredClosure, since this implies that $x_0\in P_{v_T}(\bd)\subset P_0$.

Let $\{v_0\} = T^{(0)}-T^{(0)}_0$, and enumerate the edges of $V$ containing $v_0$ as $e_0,\hdots,e_{n-1}$, where $e_0$ is described above.  Then $e_i\in\calf$ for each $i>0$, since $T = T_0\cup e_0$, and $\calf_0 = \{e_0\}\cup (\calf-\{e_i\}_{i=1}^{n-1})$.  Since $e_0$ is necessarily the first edge of the path in $T$ joining $v_0$ to $v_T$, $v_0$ is its initial vertex with the orientation from Lemma \ref{to the root}.

Since $e_0\in \calf_0$, $P_0$ has an edge corresponding to its geometric dual $\gamma_0$.  Arrange a copy of $P_{v_0}(\bd)$ so that it intersects $P_0$ in $\gamma_0$.  The isosceles triangle $T_0$ determined by $\gamma_0$ and the center of $P_{v_0}(\bd)$ has equal sides of length $J(P_{v_0}(\bd)) = J(P_{v_T}(\bd)) = J$ by hypothesis.  Furthermore, Lemma \ref{admissible closure} implies that $P_{v_0}(\bd)\in\widetilde{\calAC}_n - \widetilde{\calc}_n$ has longest side $\gamma_0$, so $\gamma_0 = T_0\cap P_{v_0}(\bd)$ by \IsoscelesDecomp.  It follows that $T_0$ intersects the interior of $P_0$.

On the other hand, the triangle determined by $\gamma_0$ and the center $x_0$ of $P_0$ has two sides of length $J$ and by \IsoscelesDecomp\ is contained in $P_0$, since $P_0\in \widetilde{\calc}_{|\calf_0|}\cup\widetilde{\calBC}_{|\calf_0|}$.  Since this triangle has the same side length collection as $T_0$, share $\gamma_0$ with it, and is on the same side of $\gamma_0$ it is identical to $T_0$. Therefore $x_0$ is the center of $P_{v_0}(\bd)$, so by \PtsToPoly, $P = P_0\cup P_{v_0}(\bd)$ is a cyclic polygon with center $x_0$ and radius $J$. 

If $w_0$ is the terminal endpoint of $e_0$, then $P_{w_0}(\bd)\subset P_0$ contains $\gamma_0$, so $\gamma_0 = P_{v_0}(\bd)\cap P_{w_0}(\bd)$ and $P$ satisfies the hypotheses for $T$ that $P_0$ satisfied for $T_0$.  It is easy to see that $D_R(P) = D_R(P_0)+D_R(P_{v_0}(\bd))$, so the result follows by an inductive argument.\end{proof}

The corollary below thus follows directly from \Monotonicity, and supplies the required lower bound without appeal to the basic algorithm.

\begin{corollary}\label{equal radii bound}  Let $T$ be a rooted tree with root vertex $v_T$, edge set $\cale$, and frontier $\calf$ such that each $v\in T^{(0)}$ is at least three-valent in $T\cup\bigcup\,\{e\,|\,(e,v)\in \calf\}$.  Fix $\bb_{\calf}\in(\mathbb{R}^+)^{\calf}$, and enumerate the edges in $\calf$ as $\{e_0,\hdots,e_{n-1}\}$ so that $b_{e_0}$ is maximal.  Define:
$$ B_{e_0} = \left\{\begin{array}{ll} b_0(b_{e_1},\hdots,b_{e_{n-1}}) & \mbox{if}\ b_{e_0} > b_0(b_{e_1},\hdots,b_{e_{n-1}}) \\ b_{e_0} & \mbox{otherwise}\end{array}\right. $$
Then for $R\geq 0$, $\bd_{\calf}\geq\bb_{\calf}$, $\bd_{\calf}\geq\bb_{\calf}$, and $\bd_{\cale}\in\overline{\mathit{Ad}}(\bd_{\calf})$ such that $J(P_v(\bd)) = J(P_{v_T}(\bd))$ for each $v\in T^{(0)}-\{v_T\}$, where $\bd=(\bd_{\cale},\bd_{\calf})$, $D_R(T,\bd) \geq D_R(B_{e_0},b_{e_1},\hdots,b_{e_{n-1}})$.\end{corollary}

\section{Computations}\label{computations}

This section is devoted to applying our previous results to prove Theorem \ref{main}.  For $r_{\beta}$  as described in the theorem, $2.8298 < \cosh r_{\beta} < 2.8299$, and if $d_{\beta} = 2r_{\beta}$ then $15.0166 < \cosh d_{\beta} < 15.0167$.  Let $r_1$ and $d_1$ satisfy $\cosh r_1 = 2.8298$ and $\cosh d_1 = 15.0166$, respectively.  Then $r_1 < r_{\beta}$ and $d_1 < d_{\beta}$, and it is easy to show that $d_1 > 2r_1$.  

Table \ref{symm ctrd areas} records the radius-$r_1$ defect of the symmetric, centered $n$-gons $P_n(d_1)$ for $n = 3$ through $6$.  These computations use \CenteredDefectBound.  In each case we have truncated the result after five decimal places, so the actual defect value is greater than what is displayed.

\begin{table}[ht]\begin{center}\begin{tabular}{llllllll}
  $n$ & $D_{r_1}(P_n(d_1))$ & $n$ & $D_{r_1}(P_n(d_1))$ & $n$ & $D_{r_1}(P_n(d_1))$ & $n$ & $D_{r_1}(P_n(d_1))$\\ \hline
  3 & 0.12586 & 4 & 0.56593 & 5 & 1.22041& 6 & 2.00496 \smallskip\end{tabular}\end{center}
\caption{Radius-$r_1$ defects of highly symmetric polygons.}
\label{symm ctrd areas}
\end{table}

We will perform an analogous computation for centered dual $2$-cells, but initially focus only on those with five frontier edges.  To begin, let us note that an Euler characteristic computation implies:

\begin{remark}\label{edge plus 3}  If $T$ is a tree with edge set $\cale$ and frontier $\calf$, such that each vertex of $T$ is at least three-valent in $T\cup\{e\,|\,(e,v)\in\calf\ \mbox{for some}\ v\in T^{(0)}\}$, then $|\calf| \geq |\cale|+3$.\end{remark}

This implies in particular that for a $2$-cell $Q$ of the centered dual decomposition containing a component $T$ of $V^{(1)}_n$, if $Q$ has five edges then $T$ has at most two.  Carrying the same argument further, we find that the possibilities for such $Q$ are exactly those showed in Figure \ref{non-ctrd 5-edge}.  In the figure, Voronoi edges are dashed and black, and centered dual edges are solid and red.  We have labeled the possibilities by the corresponding components of $V^{(1)}_n$ (in bold), where subscripts describe the valence of each vertex, with the root vertex in bold.

\begin{figure}
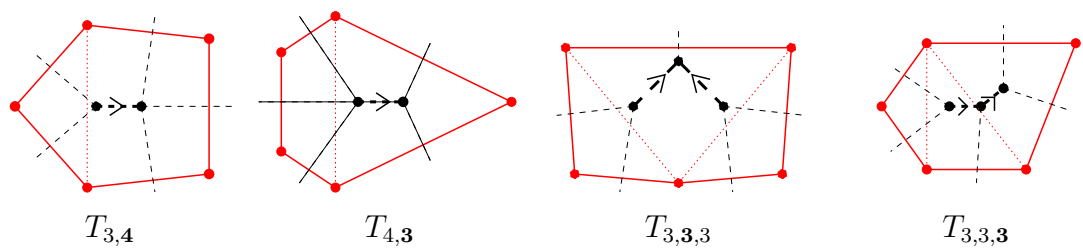
\caption{Possibilities for (non-Delaunay) $5$-edged centered dual $2$-cells.}
\label{non-ctrd 5-edge}
\end{figure}

\begin{table}[ht]\begin{tabular}{llllll}
  & Basic & Case (\ref{non-roots near centered}) & Case (\ref{root not centered})(A) & Case (\ref{root not centered})(B) & Case (\ref{equal radii}) \\ \hline
  $D_{r_1}(T_{3,\mathbf{4}},\bb)$ & $1.00510$ & \fbox{$1.17816$} & $1.57569$ & N/A & N/A \\
  $D_{r_1}(T_{4,\mathbf{3}},\bb)$ & $1.63705$ & $1.77971$ & \fbox{$1.71113$} & N/A & N/A \\
  $D_{r_1}(T_{3,\mathbf{3},3},\bb)$ & $0.80915$ & \fbox{$1.15527$} & $1.28432$ & $1.38738$ & $1.22041$ \\
  $D_{r_1}(T_{3,3,\mathbf{3}},\bb)$ & \fbox{$1.24735$} & $1.56044$ & $1.38585$ & $1.38738$ & $1.22041$ \medskip\end{tabular}
\caption{$D_{r_1}(T,\bb)$ for the $T$ from Figure \ref{non-ctrd 5-edge}, where $\bb = (d_1,d_1,d_1,d_1,d_1)$.}
\label{non-ctrd 5-edge defects}\end{table}  

Table \ref{non-ctrd 5-edge defects} records the output of computer programs implementing the algorithms of Section \ref{bounds} for the trees of Figure \ref{non-ctrd 5-edge}.  (The programs are in the supplementary materials.)  Each tree falls under the purview of the improvements to the basic algorithm described in the second half of the section, by Lemma \ref{admissible quad} for the one-edged trees and Proposition \ref{two-edge} for the others.  Since Lemma \ref{admissible quad} does not allow Case (\ref{root not centered})(B) or (\ref{equal radii}) of Proposition \ref{two-edge}, we placed ``N/A''s in the corresponding table entries.  We have boxed the best bound for each tree: the basic algorithm's output or the minimum of the improvements' (whichever is larger).

\begin{corollary}\label{no non-centered five}  If a closed orientable hyperbolic surface $F$ of genus $2$ has injectivity radius at least $d_1/2$, where $\cosh d_1 = 15.0166$, at $x\in F$, then no two-cell of the centered dual tessellation of $F$ determined by $x$ has more than four edges. \end{corollary}

\begin{proof}  If $F$ has injectivity radius at least $d_1/2$ at $x$, then since $r_1<d_1/2$ a hyperbolic disk $U$ of radius $r_1$ is embedded in $F$, centered at $x$.  The area of $U$ is $2\pi(\cosh r_1 - 1) > 2\pi\cdot 1.8298$, so the area of its complement in $F$ is less than $2\pi\cdot 0.1702 < 1.07$.

For each $2$-cell $P$ of the centered dual decomposition of $F$ determined by $\{x\}$, Proposition \ref{2-cell defect} implies that $D_{r_1}(P)$ is the area of $P - (P\cap U)$.  Let us first assume that $P$ has five edges.  Each is a geodesic arc that begins and ends at $x$, so its length is at least $d_1$.  If $P$ is centered, then \Monotonicity\ implies that $D_{r_1}(P) > D_{r_1}(P_n(d_1)) > 1.22$, by Table \ref{symm ctrd areas}.  This contradicts the fact that the total area complementary to $U$ in $F$ is less than $1.07$.

If $P$ contains a component $T$ of $V^{(1)}_n$, where $V$ is the Voronoi tessellation determined by $\{x\}$, then $T$ is one of the possibilities pictured in Figure \ref{non-ctrd 5-edge}.  Let $\calf$ be the frontier of $T$ in $V$.  Lemma \ref{admissible Voronoi} implies that $D_{r_1}(P) = D_{r_1}(T,(\bd_{\cale},\bd_{\calf}))$, where $\bd_{\calf}$ and $\bd_{\cale}\in\mathit{Ad}(\bd_{\calf})$ are defined there.  By the construction there and our hypothesis, $d_e \geq d_1$ for each $e$ with $(e,v)\in\calf$ for some $v$, so appealing to Table \ref{non-ctrd 5-edge defects}, we find that $D_{r_1}(P) > 1.15527$.  This again contradicts the fact that $F-U$ has area less than $1.07$.

By the above, no $2$-cell of the centered dual decomposition has five edges.  If a centered dual $2$-cell $Q$ with $n > 5$ edges is a centered Delaunay polygon, then by \Monotonicity\ $D_{r_1}(Q) = D_{r_1}(P_n(d_1))$.  This increases with $n$, by \CenteredDefectBound, so an appeal to Table \ref {symm ctrd areas} establishes a contradiction as above.  Now assume that $Q$ contains a component $T$ of $V^{(1)}_n$, let $\calf$ be the frontier of $T$ in $V^{(1)}$, and define $\bd_{\calf}$ and $\bd_{\cale}\in\mathit{Ad}(\bd_{\calf})$ as in Lemma \ref{admissible Voronoi}.

Again $d_e \geq d_1$ for each $e\in\cale$ or such that $(e,v)\in\calf$ for some $v$.  Thus if the root vertex $v_T$ of $T$ has valence at least $5$ in $V^{(1)}$ then: 
$$D_{r_1}(T,\bd) \geq D_{r_1}(P_{v_T}(\bd_{\cale},\bd_{\calf})) \geq D_{r_1}(P_5(d_1)) > 1.22 $$
If $v_T$ has valence four in $V^{(1)}$ and a vertex $v$ of $T$ adjacent to $v_T$ has valence at least four, let $T_0= e_v \subset T$ have frontier $\calf_0$ in $V^{(1)}$.  Applying the basic algorithm with $\bb_{\calf_0} = (d_1,\hdots,d_1)$ yields a bound of $1.8623$.  Proposition \ref{horocyclic tree} thus implies that $D_{r_1}(T,\bd) > 1.8623$.

If $v_T$ has valence three in $V^{(1)}$, so does a vertex $v$ adjacent to $v_T$ in $T$, and a vertex $w\neq v_T$ adjacent to $v$ in $T$ has valence at least four, applying the basic algorithm to $T_0=e_v\cup e_w$ with $\bb_{\calf_0} = (d_1,\hdots,d_1)$ yields a bound of $2.46104$, so Proposition \ref{horocyclic tree} implies that $D_{r_1}(T,\bd) > 2.46104$.  In all other cases $T$ has a subtree $T_0$, containing $v_T$, with the same combinatorics and frontier in $V^{(1)}$ as a tree from Figure \ref{non-ctrd 5-edge}, so as above $D_{r_1}(T,\bd) > 1.15$.  In all cases we obtain a contradiction.\end{proof}

\begin{lemma}\label{d_1 non-centered quad}  Let $T = e$ have frontier $\calf = \{(e_0,v),(e_1,v),(e_2,v_T),(e_3,v_T)\}$, where  $v_T$ is the root vertex and $v$ is the other.  For $d_{\beta}$ as in Example \ref{sharp DeBlois} and $\bd_{\calf}\geq (d_{\beta},d_{\beta},d_{\beta},d_{\beta})$, $D_{r_{\beta}}(T,(\bd_{\cale},\bd_{\calf})) > D_{r_{\beta}}(P_4(d_{\beta}))$ for each $\bd_{\cale} \in \mathit{Ad}(\bd_{\calf})$, where $r_{\beta} = d_{\beta}/2$.   \end{lemma}

\begin{proof}  We first note that if $\mathit{Ad}(\bd_{\calf})\neq\emptyset$, where $\bd_{\calf} = (d_0,d_1,d_2,d_3)$, then for $\bd_{\cale} = d_e\in\mathit{Ad}(\bd_{\calf})$, $P_{v}(\bd) = (d_e,d_0,d_1)\in\widetilde{\calAC}_3-\widetilde{\calc}_3$ and $P_{v_T}(\bd) = (d_e,d_2,d_3)\in\widetilde{\calc}_3$ by Definition \ref{admissible criteria}.  Thus \VerticalRay\ and \TriCenter\ imply that:
$$ \cosh d_0 + \cosh d_1 -1 =b_0(d_0,d_1) \leq \cosh d_e < b_0(d_2,d_3) = \cosh d_2 +\cosh d_3 - 1 $$
In particular, $\cosh d_0 + \cosh d_1 < \cosh d_2 + \cosh d_3$, so if $\bd_{\calf} \geq (d_{\beta},d_{\beta},d_{\beta},d_{\beta})$ then at least one of $d_2$ or $d_3$ is properly larger than $d_{\beta}$.

For $\bd_{\calf} \geq (d_{\beta},d_{\beta},d_{\beta},d_{\beta})$ with $\mathit{Ad}(\bd_{\calf})\neq \emptyset$, let us first suppose that the minimum of $D_{r_{\beta}}(T,(d_e,\bd_{\calf}))$ over $\overline{\mathit{Ad}}(\bd_{\calf})$ occurs at $(d^-,\bd_{\calf})$ satisfying Case (2) of Lemma \ref{admissible quad}.  The corresponding improvement of the basic algorithm, found in Section \ref{case root not centered}, outputs $0.74844$ given $R = r_2$ satisfying $\cosh r_2 = 2.8299$ and $\bb_{\calf} = (d_1,d_1,d_1,d_1)$.  Since $r_2>r_{\beta}$ and $d_1 < d_{\beta}$, it follows that $D_{r_{\beta}}(d_e,\bd_{\calf}) > 0.74844$ for each $d_e\in\mathit{Ad}(\bd_{\calf})$.  On the other hand, $D_{r_{\beta}}(P_4(d_{\beta})) < D_{r_1}(P_4(d_2)) < 0.56596$, where $\cosh d_2 = 15.0167$, so $D_{r_{\beta}}(d_e,\bd_{\calf}) >D_{r_{\beta}}(P_4(d_{\beta}))$ in this case.

Now suppose that the minimum of $D_{r_{\beta}}(T,(d_e,\bd_{\calf}))$ over $\overline{\mathit{Ad}}(\bd_{\calf})$ occurs at $(d^-,\bd_{\calf})$ satisfying Case (1) of Lemma \ref{admissible quad}.  Then $d^- = b_0(d_0,d_1) < b_0(d_2,d_3)$ by the first paragraph above, so $M_{r_{\beta}}(v_T,\bd_{\calf})$ as defined in Lemma \ref{root defect} is equal to $D_{r_{\beta}}(d^-,d_2,d_3)$.  Given $R = r_{\beta}$ and $\bd_{\calf}$, the corresponding improvement of the basic algorithm (from Section \ref{case non-roots near centered}) thus outputs $D_{r_{\beta}}(d^-,d_0,d_1) + D_{r_{\beta}}(d^-,d_2,d_3)$.

\HalfSymmQuadDefect\ implies that $D_{r_{\beta}}(P_4(d_{\beta})) = 2D_{r_{\beta}}(B_0,d_{\beta},d_{\beta})$, where $B_0 = b_0(d_{\beta},d_{\beta})$.  Since $\bd_{\calf}\geq (d_{\beta},d_{\beta},d_{\beta},d_{\beta})$, the monotonicity property of $b_0$ recorded in \BCn\ implies that $d^- \geq B_0$.  Therefore $(d^-,d_0,d_1) \geq (B_0,d_{\beta},d_{\beta})$ and $(d^-,d_2,d_3)\geq (B_0,d_{\beta},d_{\beta})$ in the sense of \Order, and by the first paragraph above the latter inequality is proper.  Thus \Monotonicity\ implies the lemma.\end{proof}

\begin{mainthm}\MainThm\end{mainthm}

\begin{proof} Suppose that a genus-two surface $F$ has injectivity radius at least $r_{\beta}$ at a point $x$.  Since $2r_{\beta} = d_{\beta} > d_1$, Corollary \ref{no non-centered five} implies that no $2$-cell of the centered dual tessellation determined by $\{x\}$ has more than four edges.  

A hyperbolic disk $U$ with radius $r_{\beta}$ is embedded in $F$ centered at $x$, and since $r_{\beta} > r_1$, the complementary area to $U$ in $F$ is less than $1.07$ (see the first paragraph of the proof of Corollary \ref{no non-centered five}).  Each edge of the centered dual decomposition has length at least $d_{\beta}$, so if $P$ is a centered quadrilateral $2$-cell of this decomposition then $P \geq P_4(d_{\beta})$ in the sense of \Order, and by \Monotonicity\  $D_{r_{\beta}}(P) \geq D_{r_{\beta}}(P_4(d_{\beta}))$. 

For a quadrilateral $2$-cell $P$ containing a component $T$ of $V^{(1)}_n$, where $V$ is the Voronoi tessellation determined by $\{x\}$, an argument like that for Remark \ref{edge plus 3} implies that $T$ and its frontier $\calf$ are as described in Lemma \ref{d_1 non-centered quad}.  The conclusion there and Lemma \ref{admissible Voronoi} thus imply that $D_{r_{\beta}}(P) > D_{r_{\beta}}(P_4(d_{\beta}))$ in this case.  

If $r_2$ satisfies $\cosh r_2 = 2.8299$ then $r_2 > r_{\beta}$, so since $d_{\beta} > d_1$ we have $D_{r_{\beta}}(P_4(d_{\beta})) > D_{r_2}(P_4(d_1)) > 0.56573$.  Thus by the above the centered dual tessellation has at most one quadrilateral $2$-cell, since $0.56573\cdot 2 = 1.13146 > 1.07$ and by Proposition \ref{2-cell defect}, $D_{r_{\beta}}(P)$ is the area of $P-(P\cap U)$ for each $2$-cell $P$.  

An Euler characteristic calculation now reveals that the centered dual tessellation of $F$ determined by $\{x\}$ consists of either six triangles or a quadrilateral and four triangles.  Recall that the tessellated surface $F_{\beta}$ of Example \ref{sharp DeBlois} has the latter combinatorics and all edges of length $d_{\beta}$.  By Corollary \ref{sharp DeBlois Delaunay} this is the Delaunay tessellation of $F_{\beta}$ determined by $\{x_{\beta}\}$, and since each polygon is centered by construction, it is also the centered dual tessellation.  Since $F_{\beta}$ has area $4\pi$, Proposition \ref{2-cell defect} thus gives:
$$  D_{r_{\beta}}(P_4(d_{\beta})) + 4\cdot D_{r_{\beta}}(P_3(d_{\beta})) = 4\pi - 2\pi(\cosh r_{\beta}-1) $$
If the centered dual tessellation of $F$ determined by $\{x\}$ has a quadrilateral $2$-cell $P$, let $T_1$, $T_2$, $T_3$, and $T_4$ be its triangular $2$-cells.  Remark \ref{edge plus 3} implies that no triangular $2$-cell of the centered dual decomposition contains a component of $V^{(1)}_n$; hence each is a centered Delaunay polygon.  Therefore since each edge of the centered dual tessellation has length at least $d_{\beta}$, $D_{r_{\beta}}(T_i) \geq D_{r_{\beta}}(P_3(d_{\beta}))$ for each $i$ by \Monotonicity.  By the above $D_{r_{\beta}}(P) \geq D_{r_{\beta}}(P_4(d_{\beta}))$, with strict inequality if $P$ contains a component of $V^{(1)}_n$.  In the latter case:\begin{align}\label{too big}
 D_{r_{\beta}}(P) + \sum_{i=1}^4 D_{r_{\beta}}(T_i) > 4\pi - 2\pi\cdot(\cosh r_{\beta} - 1) \end{align}
But this contradicts Proposition \ref{2-cell defect} and the fact that $F$ has area $4\pi$.

It follows that each $2$-cell of the centered dual decomposition determined by $\{x\}$ is centered, and hence that this is also the Delaunay tessellation.  If $F$ has injectivity radius greater than $r_{\beta}$ at $x$ then each edge of the Delaunay tessellation has length greater than $d_{\beta}$.  Thus if the Delaunay tessellation had quadrilateral component $P$ in this case, we would again have the inequality (\ref{too big}).  This is again a contradiction, and the theorem follows.\end{proof}

\section{Geometric consequences}\label{consequences}

This section describes the geometric consequences of Theorem \ref{main} for hyperbolic surfaces of genus $2$ that have large injectivity radius at some point.  It will be convenient to work with a model space that we may regard as parametrizing the set of pairs $(F,x)$, for $F$ in $\calm_2$ and $x\in F$, using the Delaunay tessellation determined by $\{x\}$.

\begin{definition}\label{to M_2}  We will say that an \textit{edge-pairing} is a fixed point-free involution $\iota\in S_{18}$.  For an edge-pairing $\iota$, define 
$$\calp_{\iota} = \left\{(d_0,\hdots,d_{17})\in(\widetilde{\calAC}_3)^6\,|\, d_i = d_{\iota(i)}\ \forall\ i,\ \sum_{j=0}^5 D_0(d_{3j},d_{3j+1},d_{3j+2}) = 4\pi\right\}$$
For $(d_0,\hdots,d_{17})\in\calp_{\iota}$ and $j\in\{0,\hdots,5\}$, let $T_j\subset\mathbb{H}^2$ be represented by $(d_{3j},d_{3j+1},d_{3j+2})$ in the sense of \CenteredSpace, with sides $\gamma_i$ such that $\ell(\gamma_i) = d_i$ for $i\in\{3j,3j+1,3j+2\}$.  Let $F_{\iota}(d_0,\hdots,d_{17})$ be obtained from $\bigsqcup_{j=0}^5 T_j$ by isometrically identifying $\gamma_i$ with $\gamma_{\iota(i)}$ for each $i$ so that $x_{i-1}\to x_{\iota(i)}$ and $x_i\to x_{\iota(i)-1}$ for each $i\in\{3j,3j+1,3j+2\}$, and induce a metric on $F_{\iota}(d_0,\hdots,d_{17})$ from those on the $T_j$ (in the sense of,say, \cite[Ch.~I.7]{BrH}).\end{definition}

Let us make a few initial observations:\begin{itemize}
\item  By the definition of $d_{\alpha}$ in Example \ref{sharp Boroczky}, $(d_{\alpha},\hdots,d_{\alpha})\in\calp_{\iota}$ for each edge-pairing $\iota$.
\item  The cellular isomorphism type of $F_{\iota}(d_0,\hdots,d_{17})$ does not depend on the choice of $(d_0,\hdots,d_{17})\in\calp_{\iota}$.  We say $\iota$ is \textit{one-vertex} if (say) $F_{\iota}(d_{\alpha},\hdots,d_{\alpha})$ has one vertex.
\item  For each one-vertex edge-pairing $\iota$ and $(d_0,\hdots,d_{17})\in\calp_{\iota}$, $F_{\iota}(d_0,\hdots,d_{17})$ is isometric to a closed, orientable hyperbolic surface of genus $2$.\end{itemize}

The final observation above follows from the fact that the triangles $T_j$ have total angle sum $2\pi$, so the single vertex of $F_{\iota}(d_0,\hdots,d_{17})$ has a neighborhood isometric to one in $\mathbb{H}^2$.  We may thus take $(d_0,\hdots,d_{17})\mapsto F_{\iota}(d_0,\hdots,d_{17})$ as defining a map $F_{\iota}\co\calp_{\iota}\to\calm_2$.  To show continuity of this map we will lift it to $\calt_2$.

\begin{definition}\label{to T_2}  For a one-vertex edge-pairing $\iota$, fix a maximal subtree $S$ of the one-skeleton of the abstract dual to $F_{\iota} = F_{\iota}(d_{\alpha},\hdots,d_{\alpha})$, and let $[\gamma_{i_0}],[\gamma_{i_1}],[\gamma_{i_2}],[\gamma_{i_3}]$ be the edges of $F_{\iota}^{(1)}$ not dual to edges of $S$.  Further fix $v\in\mathbb{H}^2$ and a geodesic ray $\delta$ from $v$. 

For $(d_0,\hdots,d_{17})\in\calp_{\iota}$ embed $T_0$ in $\mathbb{H}^2$ with its center at $v$ and $x_0 = \gamma_0\cap\gamma_1\in\delta$, and for $j>0$ embed $T_j$ in $\mathbb{H}^2$ so that $T_j\cap T_{j'} = \gamma_i$ for each $i\in\{3j,3j+1,3j+2\}$ such that $[\gamma_i]\subset F_{\iota}(d_0,\hdots,d_{17})$ is dual to an edge of $S$, where $\iota(i)\in\{3j',3j'+1,3j'+2\}$.  Let $O = \bigcup_{j=0}^5 T_j\subset\mathbb{H}^2$.  For $j\in\{0,1,2,3\}$ let $f_j\in\mathrm{Isom}^+(\mathbb{H}^2)$ satisfy $f(\gamma_{\iota(i_j)})=\gamma_{i_j}$ and $f(x_{\iota(i_j)}) = x_{i_j-1}$, and let $\widetilde{F}_{\iota}(d_0,\hdots,d_{17}) = (f_0,f_1,f_2,f_3)\subset(\mathrm{Isom}^+(\mathbb{H}^2))^4$.
\end{definition}

That an embedding of $T_0$ is prescribed by the choice of $v$ and $\delta$, given $(d_0,d_1,d_2)\in\widetilde{\calAC}_3$, follows from \PolyConvergence.  Since $S$ is a tree the embedding of $T_0$ and the criteria of Definition \ref{to T_2} determine embeddings of the other $T_j$.  Then $O = \bigcup_{j=0}^5 T_j$ is an octagon with edge set $\{\gamma_{i_j},\gamma_{\iota(i_j)}\}_{j=0}^3$, and the Poincar\`e polygon theorem implies that $\langle f_j\rangle_{j=0}^3$ is a discrete subgroup of $\mathrm{Isom}^+(\mathbb{H}^2)$ with fundamental domain $O$ and quotient isometric to $F_{\iota}(d_0,\hdots,d_{17})$.  

Fixing generators $g_0,g_1,g_2,g_3$ for $\pi_1 F$, an embedding $\mathrm{Hom}(\pi_1 F,\mathrm{Isom}^+(\mathbb{H}^2))\hookrightarrow (\mathrm{Isom}^+(\mathbb{H}^2))^4$ is given by $\rho\mapsto (\rho(g_0),\rho(g_1),\rho(g_2),\rho(g_3))$.  $\calt_2$ inherits the \textit{algebraic topology} as the subspace topology from this embedding, when it is identified with the set of discrete, faithful representations in $\mathrm{Hom}(\pi_1 F,\mathrm{Isom}^+(\mathbb{H}^2))$ (see \cite[\S 10.3]{FaMa}).  The paragraph above implies that $\widetilde{F}_{\iota}(d_0,\hdots,d_{17})\in\widetilde{\calt}_2$.  Therefore $(d_0,\hdots,d_{17})\mapsto \widetilde{F}_{\iota}(d_0,\hdots,d_{17})$ determines a map $\widetilde{F}_{\iota}\co\calp_{\iota}\to\calt_2$.  Since $F_{\iota}(d_0,\hdots,d_{17})$ is the quotient of $\mathbb{H}^2$ by $\langle f_j\rangle_{j=0}^3$, $\widetilde{F}_{\iota}$ lifts $F_{\iota}$.

\begin{lemma}\label{continuous model}  For each one-vertex edge-pairing $\iota$, $\widetilde{F}_{\iota}\co\calp_{\iota}\to\calt_2$ is continuous, where $\calt_2$ has the algebraic topology.\end{lemma}

\begin{proof}  It is well known that an orientation-preserving isometry of $\mathbb{H}^2$ is determined by its values on distinct $x,y\in\mathbb{H}^2$, and that the isometry so-determined varies continuously with the destinations of $x$ and $y$.  The following claim thus implies the lemma: the vertices of the octagon $O$ from Definition \ref{to T_2} vary continuously with $(d_0,\hdots,d_{17})\in\calp_{\iota}$.  

We will show that the vertices of each $T_j$, embedded as prescribed in Definition \ref{to T_2}, vary continuously with $(d_0,\hdots,d_{17})$.  We use induction, outward on the tree $S$ from its vertex corresponding to $T_0$.  The base case follows directly from \PolyConvergence, which implies that the vertices of $T_0$ vary continuously with $(d_0,d_1,d_2)$.

At least one edge of $T_0$, say $\gamma_0$, is dual to an edge of $S$.  Then for $j$ such that $T_j$ contains $\gamma_{\iota(0)}$, we embed $T_j$ in $\mathbb{H}^2$ as prescribed in Definition \ref{to T_2} in two steps: first embed $T_j$ with center $v$, $x_{\iota(0)}\in\delta$, and $\gamma_{\iota(0)}\subset\calh$, then move it via an isometry so that $\gamma_{\iota(0)}=\gamma_0$ and $x_{\iota(0)} = x_1$.  By \PolyConvergence, the vertices of the initial embedding vary continuously with $(d_0,\hdots,d_{17})$, so by the base case and the observation at the beginning of this proof, the vertices of the second do as well.  The general inductive step is no more complicated.\end{proof}

Since the moduli space $\calm_2$ inherits its usual topology as the quotient of a discontinuous group action on $\calt_2$ (again see \cite[\S 10.3]{FaMa}), Lemma \ref{continuous model} implies that $F_{\iota}\co \calp_{\iota}\to\calm_2$ is  continuous for each one-vertex edge-pairing $\iota$.  

We now fix attention on $\calm_2^{(r_{\beta})}$, the set of $F\in\calm_2$ with injectivity radius at least $r_{\beta}$ at some $x\in F$, where $r_{\beta}$ is as defined in Example \ref{sharp DeBlois}.  Such pairs $(F,x)$ fall under the purview of Theorem \ref{main}.  We will begin by relating $\calm_2^{(r_{\beta})}$ to the model spaces $\calp_{\iota}$.

\begin{lemma}\label{edge upper bound}  For $r\in[r_{\beta},r_{\alpha}]$, where $r_{\alpha}$ is as in Example \ref{sharp Boroczky}, let $d_r = 2r$.  There is a unique $d_1(r)\in[d_r,b_0(d_r,d_r)]$, where $b_0$ is as in \BCn, so that:
$$ 4\pi = 4\cdot D_{0,3}(d_1(r)) + 2\cdot D_0(d_1(r),d_r,d_r) $$
This satisfies $d_1(r_{\beta}) = b_0(d_{\beta},d_{\beta})$, $d_1(r_{\alpha})=d_{\alpha}$, and $d_1(r)\in(d_r,b_0(d_r,d_r))$ for $r\in (r_{\beta},r_{\alpha})$.  For a one-vertex edge-pairing $\iota$ and $(d_0,\hdots,d_{17})\in\calp_{\iota}\cap (\widetilde{\calc}_3\cup\widetilde{\calBC}_3)^6$ such that $d_i\geq d_r$ for each $i$, $d_i \leq d_1(r)$ for each $i$.\end{lemma}


\begin{proof}  For each $d\in [d_r,b_0(d_r,d_r)]$, \VerticalRay\ implies that $(d,d_r,d_r)\in\widetilde{\calc}_3\cup\widetilde{\calBC}_3$.   \DefectDerivative\ thus implies that $D_0(d,d_r,d_r)$ increases with $d$ on this interval, so:
$$ 6\cdot D_{0,3}(d_r) \leq 4\cdot D_{0,3}(d_r) + 2\cdot D_0(d,d_r,d_r) \leq 4\cdot D_{0,3}(d_r) + 2\cdot D_0(b_0(d_r,d_r),d_r,d_r) $$
Since $d_r\leq d_{\alpha}$, \DefectDerivative\ and the construction of $d_{\alpha}$ (see Example \ref{sharp Boroczky}) imply that $6\cdot D_{0,3}(d_r) \leq 4\pi$, with equality if and only if $d_r = d_{\alpha}$.  As we observed in Example \ref{dull DeBlois}, for $d_{\beta}$ as defined in Example \ref{sharp DeBlois}, \HalfSymmQuadDefect\ implies that $4\cdot D_{0,3}(d_{\beta}) + 2\cdot D_0(b_{\beta},d_{\beta},d_{\beta})=4\pi$, where $b_{\beta} = b_0(d_{\beta},d_{\beta})$.  Therefore since $d_{\beta}\leq d_r$, \DefectDerivative\ gives:
$$ 4\pi = 4\cdot D_{0,3}(d_{\beta}) + 2\cdot D_0(b_{\beta},d_{\beta},d_{\beta}) \leq 4\cdot D_{0,3}(d_r) + 2\cdot D_0(b_0(d_r,d_r),d_r,d_r), $$
with equality if and only if $d_r = d_{\beta}$.  The continuity and monotonicity of $D_0$ imply that $d_1(r)$ exists and is unique for each $r$.  Furthermore, by the above $d_r < d_1(r) < b_0(d_r,d_r)$ unless $r = r_{\beta}$ or $r = r_{\alpha}$, and $d_1(r_{\beta}) = b_{\beta}$ and $d_1(r_{\alpha}) = d_{\alpha}$.

For a one-vertex edge-pairing $\iota$ and $(d_0,\hdots,d_{17})\in\calp_{\iota}\cap (\widetilde{\calc}_3\cup\widetilde{\calBC}_3)^6$ such that $d_i\geq d_r$ for each $i$, \Monotonicity\ implies that $D_0(d_{3j},d_{3j+1},d_{3j+2}) \geq D_{0,3}(d_r)$ for each $j$ between $0$ and $5$, where $D_{0,3}$ is as in \CenteredDefectBound.  If $d = \max\{d_i\}_{i=0}^{17}$ then another application of \Monotonicity\ gives:
$$ \sum_{i=1}^6 D_0(T_i) \geq 4\cdot D_{0,3}(d_r) + 2\cdot D_0(d,d_r,d_r) $$
To justify the ``$2$'' above, note that if $d_i = d$ then $d_{\iota(i)} = d$ as well, and there is no $j$ such that $i,\iota(i)\in\{3j,3j+1,3j+2\}$ since $\iota$ is one-vertex.  If $d > d_1(r)$, then by construction (and \Monotonicity) it would follow that $\sum_{i=1}^6 D_0(T_i) > 4\pi$, contradicting $(d_0,\hdots,d_{17})\in\calp_{\iota}$.\end{proof}

\begin{lemma}\label{compact model}  If $F\in\calm_2$ has injectivity radius $r\geq r_{\beta}$ at $x$, then there is a one-vertex edge-pairing $\iota$ and $(d_0,\hdots,d_{17})\in\calp_{\iota}\cap (\widetilde{\calc}_3\cup\widetilde{\calBC}_3)^6\cap [d_{\beta},b_0(d_{\alpha},d_{\alpha})]^{18}$, such that $F$ is isometric to $F_{\iota}(d_0,\hdots,d_{17})$, taking $x$ to its vertex.\end{lemma}

\begin{proof}  For such $F\in\calm_2^{(r_{\beta})}$ and $x\in F$, let $P$ be the Delaunay tessellation of $F$ determined by $\{x\}$.  By Theorem \ref{main} and Lemma \ref{vertex polygon centered}, each $2$-cell of $P$ is centered.  If $P$ is a triangulation then it has $2$-cells $T_0,\hdots,T_5$.  Cyclically ordering the sides of each $T_j$ and recording their lengths, determines a one-vertex edge-pairing $\iota$ and $(d_0,\hdots,d_{17})\in\calp_{\iota}$ with $F = F_{\iota}(d_0,\hdots,d_{17})$ and $(d_{3j},d_{3j+1},d_{3j+2})\in\widetilde{\calc}_3$ for each $j$.

If $P$ has a quadrilateral $2$-cell $O$, a diagonal of $O$ divides it into two cyclic triangles.  Then making choices as above yields $\iota$ and $(d_0,\hdots,d_{17})$ with $F = F_{\iota}(d_0,\hdots,d_{17})$.  In this case, however, for each $j$ corresponding to a triangle in $O$ the corresponding triple $(d_{3j},d_{3j+1},d_{3j+2})$ is in $\widetilde{\calBC}_3$.

For each $i$, $d_i \geq d_{\beta} = 2r_{\beta}$ since it is an edge length of $P$, so Lemma \ref{edge upper bound} implies that $d_i\leq d_1(r)\leq b_0(d_r,d_r)$ for each $i$.  Furthermore, $b_0(d_r,d_r) \leq b_0(d_{\alpha},d_{\alpha})$ by \BCn, since $d_r\leq d_{\alpha}$ by Lemma \ref{Boroczky}.\end{proof}

\begin{lemma}\label{model to inj}  For each one-vertex edge-pairing $\iota$ and $(d_0,\hdots,d_{17})\in\calp_{\iota}\cap (\widetilde{\calc}\cup\widetilde{\calBC}_3)^6$, $F\doteq F_{\iota}(d_0,\hdots,d_{17})$ has injectivity radius $r=\min \{d_i/2\}_{i=0}^{17}$ at its vertex $x$.\end{lemma}

\begin{proof}  We argue as in Examples \ref{sharp Boroczky} and \ref{sharp DeBlois}, using \PackingVsDecomp.  Of course $F$ has injectivity radius at most $r$ at $x$.  Since each $(d_{3j},d_{3j+1},d_{3j+2})\in\widetilde{\calc}_3\cup\widetilde{\calBC}_3$ for each $j$, each $T_j$ contains the entire sector that it determines in a disk with radius $r$ centered at any of its vertices.  Since the $T_j$ have total angle measure $2\pi$, a disk with radius $r$ is embedded in $F$, centered at $x$.\end{proof}

\begin{corollary}\label{funny mumford}  $M_2^{(r_{\beta})}$ is compact.\end{corollary}

\begin{proof}  Lemmas \ref{compact model} and \ref{model to inj} together imply that $\calm_2^{(r_{\beta})}$ is the union, taken over the finite collection of one-vertex edge-pairings $\iota$, of the image of $\calp_{\iota}\cap(\widetilde{\calc}_3\cup\widetilde{\calBC}_3)^6\cap[d_{\beta},b_0(d_{\alpha},d_{\alpha})]^{18}$ under $F_{\iota}$.  Each such $F_{\iota}$ is continuous by Lemma \ref{continuous model}.  $\widetilde{\calc}_3\cup\widetilde{\calBC}_3$ is closed in $\mathbb{R}^+$ by \CenteredClosure\ and \BCn, so since $d_{\beta}>0$ its intersection with $[d_{\beta},b_0(d_{\alpha},d_{\alpha})]^3$ is compact.  The Corollary follows.\end{proof}

\begin{definition} Say the \textit{covering radius} of $\cals\subset\mathbb{H}^2$ is $\inf\,\{r\leq\infty\,|\,\mathbb{H}^2\subset \bigcup_{x\in\cals}B_r(x)\}$.\end{definition}

\begin{lemma}\label{radius to cov}  Let $F$ be a closed surface with universal cover $p\co\mathbb{H}^2\to F$, fix $\cals\subset F$ finite, and let $V$ be the Voronoi tessellation of $F$ determined by $\cals$.  Then $\widetilde{\cals}\doteq p^{-1}(\cals)$ has covering radius equal to $\max\, \{J_v\,|\,v\in V^{(0)}\}$, where $J_v$ is as in Lemma \ref{vertex radius}.\end{lemma}

\begin{proof}  It is clear that the covering radius of $\widetilde{\cals}$ is at least the quantity above, since for any $v\in\widetilde{V}^{(0)}$, $d(x,v) \geq J_{p(v)}$ for each $x\in\widetilde{\cals}$, where $\widetilde{V} = p^{-1}(V)$.  We will prove equality by showing that $P_v\subset \bigcup_{i=0}^{n-1} \overline{B}_{J_{p(v)}}(x_i)$ for each such $v$, where $P_v$ as in Lemma \ref{vertex polygon} has vertex set $\{x_i\}_{i=0}^{n-1}$, and $\overline{B}_r(x)$ is the closed ball of radius $r$ about $x$.

Let us assume that the $\{x_i\}$ are cyclically ordered in the sense of Definition \ref{cyclically ordered V}, and for each $i$ let $\gamma_i$ be the side of $P_v$ bounded by $x_{i-1}$ and $x_i$ (with $i-1$ taken modulo $n$).  Since each endpoint of $\gamma_i$ lies on a circle with radius $J_{p(v)}$, $\ell(\gamma_i)\leq 2J_{p(v)}$.  Therefore the midpoint $m_i$ of $\gamma_i$ is in $\overline{B}_{J_{p(v)}}(x_{i-1})\cap \overline{B}_{J_{p(v)}}(x_i)$.

Since the convex set $\overline{B}_{J_{p(v)}}(x_i)$ contains $x_i$, $m_i$, and $v$, it contains the right triangle in $\mathbb{H}^2$ that they determine.  Similarly, $B_{J_{p(v)}}(x_{i-1})$ contains the triangle determined by $x_{i-1}$, $m_i$, and $v$.  The union of these triangles is $T_i$ as defined in \IsoscelesDecomp, so $T_i\subset\overline{B}_{J_{p(v)}}(x_{i-1})\cup\overline{B}_{J_{p(v)}}(x_i)$.  Since this holds for each $i$, \IsoscelesDecomp\ implies that $P_v\subset \bigcup_{i=0}^{n-1} \overline{B}_{J_{p(v)}}(x_i)$.\end{proof}

For a surface $F$ with universal cover $p\co\mathbb{H}^2\to F$ and $x\in F$, it is clear that the covering radius of $F$ at $x$ (as defined below Theorem \ref{inj to cov}) is equal to the covering radius of $\widetilde{\cals} \doteq p^{-1}(x)$.

\begin{definition}  For a one-vertex edge-pairing $\iota$, say $(d_0,\hdots,d_{17})\in\calp_{\iota}$ is \textit{exceptional} if $d_i = d_{\beta}$ for all but two $i\in\{0,\hdots,17\}$.  (In this case $d_{i_0} = d_{\iota(i_0)} = b_0(d_{\beta},d_{\beta})$ for some $i_0$.)\end{definition}

\begin{lemma}\label{exceptional}  For a one-vertex edge-pairing $\iota$, if $(d_0,\hdots,d_{17})\in\calp_{\iota}$ is exceptional there exists $i_0$ such that $d_{i_0} = d_{\iota(i_0)} = b_0(d_{\beta},d_{\beta})$.  For $j_0$ with $i_0\in\{3j_0,3j_0+1,3j_0+2\}$ and $j_1$ with $\iota(i_0)\in\{3j_1,3j_1+1,3j_1+2\}$, $T_{j_0}\cup T_{j_1}$ is a $2$-cell of the Delaunay tessellation $P$ of $F_{\iota}(d_0,\hdots,d_{17})$ determined by $\{x\}$, where $x$ is the vertex.  For $j\neq j_0,j_1$, $T_j$ is a $2$-cell of $P$.\end{lemma}

\begin{proof}  \DefectDerivative\ implies that $x\mapsto D_0(x,d_{\beta},d_{\beta})$ achieves a unique maximum at $x = b_0(d_{\beta},d_{\beta})$.  Since $4\cdot D_{0,3}(d_{\beta}) + 2\cdot D_0(b_0(d_{\beta},d_{\beta}),d_{\beta},d_{\beta}) = 4\pi$ by construction (see Example \ref{dull DeBlois}), it follows by definition of $\calp_{\iota}$ that if $(d_0,\hdots,d_{n-1})\in\calp_{\iota}$ is exceptional it must have two entries equal to $b_0(d_{\beta},d_{\beta})$.  That they are exchanged by $\iota$ also follows by definition.

Note that $j_0$ and $j_1$ defined above are distinct, since $\iota$ is one-vertex.  \InCenteredClosure\ implies that $T_{j_0}$ has its center in $\gamma_{i_0}$, and $T_{j_1}$ has its in $\gamma_{\iota(i_0)}$, so their union is a centered quadrilateral with all edges of length $d_{\beta}$.  Corollary \ref{short symmetric poly} thus implies that $T_{j_0}\cup T_{j_1}$ is a $2$-cell of $P$.  It implies the same for $T_j$, $j\neq j_0,j_1$.\end{proof}

\begin{lemma}\label{model to Delaunay}  For each one-vertex edge-pairing $\iota$ and non-exceptional $(d_0,\hdots,d_{17})\in\calp_{\iota}\cap (\widetilde{\calc}_3\cup\widetilde{\calBC}_3)^6\cap [d_{\beta},b_0(d_{\alpha},d_{\alpha})]^{18}$, each $T_j$ from Definition \ref{to M_2} is a $2$-cell of the Delaunay tessellation of $F_{\iota}(d_0,\hdots,d_{17})$ determined by $\{x\}$ where $x$ is the vertex.\end{lemma}

\begin{proof}  By Lemma \ref{model to inj}, $F \doteq F_{\iota}(d_0,\hdots,d_{17})$ has injectivity radius $r = \min\{d_i/2\}_{i=0}^{17}$ at its vertex $x$.  By hypothesis $d_r = 2r = \min\{d_i\}_{i=0}^{17}$ is at least $d_{\beta}$, so Lemma \ref{edge upper bound} implies that $d_i \leq d_1(r)\leq b_0(d_r,d_r)$ for each $i$.  If $d_{i_0} =  b_0(d_r,d_r)$ for some $i_0$, we claim that $(d_0,\hdots,d_{17})$ is exceptional.

If $d_{i_0} = b_0(d_r,d_r)$ for some $i_0$ then in particular $d_1(r) = b_0(d_r,d_r)$, so $r = r_{\beta}$ by Lemma \ref{edge upper bound}.  Fix $j_0$ such that $i_0\in\{3j_0,3j_0+1,3j_0+2\}$.  Since $\iota$ is one-vertex, $j_0\neq j_0'$ such that $\iota(i_0)\in\{3j_0',3j_0'+1,3j_0'+2\}$.  Applying \Monotonicity, we have:
$$ \sum_{j=0}^5 D_0(3j,3j+1,3j+2) \geq 4\cdot D_{0,3}(d_{\beta}) + D_0(b_0(d_{\beta},d_{\beta}),d_{\beta},d_{\beta}) = 4\pi $$
The latter equality is by construction (see Example \ref{sharp DeBlois}).  If there were $i\notin \{i_0,\iota(i_0)\}$ with $d_i>d_{\beta}$ then by \Monotonicity\ the above inequality would be strict, contradicting $(d_0,\hdots,d_{17})\in\calp_{\iota}$.  The claim follows.

For non-exceptional $(d_0,\hdots,d_{17})\in\calp_{\iota}$ the claim implies that $d_i < b_0(d_r,d_r)$ for each $i$. Thus $(d_{3j},d_{3j+1},d_{3j+2})\in\widetilde{\calc}_3$ for each $j$, since for instance $d_{3j} < b_0(d_r,d_r)\leq b_0(d_{3j+1},d_{3j+2})$ (the latter inequality follows from \BCn).  Since $B_0 = b_0(d_r,d_r)$ satisfies $\cosh B_0 = 2\cosh d_r -1$ by \TriCenter, Proposition \ref{shortish poly} implies for each $j$ that $T_j$ is a $2$-cell of the Delaunay tessellation determined by $\{x\}$.\end{proof}

\begin{corollary}\label{model to cov}  For each one-vertex edge-pairing $\iota$ define $J_{\iota}\co \calp_{\iota}\to\mathbb{R}^+$ by $J_{\iota}(d_0,\hdots,d_{17}) = \max\{J(d_{3j},d_{3j+1},d_{3j+2})\}_{j=0}^5$, where $J$ is as in \RadiusFunction.  For $(d_0,\hdots,d_{17})\in\calp_{\iota}\cap (\widetilde{\calc}\cup\widetilde{\calBC}_3)^6\cap[d_{\beta},b_0(d_{\alpha},d_{\alpha}]^{18}$, $F_{\iota}(d_0,\hdots,d_{17})$ has covering radius $J_{\iota}(d_0,\hdots,d_{17})$ at its vertex $x$.\end{corollary}

\begin{proof}  For non-exceptional $(d_0,\hdots,d_{17})\in\calp_{\iota}\cap (\widetilde{\calc}\cup\widetilde{\calBC}_3)^6\cap[d_{\beta},b_0(d_{\alpha},d_{\alpha}]^{18}$ this follows from Lemmas \ref{radius to cov} and \ref{model to Delaunay}, using the bijective correspondence between the Voronoi tessellation's vertex set and the set of Delaunay $2$-cells described in Lemma \ref{vertex polygon}.  Recall in particular that for each $j$, $J(d_{3j},d_{3j+1},d_{3j+2})$ is the radius of $T_j$ by \RadiusFunction.  If $(d_0,\hdots,d_{17})$ is exceptional, Lemmas \ref{radius to cov} and \ref{exceptional} combine in the same way to give the conclusion, noting additionally that $J(b_0(d_{\beta},d_{\beta}),d_{\beta},d_{\beta}) = J(d_{\beta},d_{\beta},d_{\beta},d_{\beta})$ by \HalfSymmQuadDefect.\end{proof}

\begin{lemma}\label{no local max}  For a one-vertex edge-pairing $\iota$ and $(d_0,\hdots,d_{17})\in\calp_{\iota}\cap (\widetilde{\calc}_3\cup\widetilde{\calBC}_3)^6\cap [d_{\beta},b_0(d_{\alpha},d_{\alpha})]^{18}$, let $r = \min\{d_i/2\}$.  If for each $i$ $d_i < d_1(r)$, as defined in Lemma \ref{edge upper bound}, then $(d_0,\hdots,d_{n-1})$ deforms preserving $\min\{d_i\}_{i=0}^{17}$ but increasing $J_{\iota}(d_0,\hdots,d_{17})$.\end{lemma}

\begin{proof}  Fix $j_1$ such that $J_{\iota}(d_0,\hdots,d_{17}) = J(d_{3j_1},d_{3j_1+1},d_{3j_1+2})$ and $i_1\in\{3j_1,3j_1+1,3j_1+2\}$ such that $d_{i_1} = \max\{d_i\}_{i=3j_1}^{3j_1+2}$.  Let $j_1'$ be such that $\iota(i_1)\in \{3j_1',3j_1'+1,3j_1'+2\}$.  Since $\iota$ is one-vertex, $j_1\neq j_1'$.  There exists $i_0\notin\{i_1,\iota(i_1)\}$ with $d_{i_0} > d_r = 2r$, since otherwise by \Monotonicity\ and the definition of $d_1(r)$:
$$ \sum_{j=0}^5 D_0(d_{3j},d_{3j+1},d_{3j+2}) = 4\cdot D_{0,3}(d_r) + 2\cdot D_0(d_{i_1},d_r,d_r) < 4\pi $$
This would contradict $(d_0,\hdots,d_{17})\in\calp_{\iota}$.  Fix $j_0$ with $i_0\in\{3j_0,3j_0+1,3j_0+2\}$ and $j_0'$ with $\iota(i_0)\in\{3j_0',3j_0'+1,3j_0'+2\}$, and note as above that $j_0\neq j_0'$.  We will deform $(d_0,\hdots,d_{n-1})$ changing only $d_{i_0} = d_{\iota(i_0)}$ and $d_{i_1} = d_{\iota(i_1)}$.

Suppose first that $\{j_1,j_1'\}=\{j_0,j_0'\}$.  (In this case $T_{j_1}$ and $T_{j_1'}$ share edges corresponding to $\gamma_{i_1}$ and $\gamma_{i_0}$ in $F_{\iota}(d_0,\hdots,d_{17})$.)  Let $d'$ be the element of $\{d_{3j_1},d_{3j_1+1},d_{3j_1+2}\}$ not equal to $d_{i_1}$ or $d_{i_0}$, and let $d''$ be the corresponding element of $\{d_{3j_1'},d_{3j_1'+1},d_{3j_1'+2}\}$.  For small $t \geq 0$, we will take $d_{i_1}(t) = d_{\iota(i_1)}(t) = d_{i_1}+t$ and choose $d_{i_0}(t)=d_{\iota(i_0)}(t)$ so that $D_0(d_{i_1}+t,d_{i_0}(t),d') + D_0(d_{i_1}+t,d_{i_0}(t),d'')$ is constant.  By \DefectDerivative, $d_{i_0}(t)$ must satisfy:
$$  \frac{d}{dt}\left(d_{i_0}(t)\right) = -\frac{\sqrt{\frac{1}{\cosh^2((d_{i_1}+t)/2)}-\frac{1}{\cosh^2 J_1(t)}} + \sqrt{\frac{1}{\cosh^2((d_{i_1}+t)/2)} - \frac{1}{\cosh^2J_1'(t)}}}{\sqrt{\frac{1}{\cosh^2(d_{i_0}(t)/2)}-\frac{1}{\cosh^2J_1(t)}} + \sqrt{\frac{1}{\cosh^2(d_{i_0}(t)/2)}-\frac{1}{\cosh^2J_1'(t)}}}  $$
Above $J_1(t) = J(d_{i_1}+t,d_{i_0}(t),d')$ and $J_1'(t) = J(d_{i_1}+t,d_{i_0}(t),d'')$.  The existence theorem for ordinary differential equations implies that a unique differentiable function $d_{i_0}(t)$, defined on $[0,\epsilon)$ for some $\epsilon >0$, satisfies the equation above.  Using this equation we find that $d_{i_0}(t)$ decreases in $t$, and also that $|d_{i_0}'(t)| < 1$, since it follows that $d_{i_0}(t) < d_{i_1}+t$ for all $t\in [0,\epsilon)$.

Since $d_{i_0}(t) < d_{i_1}+ t$ for all $t > 0$, \RadiusDeriv\  implies that $\left|\frac{\partial J}{\partial d_{i_0}(t)}\right| < \left|\frac{\partial J}{\partial (d_{i_1}+t)}\right|$, and since $|d_{i_0}'(t)| < 1$ the chain rule implies that $J_1(t)$, and hence also $J_{\iota}$ increases with $t$ in this case.

There are three other possibilities for $\{j_1,j_1',j_0,j_0'\}$: one in which all four of its elements are distinct and two in which it has only three distinct elements (we do not distinguish the case $j_1 = j_0$ from $j_1 = j_0'$, or $j_1' = j_0$ from $j_1' = j_0'$).  In each case we change each element of this set by increasing $d_{i_1}$ and decreasing $d_{i_0}$, leaving all other entries constant while keeping the defect sum unchanged.  

As long as $i_0\notin\{3j_1,3j_1+1,3j_1+2\}$ (equivalently, $j_1\notin\{j_0,j_0'\}$) it is clear by \RadiusDeriv\ that $J_{\iota}$ increases with $t$, so it remains only to consider the case that $j_1 = j_0'$ but $j_1' \neq j_0$.  Taking $d_{i_1}(t) = d_{i_1}+t$, in this case $d_{i_0}(t)$ must satisfy the following differential equation: 
$$  \frac{d}{dt}\left(d_{i_0}(t)\right) = -\frac{\sqrt{\frac{1}{\cosh^2((d_{i_1}+t)/2)}-\frac{1}{\cosh^2 J_1(t)}} + \sqrt{\frac{1}{\cosh^2((d_{i_1}+t)/2)} - \frac{1}{\cosh^2J_1'(t)}}}{\sqrt{\frac{1}{\cosh^2(d_{i_0}(t)/2)}-\frac{1}{\cosh^2J_1(t)}} + \sqrt{\frac{1}{\cosh^2(d_{i_0}(t)/2)}-\frac{1}{\cosh^2J_0(t)}}}  $$
Here $J_1(t) = J(d_{i_0}(t),d_{i_1}+t,d')$, $J_1'(t) = J(d_{i_1}+t,d_{3j_1'+1},d_{3j_1'+2})$ (assuming for simplicity that $\iota(i_1)=3j_1'$), and $J_0(t) = J(d_{i_0}(t),d_{3j_0+1},d_{3j_0+2})$ (assuming that $i_0 = 3j_0$).  We may assume that all entries not in $\{d_{3j_1},d_{3j_1+1},d_{3j_1+2}\}$ equal $d_r$, since otherwise replacing $d_{i_0}$ by an entry not in the set above allows appeal to another case.  Thus in this case $J_0(t) = J(d_{i_0}(t),d_r,d_r)$ and $J_1'(t) = J(d_{i_1}+t,d_r,d_r)$.  

Unlike the first case we considered, it is not immediately obvious here that $|d_{i_0}'(t)| < 1$: the problem is that the defect derivative function $\sqrt{\frac{1}{\cosh^2(d/2)} - \frac{1}{\cosh^2 J}}$ decreases in $d$ but increases in $J$, and \RadiusDeriv\ implies that $J_0(t) < J_1(t)$.  However we have the following:

\begin{claim}  For fixed $d > 0$ and $x$ such that $(x,d,d) \in \calc_3$, the function
$$ x\mapsto \frac{1}{\cosh^2(x/2)} - \frac{1}{\cosh^2 J(x,d,d)} $$
decreases in $x$.  \end{claim}

\begin{proof}[Proof of claim]  Simplifying the formula of \TriDefectRadius\  gives:
$$ \sinh J(x) = \frac{2\sinh^2(d/2)}{\sqrt{4\sinh^2(d/2) - \sinh^2 (x/2)}} $$
Let us take $X = \cosh^2(x/2)$ and $D = \cosh^2(d/2)$.  Inserting the formula above into the function in question, after some more simplification we obtain:
$$ \frac{1}{X} - \frac{1}{\cosh^2 J(x,d,d)} = \frac{[2D - 1 - X]^2}{X[(2D-1)^2 - X]} = \left[ \frac{2D-1}{X} -1\right]\left[1 - \frac{2(D-1)(2D-1)}{(2D-1)^2-X}\right] $$
By \TriCenter, $(x,d,d)\in\calc_3$ if and only if $X < 2D-1$, thus for such $x$ the functions in brackets on the right-hand side are positive-valued.  Since they also clearly decrease with $X$, their product is  decreasing.  Since $X$ increases with $x$, the claim is proved.  \end{proof}

Since $d_{i_0}(t) < d_{i_1}+t$, the claim implies that $|d_{i_0}'(t)|<1$, so as in the first case considered it follows that $J(d_{i_0}(t),d_{i_1}+t,d')$, and hence also $J_{\iota}$ increases with $t$.  In each case above we have thus produced deformations through $\calp_{\iota}$ so that no entries change but $d_{i_0} = d_{\iota(i_0)}$ and $d_{i_1}=d_{\iota(i_1)}$ change.  Furthermore, $d_{i_1}$ increases with $t$ and $d_{i_0}(0) = d_{i_0} > d_r$, so $d_0(t)>d_r = \min\{d_i\}_{i=0}^{17}$ for $t$ near $0$.\end{proof}

\begin{lemma}\label{yes local max}  For each $r\in[r_{\beta},r_{\alpha}]$ and one-vertex edge-pairing $\iota$ there exists $(d_0,\hdots,d_{17})\in\calp_{\iota}\cap (\widetilde{\calc}_3\cup\widetilde{\calBC}_3)^6$ with $r =\min\{d_i/2\}_{i=0}^{17}$ and $d_{i_0} = d_1(r)$ for some $i_0$, where $d_1(r)$ is as in Lemma \ref{edge upper bound}.  At each such point:
$$ \sinh J_{\iota}(d_0,\hdots,d_{17}) = \frac{2\sinh^2 r}{\sqrt{4\sinh^2 r - \sinh^2(d_1(r)/2)}} $$\end{lemma}

\begin{proof}  We may for instance take $d_0 = d_{\iota(0)} = d_1(r)$ and $d_i = d_r\doteq 2r$ for each $i\neq 0,\iota(0)$.  Then $(d_0,\hdots,d_{n-1})\in\calp_{\iota}\cap (\widetilde{\calc}_3\cup\widetilde{\calBC}_3)^6$ by definition.  The computation of $J_{\iota}$ above is a direct application of \TriDefectRadius, noting that Lemma \ref{edge upper bound} and \Monotonicity\ imply that if $r =\min\{d_i/2\}_{i=0}^{17}$ and $d_{i_0} = d_1(r)$ for some $i_0$ then $d_i = d_r$ for all $i\neq i_0,\iota(i_0)$.\end{proof}

\begin{proof}[Proof of Theorem \ref{inj to cov}] For each $r\in [r_{\beta},r_{\alpha}]$ and one-vertex edge-pairing $\iota$, $J_{\iota}$ is a continuous function on $\calp_{\iota}$, being the maximum of functions which are themselves continuous by \SmoothAnglesRadius.  Hence it attains a maximum on the following compact subset: 
$$\bigcup_{j=0}^{17}\calp_{\iota}\cap(\widetilde{\calc}_3\cup\widetilde{\calBC}_3)^6\cap \left([d_{r},b_0(d_{r},d_{r})]^j\times\{d_r\}\times[d_{r},b_0(d_{r},d_{r})]^{17-j}\right)$$  
By Lemma \ref{model to inj}, this consists of those $(d_0,\hdots,d_{17})\in\calp_{\iota}\cap(\widetilde{\calc}_3\cup\widetilde{\calBC}_3)^6$ such that $F_{\iota}(d_0,\hdots,d_{17})$ has injectivity radius $r$ at its vertex $F$.  Lemma \ref{no local max} implies that $J_{\iota}$ attains its maximum at such a point as described in Lemma \ref{yes local max}, and the maximum is as described there.  Since $d_1(r)\leq b_0(d_r,d_r)$ by Lemma \ref{edge upper bound}, and $B_0 = b_0(d_r,d_r)$ satisfies $\sinh (B_0/2) = \sqrt{2}\sinh r$ by \TriDefectRadius, a simplification gives $\sinh J_{\iota}(d_0,\hdots,d_{17}) \leq \sqrt{2}\sinh r$ for each $(d_0,\hdots,d_{17})$ in the set above.  The result now follows directly from Lemma \ref{compact model} and Corollary \ref{model to cov}.
\end{proof}

\bibliographystyle{plain}
\bibliography{volumes}

\end{document}